\documentclass[12pt]{elsarticle}

\usepackage{amsmath,amssymb,amsthm}
\usepackage{mathrsfs}
\usepackage[bookmarks=true,
            bookmarksnumbered=true,
            bookmarksopen=true,
            colorlinks,
            pdfborder=001,
            linkcolor=black]{hyperref}
\usepackage{bm}
\usepackage{array}
\usepackage{float}
\usepackage{color}
\usepackage{lineno}
\usepackage{subcaption}
\usepackage{stmaryrd}
\usepackage{extarrows}
\usepackage{subcaption}
\usepackage{multirow}
\newtheorem{thm}{Theorem}[section]
\newtheorem{lem}[thm]{Lemma}
\newdefinition{rmk}{Remark}[section]
\newdefinition{definition}{Definition}
\newtheorem{example}{Example}[section]

\newproof{pf}{Proof}
\numberwithin{equation}{section}
\numberwithin{figure}{section}
\numberwithin{table}{section}
\renewcommand{\vec}[1]{\mbox{\boldmath \small $#1$}}

\allowdisplaybreaks

\newcommand\bbN{\mathbb{N}}

\newcommand\dd{\mathrm{d}}

\newcommand\abs[1]{\lvert #1 \rvert}
\newcommand\norm[1]{\left\lvert\left\lvert #1 \right\rvert\right\rvert}

\newcommand\pd[2]{\dfrac{\partial {#1}}{\partial {#2}}}

\newcommand\bF{\bm{F}}
\newcommand\bG{\bm{G}}

\newcommand\bw{\bm{w}}
\newcommand\bB{\bm{B}}
\newcommand\vx{v_1}
\newcommand\vy{v_2}
\newcommand\Bx{B_1}
\newcommand\By{B_2}
\newcommand\psix{{{\psi}_1}}

\newcommand\bv{\bm{v}}

\newcommand\bV{\bm{V}}
\newcommand\bU{\bm{U}}

\newcommand\BB{\abs{\bm{B}}^2}
\newcommand\divhB{\nabla\cdot(h\bm{B})}
\newcommand\vB{\bv\cdot\bm{B}}
\newcommand\tbF{\widetilde{\bF}}

\newcommand\xl{{i-\frac12}}
\newcommand\xr{{i+\frac12}}
\newcommand\yl{{j-\frac12}}
\newcommand\yr{{j+\frac12}}
\newcommand\jump[1]{\llbracket #1 \rrbracket}
\newcommand\mean[1]{\{\!\!\{ #1 \}\!\!\}}

\newcommand\jumpangle[1]{\langle\!\langle #1 \rangle\!\rangle}

\begin{document}

\begin{frontmatter}

  \title{High-order accurate entropy stable finite difference schemes for the
    shallow water magnetohydrodynamics}

  \author{Junming Duan}
  \ead{duanjm@pku.edu.cn}
  \author{Huazhong Tang\corref{cor1}}
  \ead{hztang@math.pku.edu.cn}
  \address{Center for Applied Physics and Technology, HEDPS and LMAM, School of Mathematical Sciences, Peking University,
    Beijing 100871, P.R. China}
  \cortext[cor1]{Corresponding author. Fax:~+86-10-62751801.}

  \begin{abstract}
  	This paper develops the high-order accurate entropy stable (ES) finite difference schemes for the
  	shallow water magnetohydrodynamic (SWMHD) equations.
    They are built on the numerical approximation of the modified SWMHD equations
    with  the Janhunen source term. 
    First, the second-order accurate well-balanced semi-discrete entropy conservative (EC) schemes are constructed,
    satisfying the entropy identity for the given convex entropy function
  	and preserving the steady states of the lake at rest (with zero magnetic field).
  	The key is to match both discretizations for the fluxes and the non-flat river bed bottom and Janhunen source terms,
  	and to find the affordable EC fluxes of the second-order EC schemes.
Next, by using the second-order EC schemes as building block,
high-order accurate well-balanced  semi-discrete EC schemes are proposed.
Then, the high-order accurate well-balanced semi-discrete ES schemes 
are derived by adding a suitable dissipation term to the EC scheme with the 
WENO reconstruction of the scaled entropy variables in order to suppress the numerical
  oscillations of the EC schemes.
After that,  	the semi-discrete schemes are integrated in time by using the high-order
  	strong stability preserving explicit  Runge-Kutta schemes to obtain the fully-discrete high-order well-balanced  schemes.
  The ES property of the Lax-Friedrichs flux is also proved and then the positivity-preserving ES schemes are  studied by using the positivity-preserving flux limiter.
Finally,  	extensive numerical tests are conducted to validate the accuracy, the well-balanced,  ES and positivity-preserving properties, and the
  	ability to capture discontinuities of our schemes.
  \end{abstract}


\end{frontmatter}

\section{Introduction}
The  shallow water equations are widely used in atmospheric flows, tides, storm surges, river and coastal
flows, lake flows, tsunamis, etc. They  describe the flow with free surface under the influence of gravity
and the bottom topology, where the vertical dimension is much smaller than any typical horizontal scale.
Numerical simulation is an effective tool to solve them and a great variety of numerical methods
are available in the literature, e.g. \cite{Bermudez1994,Kuang2017SW,Noelle2007,Tang2004SWa,Tang2004SWb,Wu2016SW,Xing2017book,Xing2005} and the references therein.

Here we are concerned with numerical methods for the shallow water magnetohydrodynamic (SWMHD) equations,
which take into account the effect of the magnetic field,
originally proposed in \cite{Gilman2000} for studying the global dynamics
of the solar tachocline.
The two-dimensional (2D) SWMHD equations with non-flat bottom topography \cite{Gilman2000,Rossmanith2003} read the following  quasi-linear hyperbolic balance laws
\begin{equation}\label{eq:SWMHD}
  \pd{\bU}{t}+\sum\limits_{\ell=1}^2\pd{\bF_\ell(\bU)}{x_\ell}=-\bG(\bU),
\end{equation}
with the divergence-free condition
\begin{equation}
  \divhB=0, \label{eq:divhB}
\end{equation}
where $\bU$ is the conservative variables vector, and $\bF_1$ and $\bF_2$ are respectively
the flux vectors along the $x_1$- and $x_2$ directions and defined by
\begin{align}\label{eq:SWMHD_comp}\begin{aligned}
  \bU&=\left(h, ~h\bv^\mathrm{T}, ~h\bB^\mathrm{T}\right)^\mathrm{T},\\
  \bF_\ell&=\left(hv_\ell, ~hv_\ell\bv^\mathrm{T}-hB_\ell\bB^\mathrm{T}+\frac12gh^2\bm{e}_\ell^\mathrm{T}, ~h(v_\ell\bB-B_\ell\bv)^\mathrm{T}\right)^\mathrm{T},\ \ell=1,2,\\
  \bG&=\left(0, gh\nabla{b}, ~\bm{0}_2^\mathrm{T}\right)^\mathrm{T},\
  \bm{0}_2^\mathrm{T}=(0,0),
\end{aligned}\end{align}
with the height of conducting fluid $h$, the fluid velocity vector $\bv=(\vx,\vy)^\mathrm{T}$,
the magnetic field vector $\bm{B}=(\Bx,\By)^\mathrm{T}$, the gravitational acceleration constant $g$,
the bottom topography $b=b(x,y)$,   and   $\bm{e}_\ell$ denotes the $\ell$th
column of the $2\times 2$ unit matrix.

Existing numerical studies for the SWMHD equations include the evolution Galerkin scheme \cite{Kroger2005},
space-time conservation element solution element (CESE) method \cite{Qamar2006}, central-upwind schemes \cite{Zia2014},
Roe-type schemes \cite{Kemm2016}, second-order entropy stable (ES) finite volume scheme (satisfying the
{semi-}discrete entropy inequality) \cite{Winters2016}, high-order CESE scheme up to 4th-order \cite{Ahmed2019}, etc.
Some have dealt with the non-flat topography \cite{Winters2016,Zia2014}, and are well-balanced in the sense that
the schemes can preserve the lake at rest.
For the numerical solutions of the SWMHD equations, we need to deal carefully with the divergence-free constraint \eqref{eq:divhB}.
In the ideal magnetohydrodynamic (MHD) case, many works have focused on this issue, for example, the projection method \cite{Brackbill1980},
the constrained transport  method and its variants \cite{Balsara2009MHDa,Evans1988,Londrillo2004,Rossmanith2006},
the eight-wave formulation of the MHD equations \cite{Powell1994}, the hyperbolic divergence cleaning method \cite{Dedner2002},
the locally divergence-free DG method \cite{LiFY2005}, the ``exactly'' divergence-free central DG method \cite{LiFY2011}, and so on.

For the quasi-linear hyperbolic conservation laws or balance laws,
it may be the case that no classical solution exists so that
the weak solution  should be  defined  in the sense of distributions.
Unfortunately, such weak solutions might be non-unique.
The entropy condition plays an essential role in
choosing the physically relevant solution from the collection of all possible week solutions.

\begin{definition}[Entropy function]
  A strictly convex function $\eta(\bU)$ is called an {\em entropy function} for the
  system \eqref{eq:SWMHD}-\eqref{eq:SWMHD_comp} if there are associated {\em entropy
  fluxes} $q_1(\bU)$ and $q_2(\bU)$ such that
  \begin{equation}\label{eq:entropy}
    q_\ell'(\bU)=\bV^\mathrm{T}\bF_\ell'(\bU),\quad \ell=1,2,
  \end{equation}
  where $\bV=\eta'(\bU)^\mathrm{T}$ is called the {\em entropy variables}, and
  $(\eta,q_\ell)$ is an {\em entropy pair}.
\end{definition}
For the smooth solutions of \eqref{eq:SWMHD}-\eqref{eq:SWMHD_comp}
with  the entropy pair $(\eta,q_\ell)$,
multiplying \eqref{eq:SWMHD} by $\bV^\mathrm{T}$ gives the entropy identity
\begin{equation}\label{eq:SmoothEntId}
  \pd{\eta(\bU)}{t}+\sum\limits_{\ell=1}^2\pd{q_\ell(\bU)}{x_\ell} = -\bV^\mathrm{T}\bG(\bU).
\end{equation}
However, if the solutions contain discontinuities, then the above identity
does not hold.
\begin{definition}[Entropy solution]
  A weak solution $\bU$ of \eqref{eq:SWMHD} is called an {\em entropy solution} or a {\em physically relevant solution} if for
  all entropy pairs $(\eta,q_\ell)$, the entropy inequality or condition
  \begin{equation}\label{eq:entropyineq}
    \pd{\eta(\bU)}{t}+\sum\limits_{\ell=1}^2\pd{q_\ell(\bU)}{x_\ell} \leqslant -\bV^\mathrm{T}\bG(\bU),
  \end{equation}
  holds in the sense of distributions.
\end{definition}

The entropy conditions are of great importance in the well-posedness of hyperbolic conservation laws
or balance laws, e.g. \eqref{eq:SWMHD}, and may improve the robustness of the numerical schemes,
thus it is meaningful to seek their numerical schemes
 satisfying  the discrete entropy inequality. 
For the scalar conservation laws, the conservative monotone schemes
are nonlinearly stable and satisfy the discrete entropy conditions
so that their solutions can converge to the entropy solution \cite{Crandall1980,Harten1983}; the E-schemes satisfy the semi-discrete entropy condition for any convex entropy \cite{Osher1984,Osher1988}. However, those schemes are only first-order accurate.
Generally, it is very hard to prove that the high-order accurate schemes of
the scalar conservation laws and the schemes for the system of hyperbolic conservation laws
satisfy the entropy inequality for any convex entropy function.
Two relative works are presented in \cite{Bouchut1996} and \cite{Hughes1986}.
The former is second-order accurate and not in the standard finite volume
form, while the latter approximates the entropy variables and needs solving
nonlinear equations at each time step.
In view of this, the researchers usually try to study the high-order accurate  entropy conservative (EC) (resp.
entropy stable (ES)) schemes, which satisfy the discrete entropy identity (resp. inequality) for a given entropy pair.
The second-order EC schemes  were built in \cite{Tadmor1987,Tadmor2003}, and their higher-order
extension was studied in \cite{Lefloch2002}.
It is known that EC schemes may become oscillatory near the
shock waves, thus additional dissipation terms have to be added into the EC schemes to obtain the ES schemes.
Combining the EC flux of the EC schemes
with the ``sign'' property of the ENO reconstructions, the`
arbitrary high-order ES schemes were constructed by using high-order
dissipation terms \cite{Fjordholm2012}.
Some ES schemes based on the discontinuous Galerkin (DG) framework were also
studied, e.g. the ES space-time DG
formulation \cite{Barth1999,Hiltebrand2014}  and the ES nodal DG schemes using suitable quadrature rules
\cite{Chen2017}.
The ES schemes based on summation-by-parts (SBP) operators were
developed for the Navier-Stokes equations \cite{Fisher2013}.
As a base of those works, the construction of the affordable two-point EC flux is  one of the key parts, and has been
extended to the shallow water equations \cite{Fjordholm2011,Gassner2016}, the MHD equations \cite{Chandrashekar2016,Winters-jcp2016},
the relativistic (magneto-)hydrodynamic equations \cite{Duan2020RHD,Duan2020RMHD,Wu2020RMHDES}, and so on.
Those ES MHD schemes are built on the modified MHD equations with the non-conservative source terms,
e.g. the Powell source terms \cite{Chandrashekar2016,Derigs2018,Powell1994} or the Janhunen source terms
\cite{Janhunen2000,Winters-jcp2016}. The Powell source term can be used to obtain the symmetrizable MHD equations,
but the solution to the Riemann problem of Powell's MHD equations
 with initial positive gas pressures
may contain a nonphysical negative gas pressure \cite{Janhunen2000},
while the Janhunen source term can preserve the conservation of the momentum and energy and  be used
to restore the positivity of gas pressure.
Due to  the adding source terms, the sufficient condition proposed in \cite{Tadmor1987} for a finite difference scheme  satisfying an entropy identity
should be modified, see \cite{Chandrashekar2016}.
One should also take care of the discretization of the source
terms, to match it to the discretization of the flux gradients and ensure that the final schemes satisfy the semi-discrete entropy inequality.

The paper aims at constructing the high-order accurate  ES
finite difference schemes for the  SWMHD equations \eqref{eq:SWMHD} for a given entropy pair.
With suitable discretization of the non-flat bottom topography and  Janhunen source terms
in the modified  SWMHD equations, a two-point EC flux is derived for constructing the semi-discrete second-order
accurate well-balanced EC schemes satisfying the entropy identity.
Our discretization of the source terms is essential to achieve both high-order accuracy and well-balance,
and does not meet the computational issue in \cite{Winters2016} when
the magnetic field is zero. 
The high-order well-balanced EC schemes are constructed by using the above two-point EC schemes as building blocks.
In order to avoid the numerical oscillation produced
by the EC schemes around the discontinuities,  some suitable dissipation terms utilizing the weighted essentially
non-oscillatory (WENO) reconstruction in the scaled entropy variables are added to the EC fluxes to get the high-order accurate well-balanced ES schemes satisfying the semi-discrete entropy inequality.
The above semi-discrete EC and ES schemes are integrated in time by using the high-order accurate explicit strong-stability preserving Runge-Kutta schemes to obtain the fully-discrete high-order accurate well-balanced 
schemes.
 The ES property of the Lax-Friedrichs flux is also proved and then the positivity{-}preserving ES schemes are developed by using the positivity-preserving flux limiter.

The rest of the paper is organized as follows.
Section \ref{section:Symm} presents the modified SWMHD equations, the entropy pair and
necessary notations.
Section \ref{section:OneD} constructs the affordable two-point EC flux and the semi-discrete EC and ES schemes for the 1D SWMHD equations,
and proves the   well-balanced properties of the EC and ES schemes.
The positivity-preserving ES scheme is also studied.
Section \ref{section:MultiD} gives the 2D well-balanced EC and ES schemes.
Extensive numerical tests are conducted in Section \ref{section:Num} to validate the effectiveness and performance of our schemes.
Section \ref{section:Conclusion} gives some conclusions.

\section{Modified SWMHD equations}\label{section:Symm}
This section gives an entropy analysis of the   SWMHD equations with the
non-flat bottom topography.

If the solutions are smooth, then the SWMHD equations \eqref{eq:SWMHD}-\eqref{eq:SWMHD_comp} can be equivalently cast into
the primitive variable form
\begin{equation}\label{eq:SWMHD_prim}
\begin{aligned}
&\pd{h}{t}+\nabla\cdot(h\bv)=0,\\
&\pd{\bv}{t}+(\bv\cdot\nabla)\bv-(\bB\cdot\nabla)\bB+g\nabla h^\mathrm{T}=-gh\nabla{b}^\mathrm{T}+\divhB\bB/h,\\
&\pd{\bB}{t}+(\bv\cdot\nabla)\bB-(\bB\cdot\nabla)\bv=\divhB\bv/h,\\
&\divhB=0.\\
\end{aligned}
\end{equation}
Defining the mathematical entropy as the total energy \cite{Fjordholm2011,Winters2016}
\begin{equation}\label{eq:entropy001}
\eta(\bU,{b}):=\dfrac12h(\abs{\bv}^2+\abs{\bB}^2)+\dfrac12gh^2+ghb,
\end{equation}
and using \eqref{eq:SWMHD_prim} gives
\begin{equation*}
  \partial_t{\eta}+\sum\limits_{\ell=1}^2\partial_{x_\ell}
  {\left[\left(\dfrac12(\abs{\bv}^2+\abs{\bB}^2)+gh+gb\right)hv_\ell
  -hB_\ell(\vB)\right]}-(\vB)\divhB=0,
\end{equation*}
which means that
under the constraint $\divhB=0$, the following quantities
\begin{equation}\label{eq:entropypair}
  \eta(\bU,{b}),\quad q_\ell(\bU,{b}):=\left(\dfrac12(\abs{\bv}^2+\abs{\bB}^2)+gh
  +gb\right)hv_\ell-hB_\ell(\vB),
\end{equation}
satisfy an additional conservation law.
However, unfortunately, the pair $(\eta, q_\ell)$ defined in \eqref{eq:entropy001}-\eqref{eq:entropypair}
does not satisfy \eqref{eq:entropy}, since
\begin{equation*}
\partial q_\ell(\bU,b)/\partial \bU =\bV^\mathrm{T}\bF_\ell'(\bU)+(\vB)(hB_\ell)'(\bU),
\end{equation*}
where the vector 
$\bV=(\partial{\eta}/\partial\bU)^\mathrm{T}$ is explicitly given by
\begin{equation*}
\bV=\left(g(h+b)-\frac12\left(\abs{\bv}^2+\BB\right), \bv^\mathrm{T}, \bB^\mathrm{T}  \right)^\mathrm{T}.
\end{equation*}
 It is not difficult to verify that the matrix ${\partial\bU}/{\partial\bV}$ is symmetric and positive definite.

 Similar to the Janhunen source term for the ideal MHD equations
\cite{Janhunen2000},
one needs to add some non-conservative source terms
to get a modified   SWMHD system  for \eqref{eq:SWMHD}-\eqref{eq:SWMHD_comp}
as follows
\begin{equation}\label{eq:symm}
\pd{\bU}{t}+\sum\limits_{\ell=1}^2\pd{\bF_\ell}{x_\ell}+
{\Psi }^\mathrm{T}\divhB=-\bG(\bU),
\end{equation}
where
\begin{equation*}
\Psi=(0,0,0,\bv^\mathrm{T}),
\end{equation*}
and satisfies $\Psi\bV=\Phi(\bV):=\vB$.
Note that the modified SWMHD equations without bottom topography have been discussed in the literature, see e.g. \cite{Winters2016}.

Taking the dot product of  $\bV$ with \eqref{eq:symm} yields
\begin{align*}
&\bV^\mathrm{T}\pd{\bU}{t} +\sum_{\ell=1}^2\left(\bV^\mathrm{T}\pd{\bF_\ell}{\bU}+\Phi(\bV)
\pd{(hB_\ell)}{\bU}\right)\pd{\bU}{x_\ell}+\bV^\mathrm{T}\bG(\bU)\\
=&\pd{\eta}{t}+\sum_{\ell=1}^2\left(\pd{q_\ell}{\bU} +\left(\Phi(\bV) -(\vB)\right)\pd{(hB_\ell)}{\bU}\right)\pd{\bU}{x_\ell}+\bV^\mathrm{T}\bG(\bU)=0,
\end{align*}
i.e.
\begin{equation}\label{eq:entropyID}
\pd{\eta}{t}+\sum\limits_{\ell=1}^2\pd{q_\ell}{x_\ell}=0,
\end{equation}
where we have used the identity
\begin{equation*}
\sum_{\ell=1}^2\pd{q_\ell{(\bU,b)}}{\bU}\pd{\bU}{x_\ell}
+\bV^\mathrm{T}\bG(\bU)=\sum_{\ell=1}^2\pd{q_\ell}{x_\ell}.
\end{equation*}
Notice that the identity \eqref{eq:entropyID} is obtained without using
the divergence-free condition, and will be useful in constructing an entropy stable (ES)
scheme  because the numerical divergence of $\divhB$ may not be zero.
Moreover, we  define the ``entropy potential'' $\psi_\ell$ from the given $(\eta(\bU,{b}), q_\ell(\bU,{b}))$ by
\begin{equation*}
\psi_\ell:= \bV^\mathrm{T} \bF_\ell(\bU)+\Phi(\bV) (hB_\ell)-q_\ell(\bU)=\frac12gh^2v_\ell,
\end{equation*}
which makes the following identity true
\begin{align*}
\int_{\Omega} \left( \pd{\eta}{t}+ \pd{q_\ell}{x_\ell} \right)\dd\bm{x}
=&\int_{\Omega} \bV^\mathrm{T}\left( \pd{\bU}{t}+ \pd{\bF_\ell(\bU)}{x_\ell}+{ \Psi}^\mathrm{T}\divhB+\bG(\bU)\right)\dd\bm{x}\\
=&\int_{\Omega} \left( \pd{\eta(\bU)}{t}+ \pd{(\bV^\mathrm{T} \bF_\ell (\bU))}{x_\ell } - \pd{\bV^\mathrm{T}}{x_\ell } \bF(\bU)+\nabla\cdot(\Phi h\bB)-\nabla\Phi\cdot(h\bB) \right)\dd\bm{x}\\
=&\int_{\Omega} \left( \pd{\eta(\bU)}{t}+ \pd{(\bV^\mathrm{T} \bF_\ell (\bU))}{x_\ell } +\nabla\cdot(\Phi h\bB) - \pd{\psi_\ell (\bU)}{x_\ell }\right)\dd\bm{x}.
\end{align*}
The ``entropy potential'' plays an important role in obtaining the sufficient condition for the two-point entropy conservative (EC) fluxes.

\begin{rmk}
	Notice that the entropy identity \eqref{eq:entropyID} is slightly different from \eqref{eq:SmoothEntId},
	since the source terms in the SWMHD equations have special structure, and then $\bV^\mathrm{T}\bG(\bU)$ in \eqref{eq:SmoothEntId}
	can be absorbed into the left-hand side by using
	\begin{equation*}
	  \pd{(ghb)}{t}+\sum_{\ell=1}^2\pd{(ghbv_\ell)}{x_\ell}= \bV^\mathrm{T}\bG(\bU).
	\end{equation*}
	In other words, the difference between the entropy pair $(\eta,q_\ell)$ in \eqref{eq:entropyID} and that in \eqref{eq:SmoothEntId}
	is $(ghb,ghbv_\ell)$.
\end{rmk}


\section{One-dimensional schemes}\label{section:OneD}
This section   constructs the high-order accurate well-balanced
EC and ES schemes for the $x$-split system of \eqref{eq:symm}, i.e.,
\begin{equation}\label{eq:symm1D}
  \pd{\bU}{t}+\pd{\bF_1(\bU)}{x}=-\Psi^\mathrm{T}\pd{(h\Bx)}{x}
  -\bm{G}_1^\mathrm{T}\pd{b}{x},
\end{equation}
where $\bm{G}_1=(0,~gh,~0,~0,~0)^\mathrm{T}$.

\subsection{Second-order EC schemes}
Let us consider a uniform mesh in $x$: $x_1<x_2<\cdots<x_{N_x}$, with the spatial step size
$\Delta x=x_i-x_{i-1},~i=2,\cdots,N_x$ and the semi-discrete conservative finite
difference scheme for \eqref{eq:symm1D} as follows
\begin{equation}\label{eq:1Dsemi}
\dfrac{\dd}{\dd t}\bU_i=-\dfrac{1}{\Delta x}\left(\widehat{\bF}_{\xr}-\widehat{\bF}_{\xl}\right)
-
\Psi_i^\mathrm{T}\dfrac{\mean{h\Bx}_\xr-\mean{h\Bx}_\xl}{\Delta x}
-(\bm{G}_1)_i^\mathrm{T}\dfrac{\mean{b}_\xr-\mean{b}_\xl}{\Delta x},
\end{equation}
where $\mean{a}_\xr$ denotes the mean value of $a$ at $x_\xr$, i.e.,
$\mean{a}_\xr=(a_i+a_{i+1})/2$, $\bU_i(t)$ and $\bV_i(t)$ approximate the point values of $\bU(x_i,t),\bV(x_i,t)$, respectively,
$\widehat{\bF}_{\xr}(t)$ is the numerical flux approximating $\bF_1(x,t)$ at $x_{\xr}=x_i+{\Delta x}/{2}$,
and the second-order central difference is used to approximate
$\partial{(h\Bx)}/\partial{x}$ and $\partial{b}/\partial{x}$ in the source terms.

\begin{definition}[Entropy conservative scheme]
  The scheme \eqref{eq:1Dsemi} is {\em entropy conservative} (EC)
     and corresponding numerical flux $\widehat{\bF}_{\xr}(t)$ is called the {\em EC flux}, if the solution of \eqref{eq:1Dsemi}  satisfies a semi-discrete entropy identity
  \begin{equation}\label{eq:discEnEq}
    \dfrac{\dd}{\dd t}\eta(\bU_i(t))+\dfrac{1}{\Delta
    x}\left(\widetilde{q}_{\xr}(t)-\widetilde{q}_{\xl}(t)\right)=0,
  \end{equation}
  for some numerical entropy flux $\widetilde{q}_{\xr}$ consistent with the given physical entropy flux $q_1$.
\end{definition}

The following lemma gives a sufficient condition for the semi-discrete scheme \eqref{eq:1Dsemi} to be  EC,
with the discretization of  the source terms.

\begin{lem}\label{lem:ECFlux}
	If a symmetric consistent two-point flux $\tbF_{\xr}:=\tbF_1(\bU_i,\bU_{i+1})$ satisfying
	\begin{equation}\label{eq:conserFlux}
	\jump{\bV}^\mathrm{T}\cdot \tbF_1=\jump{\psix}-\jump{\Phi}\mean{h\Bx}+g\jump{hb\vx}
-g\jump{h\vx}\mean{b},
	\end{equation}
	is used in  \eqref{eq:1Dsemi},
		where $\jump{a}$ and $\mean{a}$ denote the jump and mean of $a$, respectively,
	then the semi-discrete scheme \eqref{eq:1Dsemi} is second-order accurate and EC,
	with the  numerical entropy flux
	\begin{equation}\label{eq:numEntropyFlux}
	   \widetilde{q}_\xr=\mean{\bV}_\xr^\mathrm{T}\tbF_{\xr}+\mean{\Phi}_\xr\mean{h\Bx}_\xr-\mean{\psi_1}_\xr
	   +g\mean{h\vx}_\xr\mean{b}_\xr-g\mean{hb\vx}_\xr.
	\end{equation}
\end{lem}

\begin{proof}
	Left multiplying \eqref{eq:1Dsemi} by $\bV_i^\mathrm{T}$ and using $\Phi(\bV)=\Psi\bV$ gives
	\begin{align*}
	\dfrac{\dd\eta_i}{\dd t}=-\dfrac{1}{\Delta x}\left[\bV_i^\mathrm{T}\left(\tbF_{\xr}-\tbF_{\xl}\right)
	+\Phi(\bV_i)\left(\mean{h\Bx}_\xr-\mean{h\Bx}_\xl\right)
	+gh_i(\vx)_i\left(\mean{b}_\xr-\mean{b}_\xl\right)\right].
	\end{align*}
The right-hand side term can be further rearranged  as follows
	\begin{align*}
	&\bV_i^\mathrm{T}\left(\tbF_{\xr}-\tbF_{\xl}\right)
	+\Phi(\bV_i)\left(\mean{h\Bx}_\xr-\mean{h\Bx}_\xl\right)
	+gh_i(\vx)_i\left(\mean{b}_\xr-\mean{b}_\xl\right)\\
	=&\left(\mean{\bV}_\xr-\frac12\jump{\bV}_\xr\right)^\mathrm{T}\tbF_{\xr}
	-\left(\mean{\bV}_\xl+\frac12\jump{\bV}_\xl\right)^\mathrm{T}\tbF_{\xl} \\
	&+\left(\mean{\Phi}_\xr-\frac12\jump{\Phi}_\xr\right)\mean{h\Bx}_\xr
	-\left(\mean{\Phi}_\xl+\frac12\jump{\Phi}_\xl\right)\mean{h\Bx}_\xl \\
	&+g\left(\mean{h\vx}_\xr-\frac12\jump{h\vx}_\xr\right)\mean{b}_\xr
	-g\left(\mean{h\vx}_\xl+\frac12\jump{h\vx}_\xl\right)\mean{b}_\xl \\
	=&\mean{\bV}_\xr^\mathrm{T}\tbF_{\xr}+\mean{\Phi}_\xr\mean{h\Bx}_\xr+g\mean{h\vx}_\xr\mean{b}_\xr \\
	&-\mean{\bV}_\xl^\mathrm{T}\tbF_{\xl}-\mean{\Phi}_\xl\mean{h\Bx}_\xl-g\mean{h\vx}_\xl\mean{b}_\xl \\
	&-\frac12\jump{\psi_1}_\xr-\frac12\jump{\psi_1}_\xl-\frac12g\jump{hb\vx}_\xr-\frac12g\jump{hb\vx}_\xl \\
	=&\left(\mean{\bV}_\xr^\mathrm{T}\tbF_{\xr}+\mean{\Phi}_\xr\mean{h\Bx}_\xr+g\mean{h\vx}_\xr\mean{b}_\xr\right)
	-\mean{\psi_1}_\xr -g\mean{hb\vx}_\xr \\
	&-\left(\mean{\bV}_\xl^\mathrm{T}\tbF_{\xl}+\mean{\Phi}_\xl\mean{h\Bx}_\xl+g\mean{h\vx}_\xl\mean{b}_\xl\right)
	+\mean{\psi_1}_\xl +g\mean{hb\vx}_\xl
\\
=& \widetilde{q}_\xr-\widetilde{q}_\xl,
	\end{align*}
	where $a_i=\mean{a}_\xr-\frac12\jump{a}_\xr$ and $a_{i}=\mean{a}_\xl+\frac12\jump{a}_\xl$ have been used in the first equality,
	the condition \eqref{eq:conserFlux} has been used in the second equality,	and $\frac12\jump{a}_\xr+\frac12\jump{a}_\xl=\mean{a}_\xr-\mean{a}_\xl$ has been used in the  third   equality.
	Thus the scheme \eqref{eq:1Dsemi} with $ \widehat{\bF}_{\xr}= \tbF_1(\bU_i,\bU_{i+1})$ is EC in the sense of
	\begin{align*}
	\dfrac{\dd\eta_i}{\dd t}+\dfrac{1}{\Delta x}\left(\widetilde{q}_\xr-\widetilde{q}_\xl\right)=0.
	\end{align*}
	The discretization of the source terms is second-order accurate since the second-order central difference is used,
	and the results in \cite{Tadmor1987} show that the discretization of the flux gradient is second-order accurate,
	therefore the the scheme \eqref{eq:1Dsemi} with $ \widehat{\bF}_{\xr}= \tbF_1(\bU_i,\bU_{i+1})$ is second-order accurate.
\end{proof}

\begin{rmk}
	Since the central difference is used for approximating the non-flat river bed bottom and  Janhunen source terms,
	 the sufficient condition \eqref{eq:conserFlux} is different
	from that in \cite{Fjordholm2011,Winters2016}. Moreover, such discretization of the source terms
	is essential to achieve high-order accuracy and well-balance, see the subsection \ref{subsec:HOEC}.
\end{rmk}

Below we present such   EC flux satisfying \eqref{eq:conserFlux}.

\begin{thm}\label{thm:ECFlux}
  For the $x$-split
  SWMHD equations \eqref{eq:symm1D}, the following flux $\tbF_1(\bU_i,\bU_{i+1})$
  \begin{align}\label{eq:ecFlux1D}
  \begin{aligned}
  \tbF_1=\begin{pmatrix}
  \mean{h}\mean{\vx}\\
  \mean{h}\mean{\vx}^2+\dfrac{g}{2}\mean{h^2}-\mean{h\Bx}\mean{\Bx}+g\left(\mean{hb}-\mean{h}\mean{b}\right)\\
  \mean{h}\mean{\vx}\mean{\vy}-\mean{h\Bx}\mean{\By}\\
  \mean{h}\mean{\vx}\mean{\Bx}-\mean{h\Bx}\mean{\vx}\\
  \mean{h}\mean{\vx}\mean{\By}-\mean{h\Bx}\mean{\vy}
  \end{pmatrix}
  \end{aligned}
  \end{align}
  is an EC flux, consistent with the physical flux $\bF_1(\bU)$ defined in \eqref{eq:SWMHD_comp}.
\end{thm}

\begin{proof}
	The key is to use the identity
	\begin{equation}\label{eq:jumpid}
	\jump{ab}=\mean{a}\jump{b}+\mean{b}\jump{a},
	\end{equation}
	and rewrite the jumps of the entropy variables $\bV$, the
	potential $\psix$, $\Phi$, $h\vx$ and $hb\vx$ as some linear combinations of the jumps of a specially chosen
	parameter vector.
	For simplicity in derivation,  the parameter vector is chosen as $\left(h,b,\vx,\vy,\Bx,\By\right)$, then
	\begin{align*}
	\jump{\bV}^\mathrm{T}=&\Big(g\jump{h}+g\jump{b}-\mean{\vx}\jump{\vx}-\mean{\vy}\jump{\vy}-\mean{\Bx}\jump{\Bx}-\mean{\By}\jump{\By},\\
	&\jump{\vx},~\jump{\vy},~\jump{\Bx},~\jump{\By}\Big),\\
	\jump{\psi_1}=&g\mean{h}\mean{\vx}\jump{h}+\frac12g\mean{h^2}\jump{\vx},\\
	\jump{\Phi}=&\mean{\Bx}\jump{\vx}+\mean{\By}\jump{\vy}+\mean{\vx}\jump{\Bx}+\mean{\vy}\jump{\By},\\
	g\jump{hb\vx}-g\jump{h\vx}\mean{b}=&g\mean{h}\mean{\vx}\jump{b}+g\mean{hb}\jump{\vx}.
	\end{align*}
	Substituting them into \eqref{eq:conserFlux},
	and equating the coefficients of the same jump terms gives the numerical flux in \eqref{eq:ecFlux1D}.
	
	If letting $\bU_i=\bU_{i+1}$, it is easy to see the consistency of the numerical flux in \eqref{eq:ecFlux1D}.
\end{proof}

\begin{rmk} 
  For the $y$-split system of \eqref{eq:symm},
  the rotational invariance may be used to get the  EC fluxes consistent
  to $\bF_2(\bU)$ defined in \eqref{eq:SWMHD_comp}.
    The EC flux \eqref{eq:ecFlux1D} is the same as the one in \cite{Winters2016}.
\end{rmk}

\begin{rmk}
The present discretization of the source terms is different from that in \cite{Winters2016},
  and does not meet the computational issue in \cite{Winters2016} if $\mean{\Bx}=0$ or $\mean{\By}=0$. 
\end{rmk}

\begin{rmk}
  If the magnetic field $\vec B\equiv 0$, then the SWMHD equations reduce  to the shallow water equations (SWEs)  and the above SWMHD scheme   \eqref{eq:1Dsemi} with $ \widehat{\bF}_{\xr}= \tbF_1(\bU_i,\bU_{i+1})$  defined in \eqref{eq:ecFlux1D} reduces to the well-balanced EC scheme for the SWEs
  with the EC flux
 \begin{align*}
 \begin{aligned}
 \tbF_1=\begin{pmatrix}
 \mean{h}\mean{\vx}\\
 \mean{h}\mean{\vx}^2+\dfrac{g}{2}\mean{h^2}+g\left(\mean{hb}-\mean{h}\mean{b}\right)\\
 \mean{h}\mean{\vx}\mean{\vy}
 \end{pmatrix},
 \end{aligned}
 \end{align*}
  which is the same as the EC flux in \cite{Fjordholm2011}, except
  for the second component 
  due to the different approximation of the source terms.
\end{rmk}

Lemma \ref{lem:ECFlux} and Theorem \ref{thm:ECFlux} tell us that
the SWMHD scheme \eqref{eq:1Dsemi} for \eqref{eq:symm1D} with $ \widehat{\bF}_{\xr}= \tbF_1(\bU_i,\bU_{i+1})$  defined in \eqref{eq:ecFlux1D}
is second-order accurate and EC.
Moreover, we can show that it is well-balanced.

\begin{thm}\label{Theorem3.3}  
The scheme \eqref{eq:1Dsemi} with the EC flux \eqref{eq:ecFlux1D} is well-balanced,
in the sense that when the magnetic field is zero, it preserves the lake at rest, that is to say, for the given initial data
\begin{equation*}
(\vx)_i=(\vy)_i\equiv 0, ~h_i+b_i\equiv C,~\forall i,
\end{equation*}
 the solutions of \eqref{eq:1Dsemi} satisfy
\begin{equation*}
\dfrac{\dd}{\dd t}h_i\equiv 0,~\dfrac{\dd}{\dd t}(h\vx)_i\equiv 0,~\dfrac{\dd}{\dd t}(h\vy)_i\equiv 0.
\end{equation*}
\end{thm}

\begin{proof} Under the hypotheses,
one can verify that   the scheme \eqref{eq:1Dsemi}
 satisfies
	$\dfrac{\dd}{\dd t}h_i\equiv 0,~\dfrac{\dd}{\dd t}(h\vy)_i\equiv 0$ and
	\begin{align*}
	\dfrac{\dd}{\dd t}(h\vx)_i
	=&-\dfrac{g}{\Delta x}\Big[\left(\dfrac{1}{2}\mean{h^2}_\xr-\dfrac{1}{2}\mean{h^2}_\xl\right)+\left(\mean{hb}_\xr-\mean{hb}_\xl\right) \\
	&-\left(\mean{h}_\xr\mean{b}_\xr-\mean{h}_\xl\mean{b}_\xl\right)+h_i\left(\mean{b}_\xr-\mean{b}_\xl\right)\Big] \\
	=&-\dfrac{g}{\Delta x}\Big[\dfrac{1}{2}\left(\mean{h}_\xr\jump{h}_\xr+\mean{h}_\xl\jump{h}_\xl\right)
	+\left(\mean{h}_\xr b_{i+1}-\mean{h}_\xl b_{i-1}\right)\\
	&-\left(\mean{h}_\xr\mean{b}_\xr-\mean{h}_\xl\mean{b}_\xl\right) \Big]\\
	=&-\dfrac{g}{\Delta x}\Big[\dfrac{1}{2}\left(\mean{h}_\xr\jump{h}_\xr+\mean{h}_\xl\jump{h}_\xl\right)
	+\dfrac{1}{2}\left(\mean{h}_\xr\jump{b}_\xr+\mean{h}_\xl\jump{b}_\xl\right) \Big]\\
	=&-\dfrac{g}{2\Delta x}\Big[\mean{h}_\xr\jump{h+b}_\xr+\mean{h}_\xl\jump{h+b}_\xl \Big] \\
	\equiv&0.
	\end{align*}
Therefore the scheme \eqref{eq:1Dsemi} preserves the lake at rest.
\end{proof}

\begin{rmk}\label{remark3.5} The second-order EC scheme \eqref{eq:1Dsemi} satisfies
	the 1D moving equilibrium state \cite{Fjordholm2011}
	\begin{equation*}
	  m_\xr=\mean{h}_\xr\mean{\vx}_\xr\equiv C_1, ~p_i=(\vx)_i^2/2+g(h_i+b_i)\equiv C_2,~\forall i.
	\end{equation*}
 In fact, it is easy to verify that
	\begin{align*}
	&\dfrac{\dd}{\dd t}h_i
	=-\dfrac{1}{\Delta x}\left(m_\xr-m_\xl\right)\equiv0,  \\
	&\dfrac{\dd}{\dd t}(h\vx)_i
    =-\dfrac{g}{\Delta x}\Big[\frac12\left(\mean{h}_\xr\jump{p}_\xr+\mean{h}_\xl\jump{p}_\xl\right) + (\vx)_i\left(m_\xr-m_\xl\right) \Big]\equiv0.
	\end{align*}
\end{rmk}
\subsection{High-order EC schemes}\label{subsec:HOEC}
To develop the high-order well-balanced EC schemes, our task is to
get the high-order numerical fluxes and conduct the matched high-order discretization of the source term
related to the non-flat river bed bottom and  the Janhunen source term in \eqref{eq:symm1D}.

 Following the  way in \cite{Lefloch2002},
the EC flux of the $2p$th-order ($p\in \bbN^+$) accurate scheme
can be obtained by
using the linear combinations of the ``second-order accurate'' EC flux \eqref{eq:ecFlux1D} as follows
\begin{equation}\label{eq:ecFlux1D-high}
\tbF^{2p\mbox{\scriptsize th}}_\xr=\sum_{r=1}^p\alpha_r^p\sum_{s=0}^{r-1}\tbF_1(\bU_{i-s},\bU_{i-s+r}),
\end{equation}
which satisfies
\begin{equation*}
\dfrac{1}{\Delta x}\left(\tbF^{2p\mbox{\scriptsize th}}_\xr-\tbF^{2p\mbox{\scriptsize th}}_\xl\right)=
{\pd{\bF_1}{x}\Big|_{i}+}\mathcal{O}(\Delta x^{2p}).
\end{equation*}
The readers are referred to \cite{Fjordholm2012,Lefloch2002} for more details on
constructing the ``high-order accurate'' EC flux.

To make the resulting schemes  high-order   accurate, well-balanced and EC, it is essential that the high-order finite difference approximations of the spatial derivatives $\partial{(h\Bx)}/\partial{x}$ and $\partial{b}/\partial{x}$ in the source terms should match the ``high-order accurate'' EC flux \eqref{eq:ecFlux1D-high}. 
To this end, based on the observation
that the second-order central differences for the source terms
\begin{equation*}
\dfrac{(h\Bx)_{i+1}-(h\Bx)_{i-1}}{2\Delta x}=\dfrac{\mean{h\Bx}_\xr-\mean{h\Bx}_\xl}{\Delta x},~
\dfrac{b_{i+1}-b_{i-1}}{2\Delta x}=\dfrac{\mean{b}_\xr-\mean{b}_\xl}{\Delta x},
\end{equation*}
used in the second-order EC scheme \eqref{eq:1Dsemi}, have
the same form as the discretization of the flux gradient,
  using those second-order central differences as a building block can obtain the high-order approximations of the source terms
as follows
\begin{equation*}
(\widetilde{{h\Bx}})^{2p\mbox{\scriptsize th}}_\xr=\dfrac12\sum_{r=1}^p\alpha_r^p\sum_{s=0}^{r-1}
\left[(h\Bx)_{i-s}+(h\Bx)_{i-s+r}\right],\ \
~(\widetilde{b})^{2p\mbox{\scriptsize th}}_\xr=\dfrac12\sum_{r=1}^p\alpha_r^p\sum_{s=0}^{r-1}\left(b_{i-s}+b_{i-s+r}\right),
\end{equation*}
where the linear combination coefficients are the same as those
in the ``high-order accurate EC flux'' \eqref{eq:ecFlux1D-high}.
Similarly, it is not difficult  to verify
\begin{align*}
\dfrac{1}{\Delta x}\left(\widetilde{(h\Bx)}^{2p\mbox{\scriptsize th}}_\xr-\widetilde{(h\Bx)}^{2p\mbox{\scriptsize th}}_\xl\right)={\pd{(h\Bx)}{x}\Big|_{i}+}\mathcal{O}(\Delta x^{2p}),\\
~\dfrac{1}{\Delta x}\left(\widetilde{(b)}^{2p\mbox{\scriptsize th}}_\xr-\widetilde{(b)}^{2p\mbox{\scriptsize th}}_\xl\right)={\pd{b}{x}\Big|_{i}+}\mathcal{O}(\Delta x^{2p}).
\end{align*}
Such treatment can also be found in \cite{Derigs2018}.

In summary, by approximating \eqref{eq:symm1D}, we  obtain the following $2p$th order semi-discrete EC scheme
\begin{equation}\label{eq:1Dsemi_HO}
\dfrac{\dd}{\dd t}\bU_i
=-\dfrac{1}{\Delta x}\left(\tbF^{2p\mbox{\scriptsize th}}_\xr-\tbF^{2p\mbox{\scriptsize th}}_\xl\right)
-\dfrac{\Psi^\mathrm{T}_i}{\Delta x}\left(\widetilde{(h\Bx)}^{2p\mbox{\scriptsize th}}_\xr-\widetilde{(h\Bx)}^{2p\mbox{\scriptsize th}}_\xl\right)
-\dfrac{(\bm{G}_1)_i^\mathrm{T}}{\Delta x}\left(\widetilde{(b)}^{2p\mbox{\scriptsize th}}_\xr-\widetilde{(b)}^{2p\mbox{\scriptsize th}}_\xl\right),
\end{equation}
which satisfies  the entropy identity
 \eqref{eq:discEnEq}
with the numerical entropy flux
\begin{equation*}
  (\widetilde{q})^{2p\mbox{\scriptsize th}}_\xr=\sum_{r=1}^p\alpha_r^p\sum_{s=0}^{r-1}\widetilde{q}(\bU_{i-s},\bU_{i-s+r}).
\end{equation*}
It is a linear combination of the two-point numerical entropy flux \eqref{eq:numEntropyFlux}.
For example, when $p=3$, the  expression of the ``$6$th-order accurate'' EC flux is explicitly given as follows
\begin{align}\label{eq:6thEC}
  \widetilde{\bF}^{6\mbox{\scriptsize th}}_{\xr}=&\dfrac32\widetilde{\bF}{_1}(\bU_i,\bU_{i+1})
  -\dfrac{3}{10}\left[\widetilde{\bF}{_1}(\bU_{i-1},\bU_{i+1})
  +\widetilde{\bF}{_1}(\bU_i,\bU_{i+2})\right]\nonumber\\
  &+\dfrac{1}{30}\left[\widetilde{\bF}{_1}(\bU_{i-2},\bU_{i+1})
  +\widetilde{\bF}{_1}(\bU_{i-1},\bU_{i+2})
  +\widetilde{\bF}{_1}(\bU_i,\bU_{i+3})\right].
\end{align}

It can also be verified that
 the scheme \eqref{eq:1Dsemi_HO} is well-balanced in the sense of Theorem \ref{Theorem3.3},
  since the numerical fluxes and the numerical source terms in \eqref{eq:1Dsemi_HO} are formed by the linear combinations of the fluxes and the source terms in the second-order scheme \eqref{eq:1Dsemi} with the same coefficients,
  	specifically, the second equation in \eqref{eq:1Dsemi_HO} is
  written as follows
  	\begin{align*}
  	\dfrac{\dd}{\dd t}(h\vx)_i
  	=&-\dfrac{g}{2\Delta x}\sum_{r=1}^p\alpha_r^p\Big[\dfrac{h_{i+r}+h_i}{2}\left((h+b)_{i+r}-(h+b)_i\right)+\dfrac{h_{i}+h_{i-r}}{2}\left((h+b)_{i}-(h+b)_{i-r}\right) \Big].
  	\end{align*}
For the 1D moving equilibrium states discussed in Remark \ref{remark3.5}, one needs to impose very restrictive conditions
	\begin{align*}
	  & \left(h_{i}+h_{i\pm r}\right)\left((\vx)_{i}+(\vx)_{i\pm r}\right)\equiv C_1, ~\forall i, ~ r=1,\cdots,p, \\
	  &p_i=(\vx)_i^2/2+g(h_i+b_i)\equiv C_2,~\forall i.
	\end{align*}

\subsection{ES schemes}

It is known that for the quasi-linear hyperbolic conservation laws or balance laws,
the entropy identity is available only if the solution is smooth.
In other words, the entropy is not conserved if the discontinuities such as the
shock waves appear in the solution.
Moreover, the EC scheme may produce serious unphysical oscillations near
the discontinuities.
 Those motivate us to develop the ES scheme in this section by adding a suitable dissipation term to the EC scheme to avoid the unphysical oscillations produced by the EC scheme and to satisfy the entropy inequality for { the } given
entropy pair.

Following \cite{Tadmor1987}, adding a dissipation term  to the EC flux $\widetilde{\bm{F}}_{\xr}$ gives the ES flux
\begin{align}\label{eq:stableflux}
  \widehat{\bF}_{\xr}=\widetilde{\bF}_{\xr}-\dfrac12 \bm{D}_{\xr}\jump{\bV}_{\xr},
\end{align}
satisfying
\begin{equation}\label{eq:es_flux}
  \jump{\bV}^\mathrm{T}\cdot \widehat{\bF}_{\xr}-\jump{\psix}+\jump{\Phi}\mean{h\Bx}-g\jump{hb\vx}+g\jump{h\vx}\mean{b}\leqslant 0,
\end{equation}
where $\bm{D}_{\xr}$ is  a symmetric positive semi-definite matrix.
It is easy to prove that the scheme \eqref{eq:1Dsemi} or \eqref{eq:1Dsemi_HO} with the numerical flux \eqref{eq:stableflux}
is ES, i.e,  satisfying the semi-discrete entropy inequality
\begin{equation*}
  \dfrac{\dd}{\dd t}\eta(\bU_i(t))+\dfrac{1}{\Delta x}\left(\widehat{q}_{\xr}(t)-\widehat{q}_{\xl}(t)\right)\leqslant 0,
\end{equation*}
for some numerical entropy flux function $\widehat{q}_{\xr}$ consistent with the physical entropy flux $q_1$.


	Motivated by the Cholesky decomposition and
the dissipation term in the (local) Lax-Friedrichs flux
		\begin{equation*}
		-\dfrac12\alpha_\xr\jump{\bU}_\xr
		\approx-\dfrac12\alpha_\xr\pd{\bU}{\bV}\Big|_\xr\jump{\bV}
	=-\dfrac12\alpha_\xr\bm{R}_\xr\bm{R}^\mathrm{T}_\xr\jump{\bV},
		\end{equation*}
		the matrix $\bm{D}_\xr$ in \eqref{eq:stableflux} can be chosen as
		\begin{equation*}
		\bm{D}_{\xr}=\alpha_\xr\bm{R}_\xr\bm{R}^\mathrm{T}_\xr.
		\end{equation*}
		Here $\alpha_\xr=\max_{m=i,i+1}\left\{\abs{(\vx^n)_m}
    +\sqrt{gh_m^n+(\Bx^n)_m^2} \right\}$ and $\bm{R}\bm{R}^T$ is the Cholesky decomposition of the matrix $\pd{\bU}{\bV}$ with
\begin{equation*}
		  \bm{R}=\begin{pmatrix}
		  1/\sqrt{g} & 0 & 0 & 0 & 0 \\
		  \vx/\sqrt{g} & \sqrt{h} & 0 & 0 & 0 \\
		  \vy/\sqrt{g} & 0 & \sqrt{h} & 0 & 0 \\
		  \Bx/\sqrt{g} & 0 & 0 & \sqrt{h} & 0 \\
		  \By/\sqrt{g} & 0 & 0 & 0 & \sqrt{h}
		  \end{pmatrix},
		\end{equation*}
and $\alpha_\xr$ and $\bm{R}_{\xr}$ are calculated by using the arithmetic mean values
$\mean{h}_{\xr}$, $\mean{\bv}_{\xr}$, and
		$\mean{\bm{B}}_{\xr}$.


To obtain the arbitrary high-order accurate ES scheme,
the dissipation term in \eqref{eq:stableflux} has to
be improved.
For example, it can be done by using the ENO reconstruction of the scaled entropy variables
$\bw=\bm{R}^\mathrm{T}\bV$ \cite{Fjordholm2012}.
More specifically, use the $k$th order accurate ENO
reconstruction of $\bw$ to obtain the left and right limit values at $x_{\xr}$, denoted by
$\bw_{\xr}^-$ and $\bw_{\xr}^+$, and then define
\begin{equation*}
  \jumpangle{\bw}_{\xr}=\bw_{\xr}^+-\bw_{\xr}^-.
\end{equation*}
Combining such reconstructed jump with the ``$2p$th-order EC flux''
$\tilde{\bF}^{2p\mbox{\scriptsize th}}$ and the $2p$th-order discretization of the source terms
gives the $k$th order ES scheme as follows
\begin{equation}\label{eq:1DES_HO}
\dfrac{\dd}{\dd t}\bU_i=-\dfrac{1}{\Delta x}\left(\widehat{\bF}^{k\mbox{\scriptsize th}}_\xr-\widehat{\bF}^{k\mbox{\scriptsize th}}_\xl\right)
-\dfrac{\Psi_i^\mathrm{T}}{\Delta x}\left(\widetilde{(h\Bx)}^{2p\mbox{\scriptsize th}}_\xr-\widetilde{(h\Bx)}^{2p\mbox{\scriptsize th}}_\xl\right)
-\dfrac{(\bm{G}_1)_i^\mathrm{T}}{\Delta x}\left(\widetilde{(b)}^{2p\mbox{\scriptsize th}}_\xr-\widetilde{(b)}^{2p\mbox{\scriptsize th}}_\xl\right),
\end{equation}
where   $p=k/2$ for even $k$ and $p=(k+1)/2$ for odd $k$, and
\begin{equation}\label{eq:HOstable}
  \widehat{\bF}^{k\mbox{\scriptsize th}}_{\xr}=\widetilde{\bF}^{2p\mbox{\scriptsize th}}_{\xr}
  -\dfrac12 \alpha_\xr\bm{R}_{\xr}\jumpangle{\bw}_{\xr}.
\end{equation}
The semi-discrete numerical schemes \eqref{eq:1DES_HO}
is ES if the reconstruction satisfies the  ``sign'' property \cite{Fjordholm2012}
\begin{equation*}
  \text{sign}(\jumpangle{\bw}_{\xr})=\text{sign}(\jump{\bw}_{\xr}).
\end{equation*}
which does hold for the ENO reconstructions \cite{Fjordholm2013}.

Moreover, one can also obtain higher-order accurate ES scheme with the WENO reconstruction instead of
the ENO reconstruction, if the same number of candidate points values are used.
In view of that a general WENO reconstruction may not satisfy the ``sign'' property, following \cite{Biswas2018},
the dissipation term in \eqref{eq:stableflux} may be modified as follows
\begin{equation}\label{eq:switch}
  	\widehat{\bF}^{k\mbox{\scriptsize th}}_{\xr}=\widetilde{\bF}^{2p\mbox{\scriptsize th}}_{\xr}
  	-\dfrac12 \alpha_\xr\bm{S}_{\xr}\bm{R}_{\xr}\jumpangle{\bw}_{\xr}.
\end{equation}
where $\bm{S}^l_{\xr}$ is a switch function defined by
\begin{equation*}
  \bm{S}^l_{\xr}=\begin{cases}
    1,\quad &\text{if}
    ~\text{sign}(\jumpangle{\bw}^l_{\xr})=\text{sign}(\jump{\bw}^l_{\xr})\not= 0, \\
    0,\quad &\text{otherwise},
  \end{cases}
\end{equation*}
here the superscript $l$ denotes the $l$th entry of the diagonal matrix $\bm{S}_{\xr}$
or the $l$th component of the jump of $\bw$.
One can verify that  the adding dissipation term becomes zero when the WENO reconstruction does not satisfy the ``sign''
property, and thus the semi-discrete numerical
scheme with the flux \eqref{eq:switch} is ES.

\begin{rmk}
	At the steady state, the entropy variables $\bV^\mathrm{T}$ become $\left(gC,0,0,0,0\right)$,
	so that   the low-order dissipation term with $\jump{\bV}$ and the high-order dissipation term with
	$\jumpangle{\bw}=\bm{R}\jumpangle{\bV}$  all vanish.
Thus the constructed ES schemes are well-balanced.
\end{rmk}

\subsection{Time discretization}
This paper  uses the following third-order accurate strong stability preserving explicit Runge-Kutta (RK3) time
discretization
\begin{equation}\label{eq:rk3}
\begin{aligned}
  &\bU^{(1)}=\bU^n+\Delta t \bm{L}(\bU^n),\\
  &\bU^{(2)}=\dfrac34\bU^n+\dfrac14\left(\bU^{(1)}+\Delta t \bm{L}(\bU^{(1)})\right),\\
  &\bU^{n+1}=\dfrac13\bU^n+\dfrac23\left(\bU^{(2)}+\Delta t \bm{L}(\bU^{(2)})\right),
\end{aligned}
\end{equation}
to  integrate in time the semi-discrete schemes \eqref{eq:1Dsemi}, \eqref{eq:1Dsemi_HO} or \eqref{eq:1DES_HO},
where $[\bm{L}(\bU)]$ corresponds to their right-hand side.

\subsection{Positivity-preserving ES schemes}
  This section restricts to the flat bottom topography.
Generally, the high-order ES scheme \eqref{eq:1DES_HO}   integrated with the RK3 \eqref{eq:rk3}
may numerically produce the negative water height 
so that  the numerical simulation fails.
This section develops a high-order positivity-preserving  ES scheme
based on \eqref{eq:1DES_HO},
satisfying $h_i^{n+1}>\varepsilon,\forall i$, if $h_i^n>\varepsilon,\forall i$ for
a small positive number $\varepsilon$ (usually taken as $10^{-13}$
in numerical tests), which means there is no dry area in the solutions.
Since the RK3 \eqref{eq:rk3} is a convex combination of the forward Euler time discretization, it only needs to consider
the first component of the semi-discrete scheme \eqref{eq:1DES_HO} with the forward Euler time discretization, that is
\begin{align}
h^{n+1}_i=h^n_i-\dfrac{\Delta t}{\Delta x}\left[h(\widehat{\bF}_\xr^{k\mbox{\scriptsize th}}) - h(\widehat{\bF}_\xl^{k\mbox{\scriptsize th}})\right],
\end{align}
where $h(\widehat{\bF})$ denotes the first component of the vector $\widehat{\bF}$.

Before that, we first prove that the scheme \eqref{eq:1Dsemi} with the local Lax-Friedrichs flux is  positivity-preserving.

\begin{lem}
  The semi-discrete scheme \eqref{eq:1Dsemi}  discretized with the forward Euler time discretization and
  with the local Lax-Friedrichs flux
  \begin{align}\label{LLF-flux}
    \widehat{\bF}_\xr^{LF}=\mean{\bF_1}_\xr-{\alpha_\xr}\jump{\bU}_\xr/{2}, \quad
    \alpha_\xr=\max_{m=i,i+1}\left\{\abs{(\vx^n)_m}
    +\sqrt{gh_m^n+(\Bx^n)_m^2} \right\},
  \end{align}
  is positivity-preserving
  under the CFL condition
  \begin{align}\label{eq:CFL}
  \Delta t=\dfrac{\mu\Delta x}{\max\limits_{i}\left\{\abs{(\vx)_i}+\sqrt{gh_i+(\Bx)_i^2}\right\}},\quad \mu 
  \leqslant \frac12.
  \end{align}
\end{lem}
\begin{proof}
The first component of the Lax-Friedrichs scheme can be split as
\begin{align*}
h_i^{n+1}=:\dfrac12\left(h^{+,\mbox{\scriptsize LF}}_i+h^{-,\mbox{\scriptsize LF}}_i\right),\quad
h^{\pm,\mbox{\scriptsize LF}}_i=h_i^n\mp\dfrac{2\Delta t}{\Delta x}h(\widehat{\bF}_{i\pm \frac12}^{LF}),
\end{align*}
so it holds
\begin{align*}
  h^{\pm,\mbox{\scriptsize LF}}_i=h_i^n\left(1-\frac{\Delta t}{\Delta x}\left(\alpha_{i\pm\frac12} \pm(\vx)^n_i\right)\right)
  +\frac{\Delta t}{\Delta x}h_{i\pm1}^n\left(\alpha_{i\pm\frac12}\mp(\vx)^n_{i\pm1}\right).
\end{align*}
Thus $h^{\pm,\mbox{\scriptsize LF}}_i>0$ under the CFL condition \eqref{eq:CFL}, and then $h^{n+1}_i>0$.
\end{proof}

The following lemma shows that the Lax-Friedrichs flux \eqref{LLF-flux} is ES even if $h\Bx$ or $h\By$ is not constant. 
Its proof is direct without the assumption that the
1D exact Riemann solution of $x$-split system is ES.
In fact, when there are jumps in $h\Bx$ or $h\By$ at the cell interface in the 2D case, whether such assumption is available needs further investigation. 

\begin{lem}
	The Lax-Friedrichs flux is ES when the bottom topography is flat.
\end{lem}

\begin{proof}
	Substituting the Lax-Friedrichs flux into the inequality \eqref{eq:es_flux} and using identity \eqref{eq:jumpid} gives
	\begin{align}
	&\jump{\bV}^\mathrm{T}\cdot \widehat{\bF}^{\mbox{\scriptsize LF}}
-\jump{\psix}+\jump{\Phi}\mean{h\Bx} \nonumber
\\ =&g\mean{h\vx}\jump{h}-\frac{g\alpha}{2}\jump{h}^2+\frac{g}{2}\mean{h^2}\jump{\vx}-\frac{g}{2}\jump{h^2\vx} \nonumber\\
	&-\frac12\mean{h\vx}\jump{\vx^2+\vy^2+\Bx^2+\By^2}+\frac{\alpha}{4}\jump{h}\jump{\vx^2+\vy^2+\Bx^2+\By^2} \nonumber \\
	&+\mean{h\vx^2-h\Bx^2}\jump{\vx}+\mean{h\vx\vy-h\Bx\By}\jump{\vy}-\frac{\alpha}{2}\jump{h\vx}\jump{\vx}-\frac{\alpha}{2}\jump{h\vy}\jump{\vy} \nonumber\\
	&+\mean{h\vx\Bx-h\Bx\vx}\jump{\Bx}+\mean{h\vx\By-h\Bx\vy}\jump{\By}-\frac{\alpha}{2}\jump{h\Bx}\jump{\Bx}-\frac{\alpha}{2}\jump{h\By}\jump{\By}\nonumber\\
	&+\jump{\vx\Bx+\vy\By}\mean{h\Bx}\nonumber \\
	=&g\left( (\mean{h\vx}-\mean{h}\mean{\vx})\jump{h}-\frac{\alpha}{2}\jump{h}^2 \right) \nonumber\\
	&-\frac{\alpha}{2}\mean{h}\left(\jump{\vx}^2+\jump{\vy}^2+\jump{\Bx}^2+\jump{\By}^2 \right) \nonumber\\
	&+(\mean{h\vx^2}-\mean{h\vx}\mean{\vx})\jump{\vx}+(\mean{h\vx\vy}-\mean{h\vx}\mean{\vy})\jump{\vy} \nonumber\\
	&+(\mean{h\vx\Bx}-\mean{h\vx}\mean{\Bx})\jump{\Bx}+(\mean{h\vx\By}-\mean{h\vx}\mean{\By})\jump{\By} \nonumber\\
	&-(\mean{h\Bx^2}-\mean{h\Bx}\mean{\Bx})\jump{\vx}-(\mean{h\Bx\By}-\mean{h\Bx}\mean{\By})\jump{\vy} \nonumber\\
	&-(\mean{h\Bx\vx}-\mean{h\Bx}\mean{\vx})\jump{\Bx}-(\mean{h\Bx\vy}
-\mean{h\Bx}\mean{\vy})\jump{\By}.
\nonumber
	\end{align}
	It can be further simplified by using the identity $\mean{ab}-\mean{a}\mean{b}=\frac{1}{4}\jump{a}\jump{b}$ as follows
	\begin{align}
	&\jump{\bV}^\mathrm{T}\cdot
\widehat{\bF}^{\mbox{\scriptsize LF}}
-\jump{\psix}+\jump{\Phi}\mean{h\Bx} \nonumber\\
	=&\frac{g}{4}\jump{h}^2\left( \jump{\vx}^2-\alpha \right)-\frac{1}{2}\jump{h\Bx}\left(\jump{\vx}\jump{\Bx}+\jump{\vy}\jump{\By}\right) \nonumber\\
	&+\left(\frac14\jump{h\vx}-\frac{\alpha}{2}\mean{h}\right)\left(\jump{\vx}^2+\jump{\vy}^2+\jump{\Bx}^2+\jump{\By}^2 \right) \nonumber\\
	&+(\mean{h\vx^2}-\mean{h\vx}\mean{\vx})\jump{\vx}+(\mean{h\vx\vy}-\mean{h\vx}\mean{\vy})\jump{\vy} \nonumber\\
	\triangleq&\mathcal{A}+\mathcal{B}+\mathcal{C},
\nonumber
	\end{align}
	where
	\begin{align*}	
	\mathcal{A}&=\frac{g}{4}\jump{h}^2\left( \jump{\vx}-2\alpha \right), \\
	\mathcal{B}&=\frac14\left(\jump{\vx}^2+\jump{\Bx}^2\right)\left(\jump{h\vx}-2\alpha\mean{h}\right)-\frac12\jump{h\Bx}\jump{\vx}\jump{\Bx},\\
	\mathcal{C}&=\frac14\left(\jump{\vy}^2+\jump{\By}^2\right)\left(\jump{h\vx}-2\alpha\mean{h}\right)-\frac12\jump{h\Bx}\jump{\vy}\jump{\By}.
	\end{align*}
	Since $\alpha=\max\big\{\abs{(\vx)_L}+\sqrt{gh_L+(\Bx)_L^2}$, $\abs{(\vx)_R}+\sqrt{gh_R+(\Bx)_R^2}\big\}$,
	it is easy to obtain
	\begin{equation*}
	\jump{\vx}\leqslant2\alpha, ~(\vx\pm\Bx)_{L,R}\leqslant\alpha,
	\end{equation*}
	then $\mathcal{A}\leqslant0$.

	If $\jump{h\Bx}\geqslant0$, then $\mathcal{B}\leqslant0$, because	$$\jump{h(\vx+\Bx)}-2\alpha\mean{h}=h_R\left[(\vx+\Bx)_R
-\alpha\right]-h_L\left[(\vx+\Bx)_L+\alpha\right]\leqslant0,$$
	and
	\begin{align*} \mathcal{B}=\frac14\left(\jump{\vx}^2
+\jump{\Bx}^2\right)\left(\jump{h(\vx+\Bx)}-2\alpha\mean{h}\right)
-\left(\jump{\vx}+\jump{\Bx}\right)^2\jump{h\Bx}; 
	\end{align*}
 otherwise, 
  it still holds that $\mathcal{B}\leqslant0$, because $$\jump{h(\vx-\Bx)}-2\alpha\mean{h}=h_R\left[(\vx-\Bx)_R-\alpha\right]
 -h_L\left[(\vx-\Bx)_L+\alpha\right]\leqslant0,$$
	and 	
\begin{align*}
	\mathcal{B} &=\frac14\left(\jump{\vx}^2+\jump{\Bx}^2\right)\left(\jump{h(\vx-\Bx)}
-2\alpha\mean{h}\right)+\left(\jump{\vx}-\jump{\Bx}\right)^2\jump{h\Bx}.
	\end{align*}
Similarly, $\mathcal{C}\leqslant0$. Therefore the Lax-Friedrichs flux is ES.
\end{proof}

Based on those discussions,  the high-order positivity-preserving ES schemes  can be constructed by using the Lax-Friedrichs flux and the  positivity-preserving limiter \cite{Hu2013}.
The limited numerical flux $\widehat{\bF}_\xr^{k\mbox{\scriptsize th,PP}}$ is given by
\begin{equation*}
  \widehat{\bF}_\xr^{k\mbox{\scriptsize th,PP}}=\theta_\xr\widehat{\bF}_\xr^{k\mbox{\scriptsize th}}+
  (1-\theta_\xr)\widehat{\bF}_\xr^{\mbox{\scriptsize LF}},
\end{equation*}
where $\theta_\xr=\min\{\theta_\xr^+, \theta_\xr^-\}\in[0,1]$ is the scaling factor corresponding to the two neighboring grid points,
which share the same flux $\widehat{\bF}^{k\mbox{\scriptsize th}}_\xr$, and
\begin{equation*}
  \theta_\xr^\pm=\begin{cases}
    \left(h_{\xr\mp\frac12}^{\pm,\mbox{\scriptsize LF}}-\varepsilon\right) /
     \left(h_{\xr\mp\frac12}^{\pm,\mbox{\scriptsize LF}}-h_{\xr\mp\frac12}^{\pm,k\mbox{\scriptsize th}}\right),
     \quad & \text{if}~h_{\xr\mp\frac12}^{\pm,k\mbox{\scriptsize th}}<\varepsilon,\\
     1, \quad & \text{otherwise}.
  \end{cases}
\end{equation*}
	It is worth noting that the discretization of the source terms should also be replaced by the corresponding
	convex combinations as follows
	\begin{align*}
	&\widetilde{(h\Bx)}^{2p\mbox{\scriptsize th,PP}}_\xr=\theta_\xr\widetilde{(h\Bx)}^{2p\mbox{\scriptsize th}}_\xr+
	(1-\theta_\xr)\mean{h\Bx}_\xr.
	\end{align*}
It is easy to verify that the  water height updated by the high-order positivity-preserving ES schemes  satisfies
\begin{align*}
  h^{n+1}_i=&\dfrac12\left(h^n_i-\frac{2\Delta t}{\Delta x}h(\widehat{\bF}_\xr^{k\mbox{\scriptsize th,PP}})\right)
  +\dfrac12\left(h^n_i+\frac{2\Delta t}{\Delta x}h(\widehat{\bF}_\xl^{k\mbox{\scriptsize th,PP}})\right)\\
  =&\dfrac12\left[\theta_\xr h_{i}^{+,k\mbox{\scriptsize th}}
  +(1-\theta_\xr)h_{i}^{+,\mbox{\scriptsize LF}}\right]
  +\dfrac12\left[\theta_\xl h_{i}^{-,k\mbox{\scriptsize th}}
  +(1-\theta_\xl)h_{i}^{-,\mbox{\scriptsize LF}}\right]
    >\varepsilon,
\end{align*}
and the limited flux $\widehat{\bF}_\xr^{k\mbox{\scriptsize th,PP}}$ is consistent and ES since it is a convex combination of the high-order ES flux
and the ES Lax-Friedrichs flux, and   does not destroy the high-order accuracy
\begin{equation*}
  \norm{\widehat{\bF}_\xr^{k\mbox{\scriptsize th,PP}}-\widehat{\bF}_\xr^{k\mbox{\scriptsize th}}}
  \leqslant (1-\theta_\xr)\norm{\widehat{\bF}_\xr^{\mbox{\scriptsize LF}}-\widehat{\bF}_\xr^{k\mbox{\scriptsize th}}},
\end{equation*}
with $1-\theta_\xr=\mathcal{O}(\Delta x^{k})$ \cite{Hu2013}.

\section{Two-dimensional schemes}\label{section:MultiD}
This section extends the 1D high-order EC and ES schemes developed in Section \ref{section:OneD}
to the 2D SWMHD system \eqref{eq:SWMHD}.
For convenience, 
the notation $(x_1,x_2)$ is
replaced with $(x,y)$.  Our attention is limited to
a uniform Cartesian mesh
$\{(x_i,y_j),~i=1,\cdots,N_x,~j=1,\cdots,N_y\}$
with the spatial step sizes $\Delta x,\Delta y$ so that
the extension of the 1D  schemes to \eqref{eq:SWMHD}
can be done  by approximating \eqref{eq:symm} in a dimension by dimension fashion.
To avoid repetition, the detailed extension is not described below.

 At each grid point $(x_i,y_j),~i=1,\cdots,N_x,~j=1,\cdots,N_y$,
the 2D SWMHD system \eqref{eq:SWMHD} 
can be approximated by the following second-order accurate well-balanced semi-discrete EC scheme
\begin{align}\label{eq:2DECsemi}
\dfrac{\dd}{\dd t}\bU_{i,j}&+\dfrac{1}{\Delta x}\left(\widetilde{\bF}_{1,i+\frac12,j}-\widetilde{\bF}_{1,i-\frac12,j}\right)
+\dfrac{1}{\Delta y}\left(\widetilde{\bF}_{2,i,j+\frac12}-\widetilde{\bF}_{2,i,j-\frac12}\right)= \nonumber \\
&- \Psi_{i,j}^\mathrm{T}\dfrac{\mean{h\Bx}_{\xr,j}-\mean{h\Bx}_{\xl,j}}{\Delta x}
- \Psi_{i,j}^\mathrm{T}\dfrac{\mean{h\By}_{i,\yr}-\mean{h\By}_{i,\yl}}{\Delta y} \nonumber\\
&-(\bm{G}_1)_{i,j}^\mathrm{T}\dfrac{\mean{b}_{\xr,j}-\mean{b}_{\xl,j}}{\Delta x}
-(\bm{G}_2)_{i,j}^\mathrm{T}\dfrac{\mean{b}_{i,\yr}-\mean{b}_{i,\yl}}{\Delta y},
\end{align}
where $\bm{G}_2=(0,~0,~gh,~0,~0)^\mathrm{T}$,
and $\widetilde{\bF}_{1,i\pm\frac12,j}$ and $\widetilde{\bF}_{2,i,j\pm\frac12}$ are
the $x$- and $y$-directional EC fluxes, respectively.

Similarly, using the EC scheme \eqref{eq:2DECsemi} as building block
can give a $2p$th-order well-balanced semi-discrete EC scheme for
the 2D SWMHD system \eqref{eq:SWMHD} 
 as follows
\begin{align}\label{eq:2DECsemi_HO}
\dfrac{\dd}{\dd t}\bU_{i,j}&+\dfrac{1}{\Delta x}\left(\widetilde{\bF}^{2p\mbox{\scriptsize th}}_{1,i+\frac12,j}-\widetilde{\bF}^{2p\mbox{\scriptsize th}}_{1,i-\frac12,j}\right)
+\dfrac{1}{\Delta y}\left(\widetilde{\bF}^{2p\mbox{\scriptsize th}}_{2,i,j+\frac12}-\widetilde{\bF}^{2p\mbox{\scriptsize th}}_{2,i,j-\frac12}\right)= \nonumber \\
&-\dfrac{ \Psi_{i,j}^\mathrm{T}}{\Delta x}\left({\widetilde{(h\Bx)}^{2p\mbox{\scriptsize th}}_{\xr,j}-\widetilde{(h\Bx)}^{2p\mbox{\scriptsize th}}_{\xl,j}}\right)
-\dfrac{\Psi_{i,j}^\mathrm{T}}{\Delta y}\left({\widetilde{(h\By)}^{2p\mbox{\scriptsize th}}_{i,\yr}-\widetilde{(h\By)}^{2p\mbox{\scriptsize th}}_{i,\yl}}\right) \nonumber\\
&-\dfrac{(\bm{G}_1)_{i,j}^\mathrm{T}}{\Delta x}\left({\widetilde{(b)}^{2p\mbox{\scriptsize th}}_{\xr,j}-\widetilde{(b)}^{2p\mbox{\scriptsize th}}_{\xl,j}}\right)
-\dfrac{(\bm{G}_2)_{i,j}^\mathrm{T}}{\Delta y}\left({\widetilde{(b)}^{2p\mbox{\scriptsize th}}_{i,\yr}-\widetilde{(b)}^{2p\mbox{\scriptsize th}}_{i,\yl}}\right),
\end{align}
where
\begin{equation*}
\begin{aligned}
&\widetilde{\bF}^{2p\mbox{\scriptsize th}}_{1,i+\frac12,j}=\sum_{r=1}^p\alpha_r^p\sum_{s=0}^{r-1}\tbF_1(\bU_{i-s,j},\bU_{i-s+r,j}), \\
&\widetilde{\bF}^{2p\mbox{\scriptsize th}}_{2,i,j+\frac12}=\sum_{r=1}^p\alpha_r^p\sum_{s=0}^{r-1}\tbF_2(\bU_{i,j-s},\bU_{i,j-s+r}), \\
&(\widetilde{{h\Bx}})^{2p\mbox{\scriptsize th}}_{i+\frac12,j}=\dfrac12\sum_{r=1}^p\alpha_r^p\sum_{s=0}^{r-1}\left[(h\Bx)_{i-s,j}+(h\Bx)_{i-s+r,j}\right],\\
&(\widetilde{{h\By}})^{2p\mbox{\scriptsize th}}_{i,j+\frac12}=\dfrac12\sum_{r=1}^p\alpha_r^p\sum_{s=0}^{r-1}\left[(h\By)_{i,j-s}+(h\By)_{i,j-s+r}\right],\\
&(\widetilde{b})^{2p\mbox{\scriptsize th}}_{i+\frac12,j}=\dfrac12\sum_{r=1}^p\alpha_r^p\sum_{s=0}^{r-1}\left(b_{i-s,j}+b_{i-s+r,j}\right), \\
&(\widetilde{b})^{2p\mbox{\scriptsize th}}_{i,j+\frac12}=\dfrac12\sum_{r=1}^p\alpha_r^p\sum_{s=0}^{r-1}\left(b_{i,j-s}+b_{i,j-s+r}\right).
\end{aligned}
\end{equation*}
Then adding a suitable dissipation term to \eqref{eq:2DECsemi_HO}
gives  a $k$th-order well-balanced semi-discrete ES scheme for the 2D SWMHD system
\eqref{eq:SWMHD} 
as follows
\begin{align}\label{eq:2DESsemi_HO}
\dfrac{\dd}{\dd t}\bU_{i,j}&+\dfrac{1}{\Delta x}\left(\widehat{\bF}^{k\mbox{\scriptsize th}}_{1,i+\frac12,j}-\widehat{\bF}^{k\mbox{\scriptsize th}}_{1,i-\frac12,j}\right)
+\dfrac{1}{\Delta y}\left(\widehat{\bF}^{k\mbox{\scriptsize th}}_{2,i,j+\frac12}-\widehat{\bF}^{k\mbox{\scriptsize th}}_{2,i,j-\frac12}\right)= \nonumber \\
&-\dfrac{ \Psi_{i,j}^\mathrm{T}}{\Delta x}\left({\widetilde{(h\Bx)}^{2p\mbox{\scriptsize th}}_{\xr,j}-\widetilde{(h\Bx)}^{2p\mbox{\scriptsize th}}_{\xl,j}}\right)
-\dfrac{ \Psi_{i,j}^\mathrm{T}}{\Delta y}\left({\widetilde{(h\By)}^{2p\mbox{\scriptsize th}}_{i,\yr}-\widetilde{(h\By)}^{2p\mbox{\scriptsize th}}_{i,\yl}}\right) \nonumber\\
&-\dfrac{(\bm{G}_1)_{i,j}^\mathrm{T}}{\Delta x}\left({\widetilde{(b)}^{2p\mbox{\scriptsize th}}_{\xr,j}-\widetilde{(b)}^{2p\mbox{\scriptsize th}}_{\xl,j}}\right)
-\dfrac{(\bm{G}_2)_{i,j}^\mathrm{T}}{\Delta y}\left({\widetilde{(b)}^{2p\mbox{\scriptsize th}}_{i,\yr}-\widetilde{(b)}^{2p\mbox{\scriptsize th}}_{i,\yl}}\right),
\end{align}
where  $p=k/2$ for even $k$ and $p=(k+1)/2$ for odd $k$,
\begin{equation*}
\begin{aligned}
\widehat{\bm{F}}^{k\mbox{\scriptsize th}}_{1,i+\frac12,j}=\widetilde{\bF}^{2p\mbox{\scriptsize th}}_{1,i+\frac12,j}
-\dfrac12\alpha_{i+\frac12,j}\bm{S}_{i+\frac12,j}
\bm{R}_{i+\frac12,j}\jumpangle{\bw}_{i+\frac12,j},
\\
\widehat{\bm{F}}^{k\mbox{\scriptsize th}}_{2,i,j+\frac12}=\widetilde{\bF}^{2p\mbox{\scriptsize th}}_{2,i,j+\frac12}
-\dfrac12\alpha_{i,\yr}\bm{S}_{i,\yr}\bm{R}_{i,\yr}\jumpangle{\bw}_{i,\yr},
\end{aligned}
\end{equation*}
 the jumps $\jumpangle{\bw}_{\xr,j},\jumpangle{\bw}_{i,\yr}$ are respectively obtained by using the WENO reconstruction in the $x$- and $y$-directions, and the viscosities $\alpha_{i+\frac12,j}$ and $\alpha_{i,j+\frac12}$ are respectively chosen in $x$- and $y$-directions.

For the time discretization, the third-order Runge-Kutta scheme \eqref{eq:rk3} is used.
The analysis of the EC, ES, and well-balanced properties  of the above 2D EC and ES schemes is similar to the 1D case, so that it is omitted here.
Moreover, the ES property of the Lax-Friedrichs flux can also be used to develop the 2D positivity-preserving
ES schemes by using the positivity-preserving flux limiter.
%
%

\section{Numerical results}\label{section:Num}
This section conducts some numerical experiments to validate the performance of our EC and ES schemes
for the SWMHD equations \eqref{eq:SWMHD} and its $x$-split system.
Unless otherwise stated, all computations take the CFL number $\mu$ as $0.5$,
and the 5th-order schemes  with the fifth-order WENO reconstruction in \cite{Borges2008}.
For the accuracy tests, the time stepsize $\Delta t$ is taken as
$\mu {\Delta x}^{6/3}$ (resp. $\mu {\Delta x}^{5/3}$)
for the $6$th order EC schemes (resp. the $5$th order
ES schemes) to make the spatial error dominant.

\subsection{One-dimensional case}
\begin{example}[Accuracy test]\label{ex:acc1D}\rm
  This example is used to verify the accuracy.
  The computational domain is $[0,1]$ with periodic boundary conditions, and $g=1$.
   The constructed exact solution is given as follows
  \begin{align*}
    &h(x, t)=1, \ \vx(x, t)=0,\ \vy(x, t)=\sin(2\pi (x+t)),
    \ \Bx(x, t)=1, \ \By(x, t)=\vy(x, t).
  \end{align*}
\end{example}
Table \ref{tab:acc1D} lists the errors and the orders of convergence in $\vy$ at $t=1$
obtained by using our EC   and ES schemes.
It is seen that these schemes get the sixth-order and the fifth-order
accuracy as expected.

\begin{table}[!htb]
  \centering
  \begin{tabular}{r|cc|cc|cc|cc} \hline
  \multirow{2}{*}{$N_x$} & \multicolumn{4}{c|}{EC scheme} & \multicolumn{4}{c}{ES scheme} \\ \cline{2-9}
                         & $\ell^1$ error & order & $\ell^\infty$ error & order & $\ell^1$ error & order & $\ell^\infty$ error & order \\ \hline
 10 & 1.575e-04 &  -   & 2.433e-04 &  -   & 1.126e-03 &  -   & 1.605e-03 &  -   \\
 20 & 2.706e-06 & 5.86 & 4.181e-06 & 5.86 & 3.015e-05 & 5.22 & 5.303e-05 & 4.92 \\
 40 & 4.276e-08 & 5.98 & 6.690e-08 & 5.97 & 9.048e-07 & 5.06 & 1.486e-06 & 5.16 \\
 80 & 6.700e-10 & 6.00 & 1.051e-09 & 5.99 & 2.830e-08 & 5.00 & 4.492e-08 & 5.05 \\
160 & 1.050e-11 & 6.00 & 1.650e-11 & 5.99 & 8.852e-10 & 5.00 & 1.393e-09 & 5.01 \\
 \hline
  \end{tabular}
  \caption{Example \ref{ex:acc1D}: Errors and orders of convergence in $\vy$ at $t=1$.}
  \label{tab:acc1D}
\end{table}

\begin{example}[Well-balanced test \cite{Zia2014}]\label{ex:1DWB}\rm
	It is used to verify the well-balanced property of our EC and ES schemes. 
	The bottom topography is taken as a smooth function
	\begin{equation}\label{eq:1Dsmooth_b}
	b(x)=0.2e^{-(x+1)^2/2}+0.3e^{-(x-1.5)^2},
	\end{equation}
	or a discontinuous function
	\begin{equation}\label{eq:1Ddiscontinuous_b}
	b(x)=0.5\chi_{[-4,4]},
	\end{equation}
and then the initial data are specified as $h(x)=1-b(x)$,
$\vx=0$, and $\vec B=0$.
	The computational domain is $[-10,10]$, and the problem is numerically solved until $t=10$ with $N_x=40$ and $g=1$.
\end{example}
The surface level $h+b$ and the bottom $b$ are shown in Figure \ref{fig:1DWB},
and the errors in $h$ and $\vx$ are given in Table \ref{tab:1DWB}.
It can be seen that the errors are at the level of round-off errors for the   double precision, and the well-balanced property is verified.

\begin{figure}[ht!]
	\begin{subfigure}[b]{0.5\textwidth}
		\centering
		\includegraphics[width=1.0\textwidth]{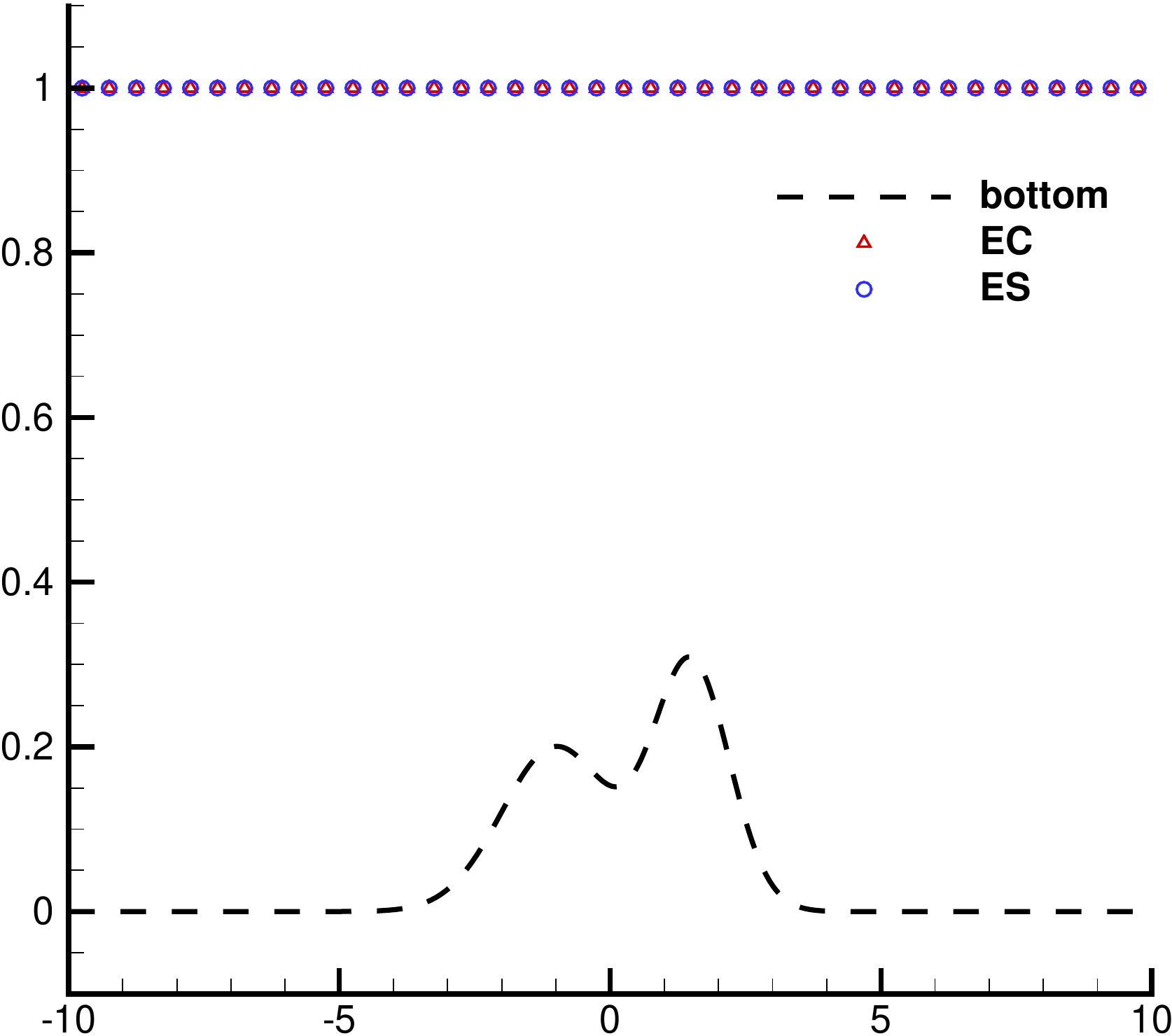}
		\caption{Bottom topography \eqref{eq:1Dsmooth_b}}
	\end{subfigure}
	\begin{subfigure}[b]{0.5\textwidth}
		\centering
		\includegraphics[width=1.0\textwidth]{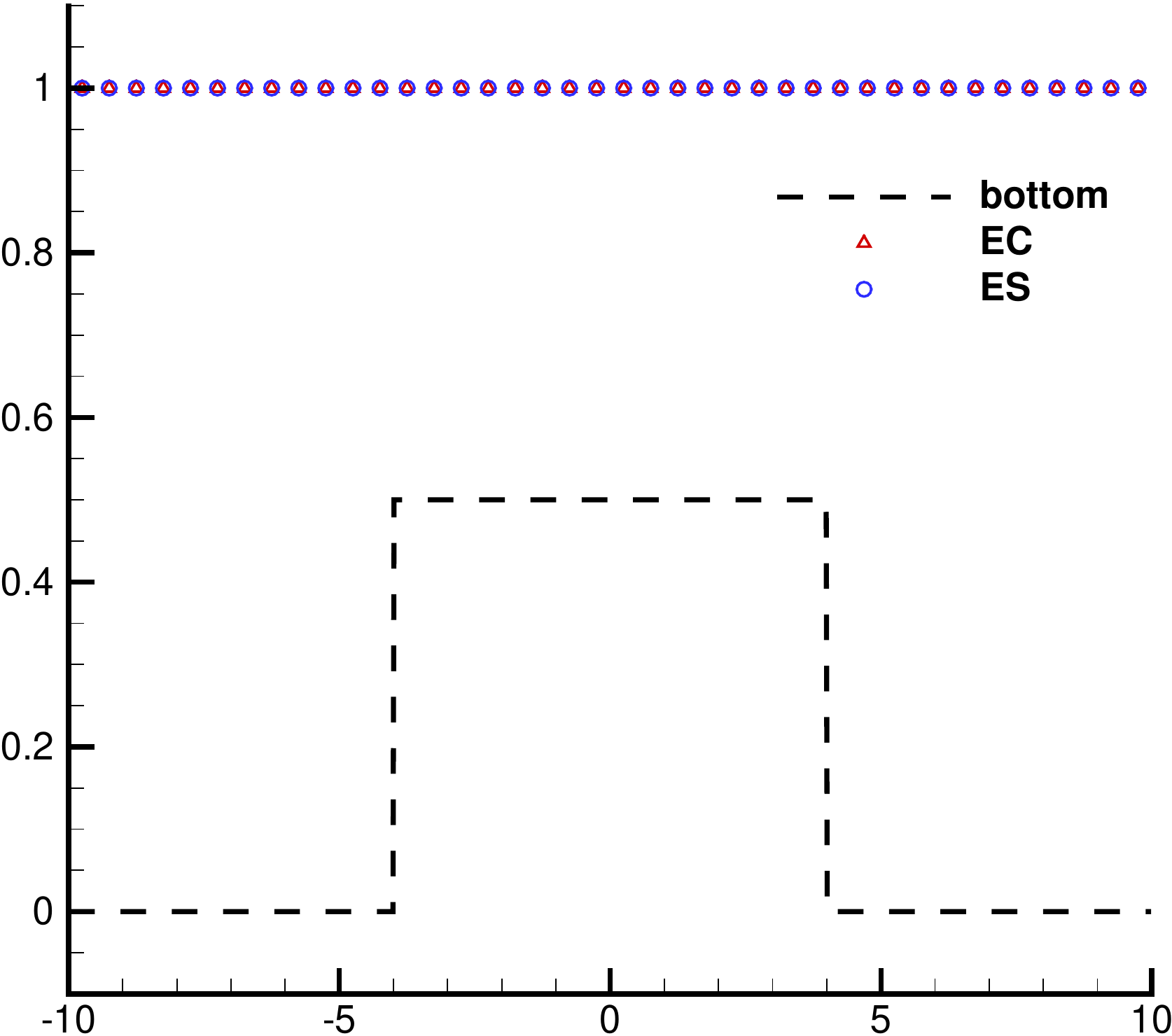}
		\caption{Bottom topography \eqref{eq:1Ddiscontinuous_b}}
	\end{subfigure}
	\caption{Example \ref{ex:1DWB}: The symbols ``$\triangle$'' and ``$\circ$'' denote
		the numerical solutions at $t=10$ obtained by using the EC and the ES schemes with $N_x=40$, respectively.  }
	\label{fig:1DWB}
\end{figure}

\begin{table}[!ht]
	\centering
	
	\begin{tabular}{rc|cc|cc} \hline
		\multirow{2}{*}{} &  & \multicolumn{2}{c|}{EC scheme} & \multicolumn{2}{c}{ES scheme} \\ \cline{3-6}
	                      &  & $\ell^1$ error & $\ell^\infty$ error & $\ell^1$ error & $\ell^\infty$ error \\ \hline
		\multirow{2}{*}{\eqref{eq:1Dsmooth_b}} & $h$   &  9.825e-16 & 2.554e-15 & 1.035e-15 & 2.554e-15 \\
	                             	           & $\vx$ &  5.463e-16 & 1.636e-15 & 5.902e-16 & 1.638e-15 \\ \hline
 \multirow{2}{*}{\eqref{eq:1Ddiscontinuous_b}} & $h$   &  2.484e-16 & 1.776e-15 & 2.262e-16 & 8.882e-16 \\
	                                           & $\vx$ &  2.445e-16 & 2.046e-15 & 2.309e-16 & 1.617e-15 \\
		\hline
	\end{tabular}
	\caption{Example \ref{ex:1DWB}: Errors in $h$ and $\vx$ at $t=10$ for the bottom topography \eqref{eq:1Dsmooth_b} and \eqref{eq:1Ddiscontinuous_b}.}
	\label{tab:1DWB}
\end{table}

\begin{example}[Steady state problem with wavy bottom \cite{Zia2014}]\label{ex:1Dsteady}\rm
	This example is adapted from the problem in \cite{Xu2002} and used to check the dissipative and dispersive errors in the ES scheme.
	The computational domain and the bottom topography are the same as the last problem,  $g=9.812$, and the initial data are
	\begin{equation*}
	(h,\vx,\vy,\Bx,\By)=\begin{cases}
	(1,~1,~0,~0.05,~0),        & x<0, \\
	(1,~1,~0,~0.1,~0.1),       & x>0.
	\end{cases}
	\end{equation*}
\end{example}

The results obtained by using the ES scheme with $N_x=50,100$ are shown in Figure \ref{fig:1Dsteady},
and the reference solutions are obtained by using the ES scheme with $N_x=1000$.
We can see that the accurate solutions can be obtained even with the coarse mesh $N_x=50$.

\begin{figure}[ht!]
	\begin{subfigure}[b]{0.5\textwidth}
		\centering
		\includegraphics[width=1.0\textwidth]{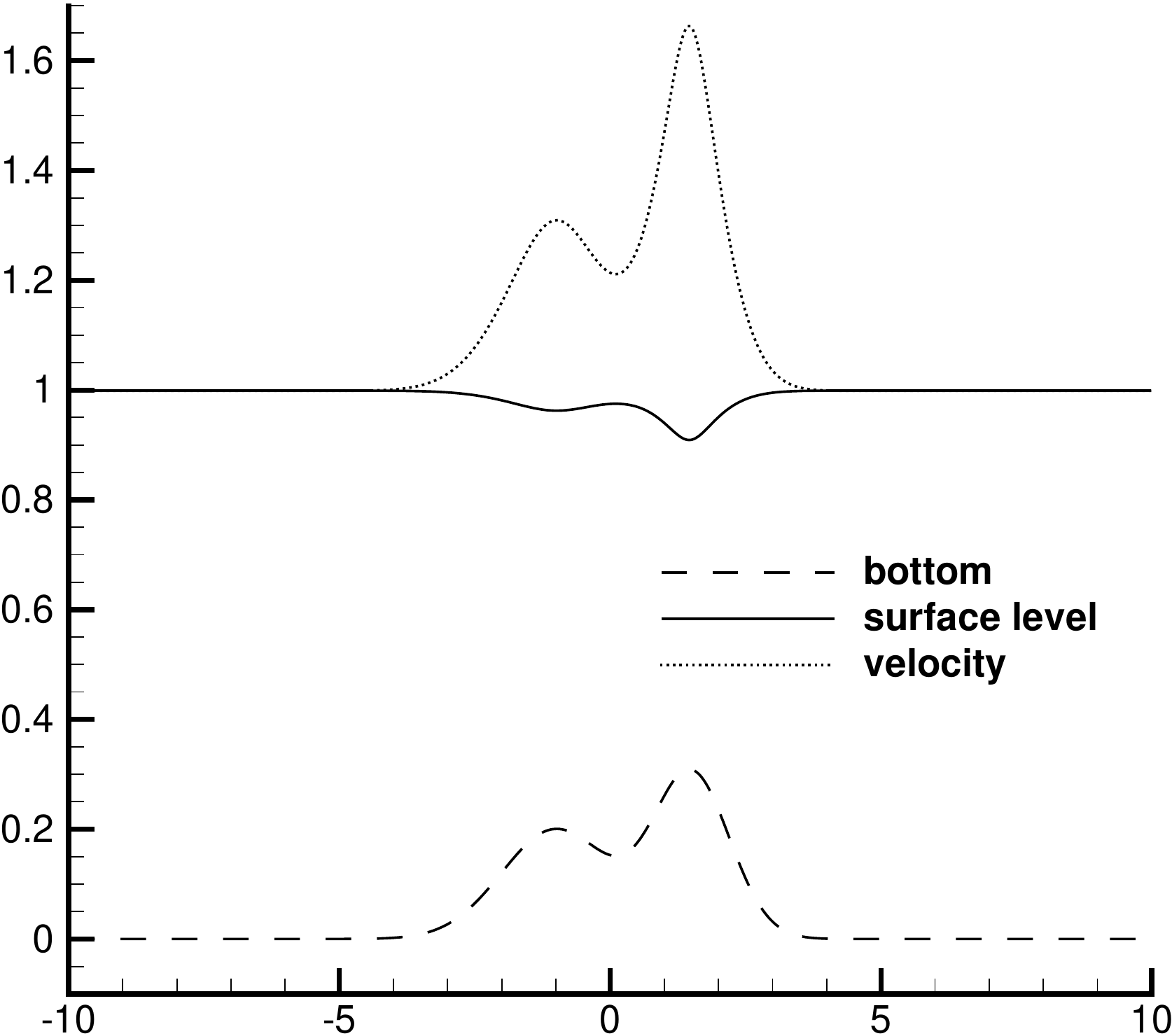}
	\end{subfigure}
	\begin{subfigure}[b]{0.5\textwidth}
		\centering
		\includegraphics[width=1.0\textwidth]{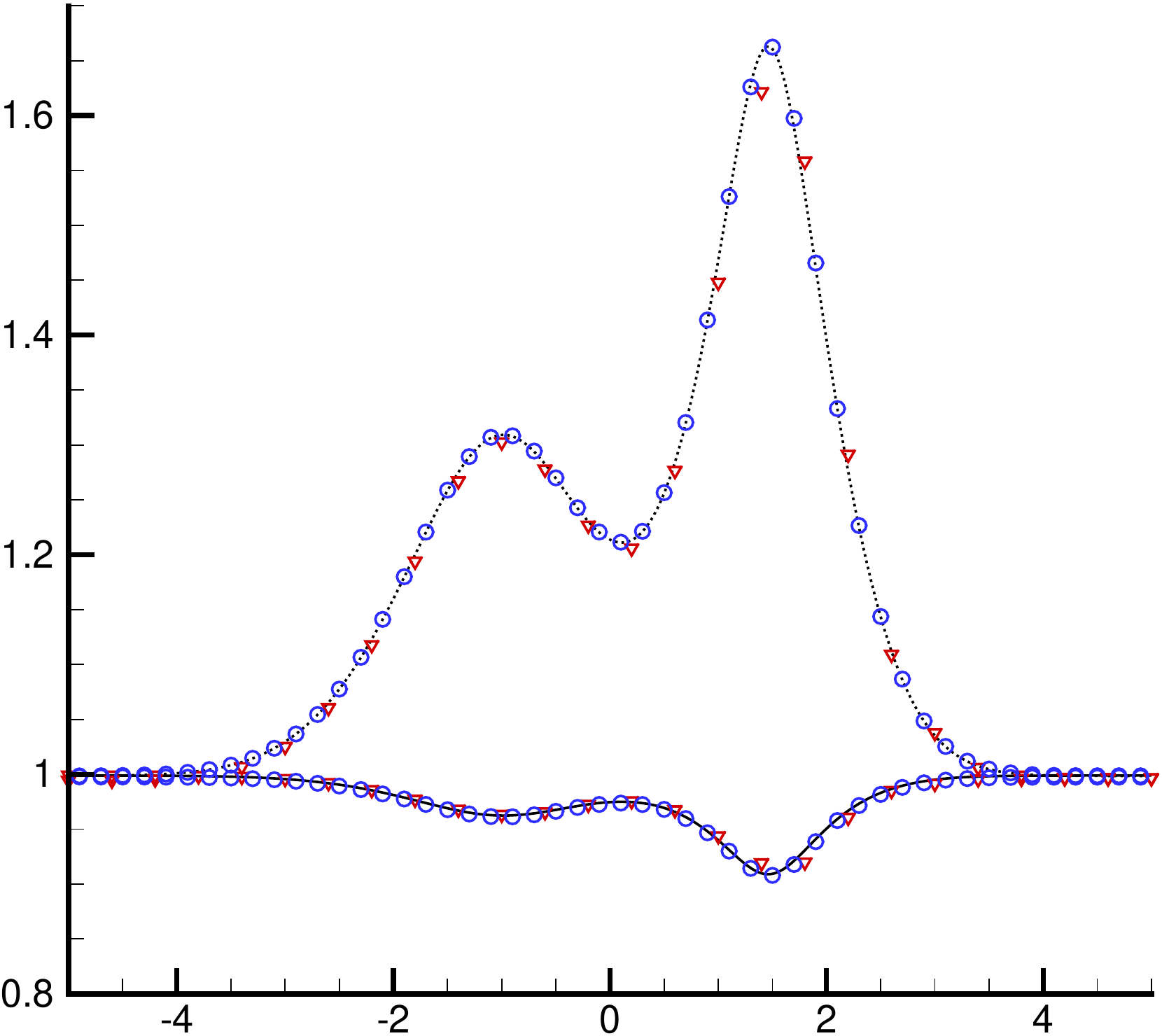}
	\end{subfigure}
	\caption{Example \ref{ex:1Dsteady}: Left: The bottom topography and the reference solutions obtained by using
		the ES scheme with $N_x=1000$. Right:
		 The enlarged view of the numerical solutions  obtained by using the ES schemes with $N_x=50$ (``$\triangledown$'') and $N_x=100$ (``$\circ$''), respectively.  }
	\label{fig:1Dsteady}
\end{figure}

\begin{example}[Small perturbation of a steady state]\label{ex:1Dperturb}\rm
	To examine the ability of capturing small perturbation of a steady state, consider
	two quasi-stationary problems.
		The first problem is considered in \cite{LeVeque1998,Xing2005}.
	The bottom topography consists of one hump
	\begin{equation*}
	  b(x)=\begin{cases}
	  0.25(\cos(10\pi(x-1.5))+1), \quad & \text{if} ~~1.4<x<1.6,\\
	  0,         \quad & \text{otherwise},
	  \end{cases}
	\end{equation*}
	and the initial data are
	\begin{equation*}
	h=\begin{cases}
	1-b(x)+\epsilon,      &\quad \text{if} ~~1.1<x<1.2, \\
	1-b(x),               &\quad \text{otherwise},
	\end{cases}
	\end{equation*}
	with zero velocity and zero magnetic field.
The second quasi-stationary problem takes into account the magnetic field such that $h\Bx=1$.
Those problems are solved until $t=0.2$ with the computational domain $[0,2]$,  $g=9.812$, and $\epsilon=0.2, 0.001$.
\end{example}

The results with zero magnetic field are shown in Figure \ref{fig:1DSW_perturb},
while those with non-zero magnetic field are shown in Figure \ref{fig:1DSWMHD_perturb}.
The solutions obtained by using the ES scheme with $N_x=200$ are
 compared to the reference solutions
obtained by using the ES scheme with a fine mesh of $N_x=3000$.
It can be seen that the structures in the solutions are well captured with no spurious oscillations, and
the results with zero magnetic field are well comparable to those in \cite{Xing2005}.

\begin{figure}[ht!]
	\begin{subfigure}[b]{0.24\textwidth}
		\centering
		\includegraphics[width=1.0\textwidth]{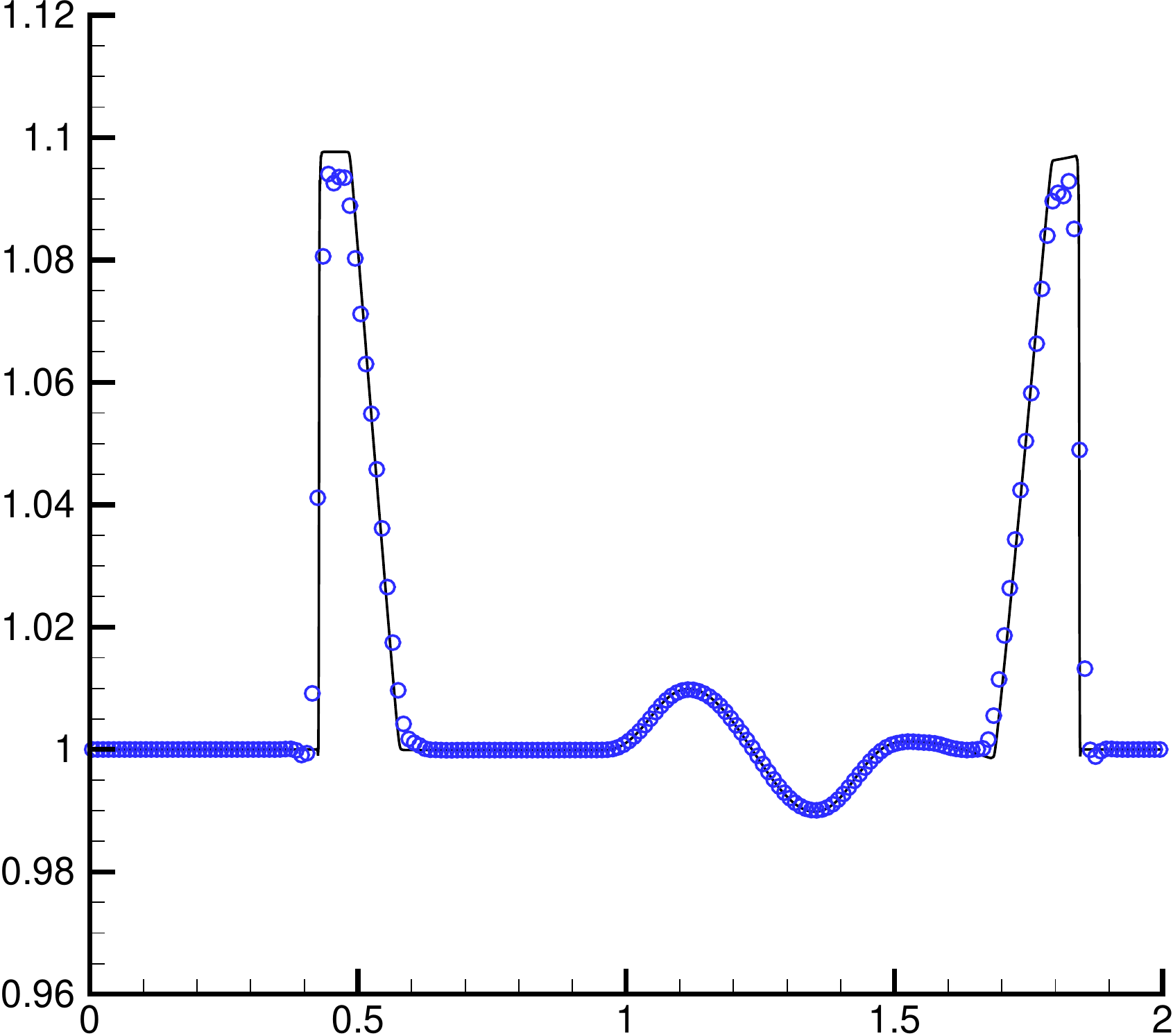}
		\caption{$h+b$ for $\epsilon=0.2$}
	\end{subfigure}
	\begin{subfigure}[b]{0.24\textwidth}
		\centering
		\includegraphics[width=1.0\textwidth]{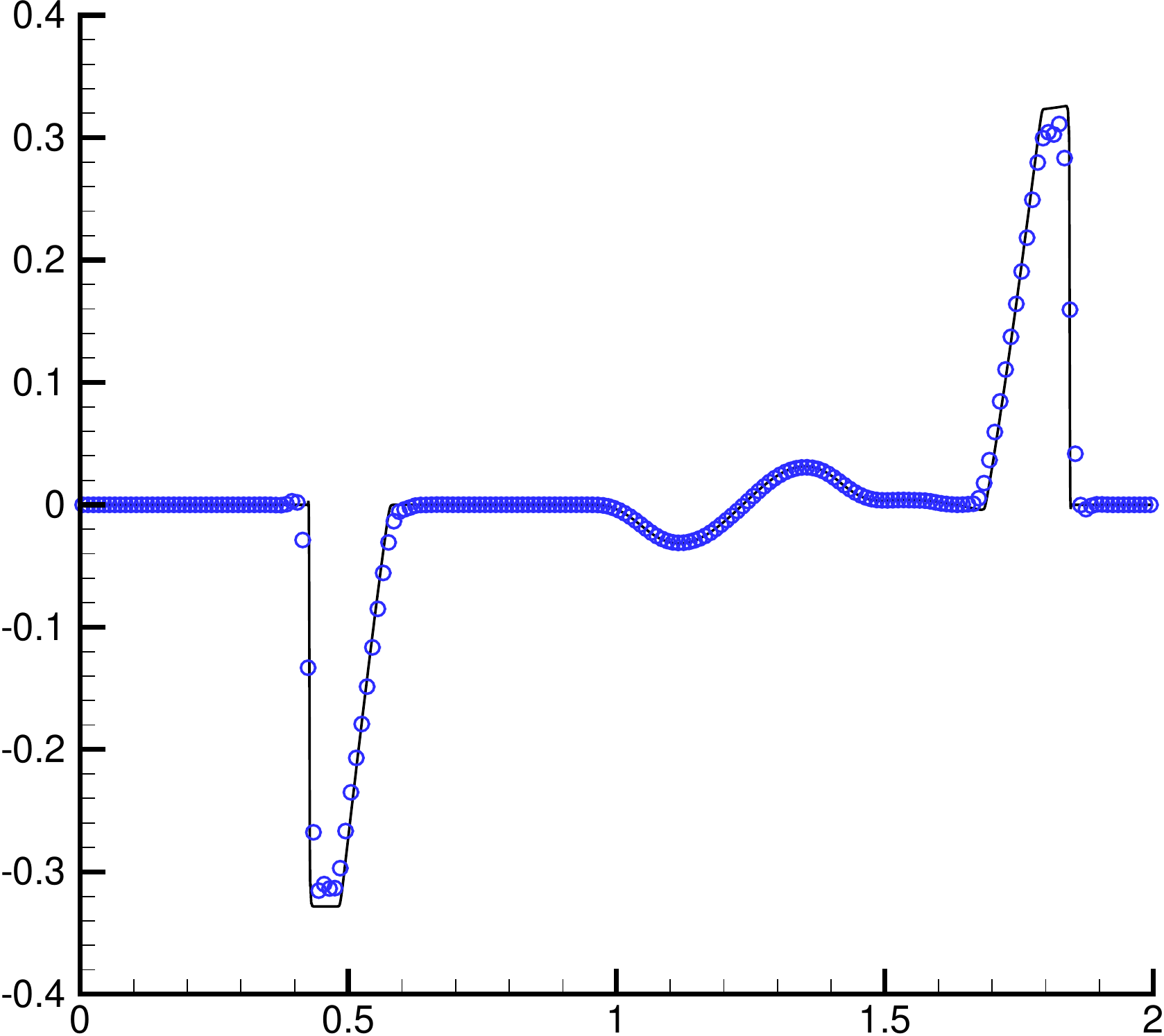}
		\caption{$h\vx$ for $\epsilon=0.2$}
	\end{subfigure}
	\begin{subfigure}[b]{0.24\textwidth}
		\centering
		\includegraphics[width=1.0\textwidth]{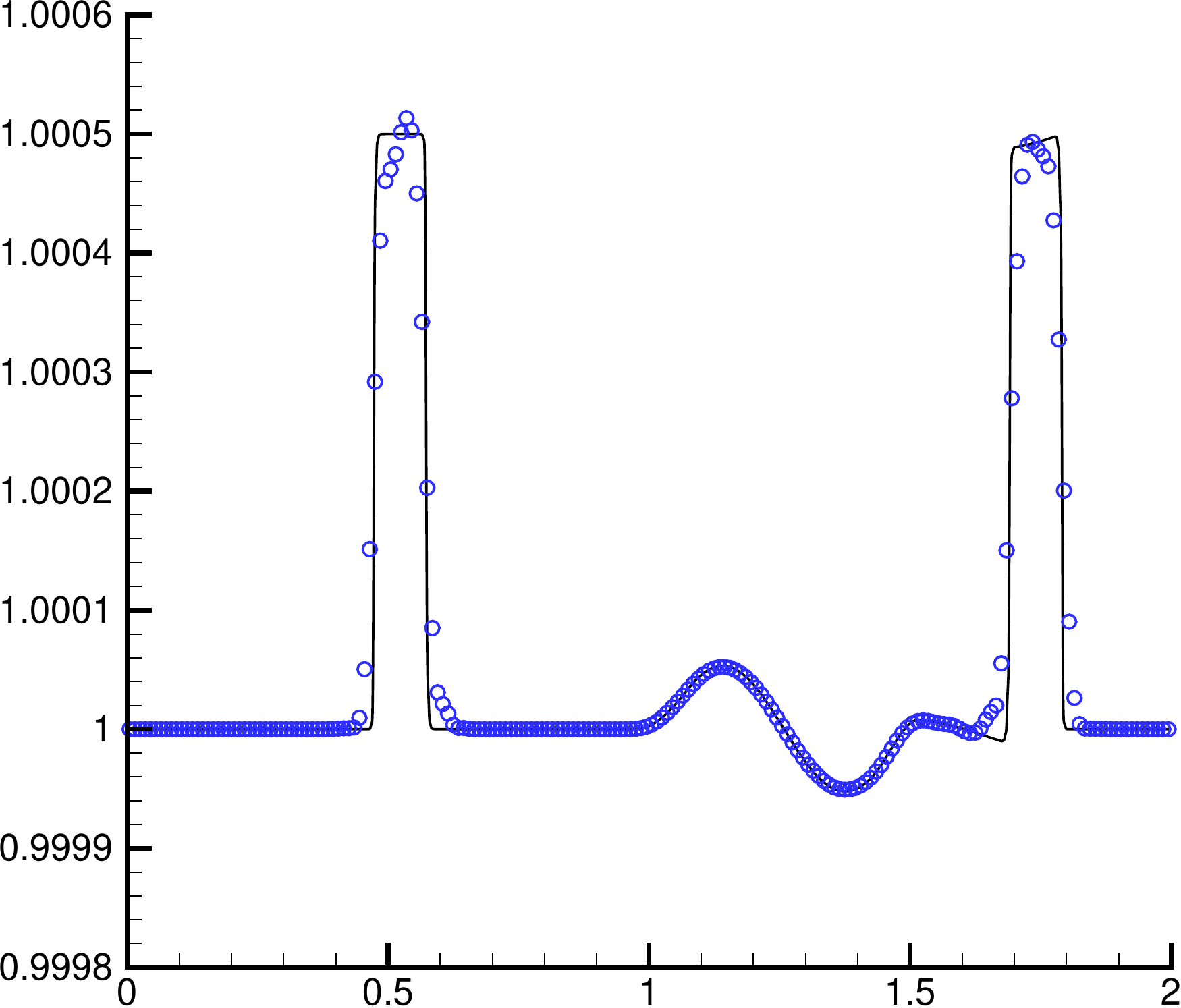}
		\caption{$h+b$ for $\epsilon=0.001$}
	\end{subfigure}
	\begin{subfigure}[b]{0.24\textwidth}
		\centering
		\includegraphics[width=1.0\textwidth]{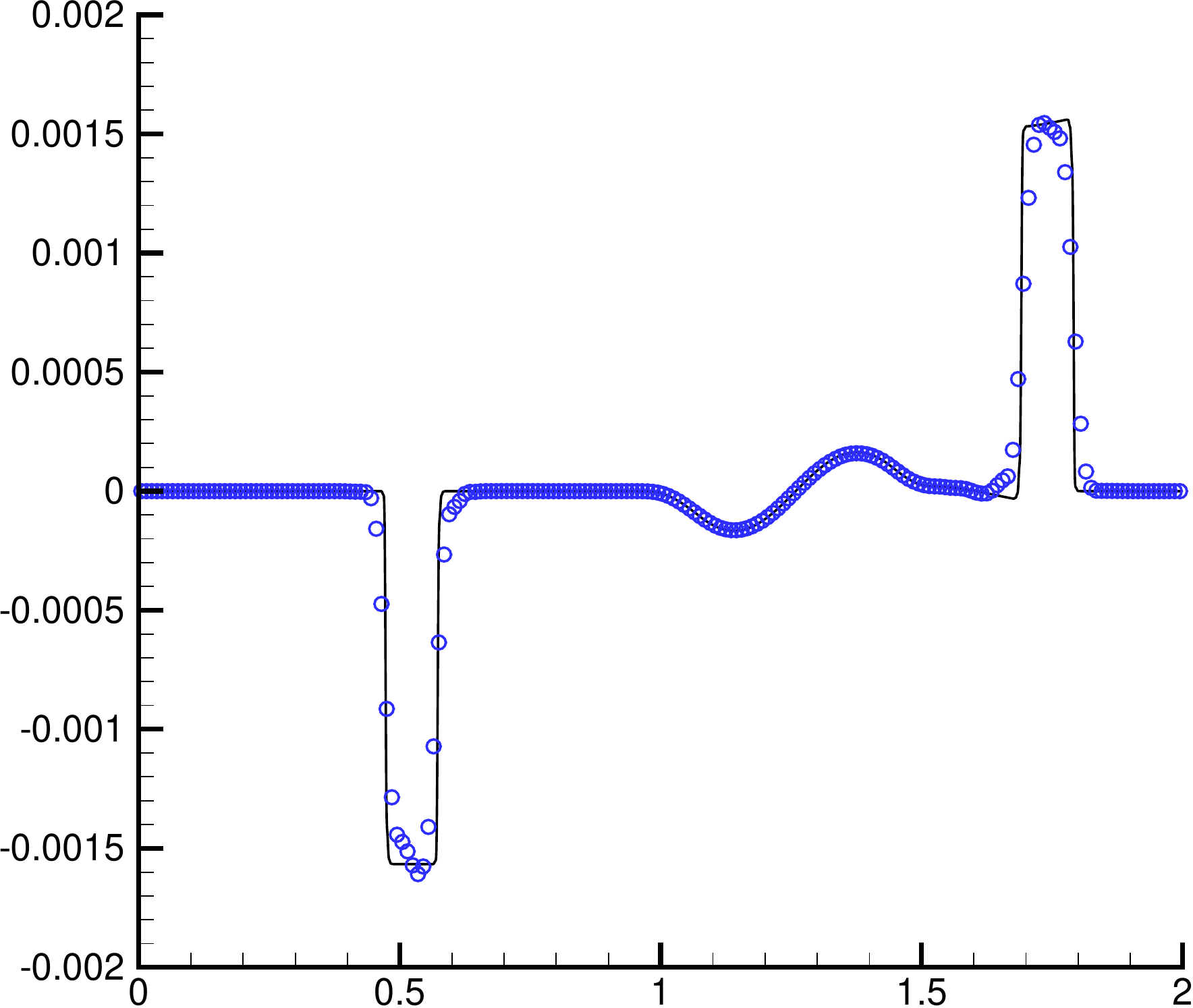}
		\caption{$h\vx$ for $\epsilon=0.001$}
	\end{subfigure}
	\caption{Example \ref{ex:1Dperturb}:  The surface level $h+b$
				and the discharge $h\vx$ obtained by using the ES scheme with $N_x=200$. 
				The solid lines denote the reference solutions obtained by using the ES schemes with $N_x=3000$.
}
	\label{fig:1DSW_perturb}
\end{figure}

\begin{figure}[ht!]
	\begin{subfigure}[b]{0.24\textwidth}
		\centering
		\includegraphics[width=1.0\textwidth]{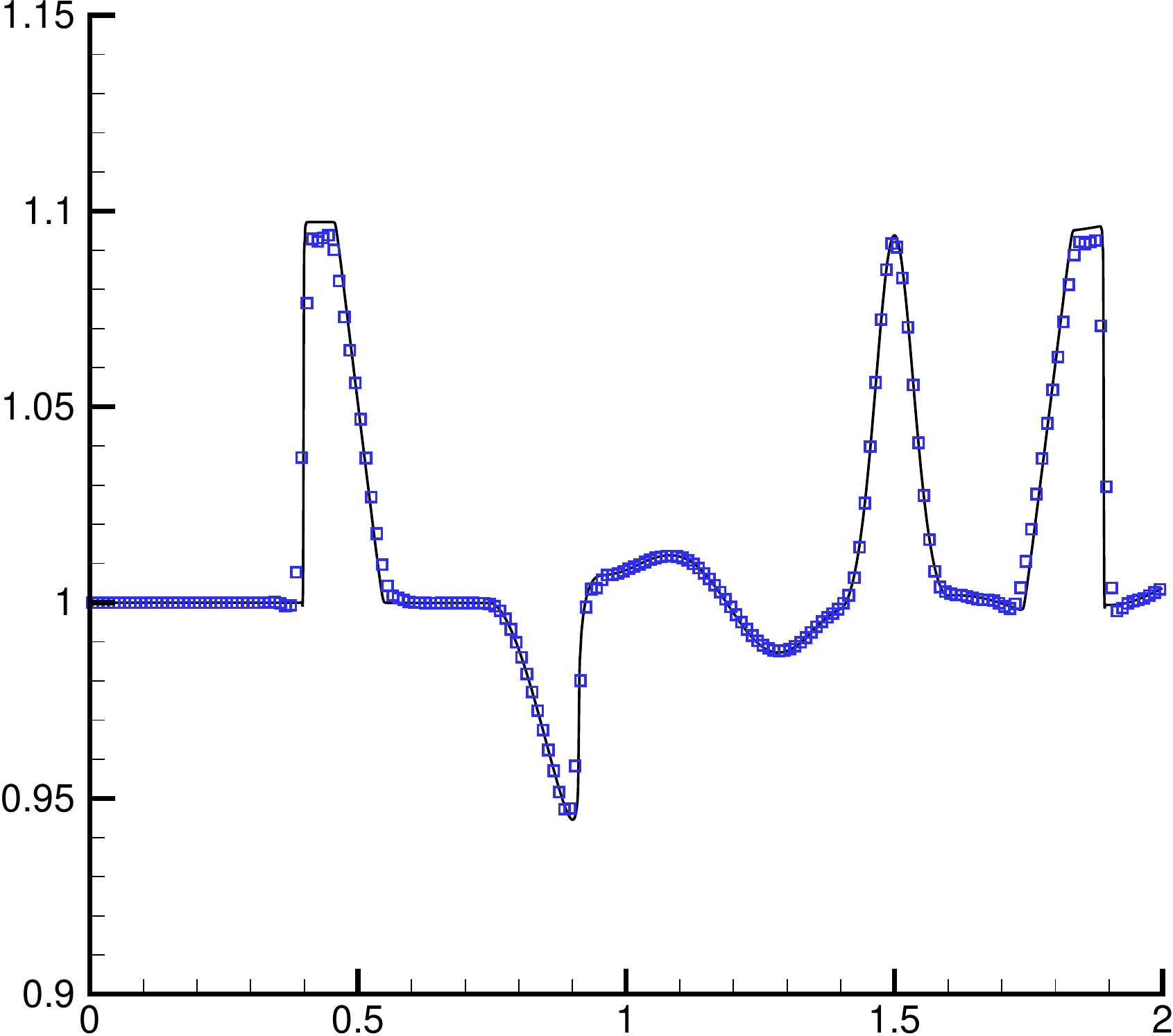}
		\caption{$h+b$ { for $\epsilon=0.2$}}
	\end{subfigure}
	\begin{subfigure}[b]{0.24\textwidth}
		\centering
		\includegraphics[width=1.0\textwidth]{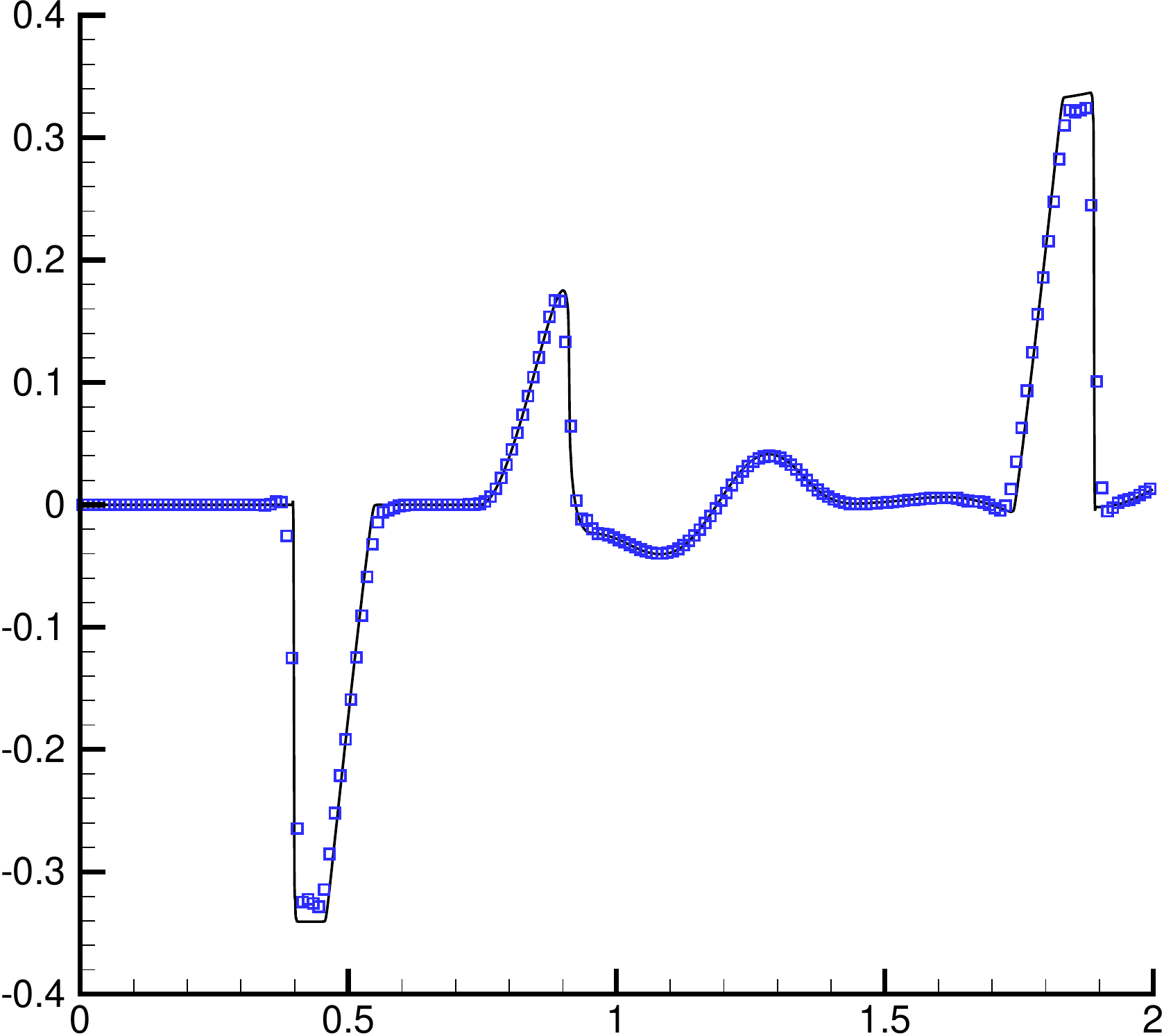}
		\caption{$h\vx$ {for $\epsilon=0.2$}}
	\end{subfigure}
	\begin{subfigure}[b]{0.24\textwidth}
		\centering
		\includegraphics[width=1.0\textwidth]{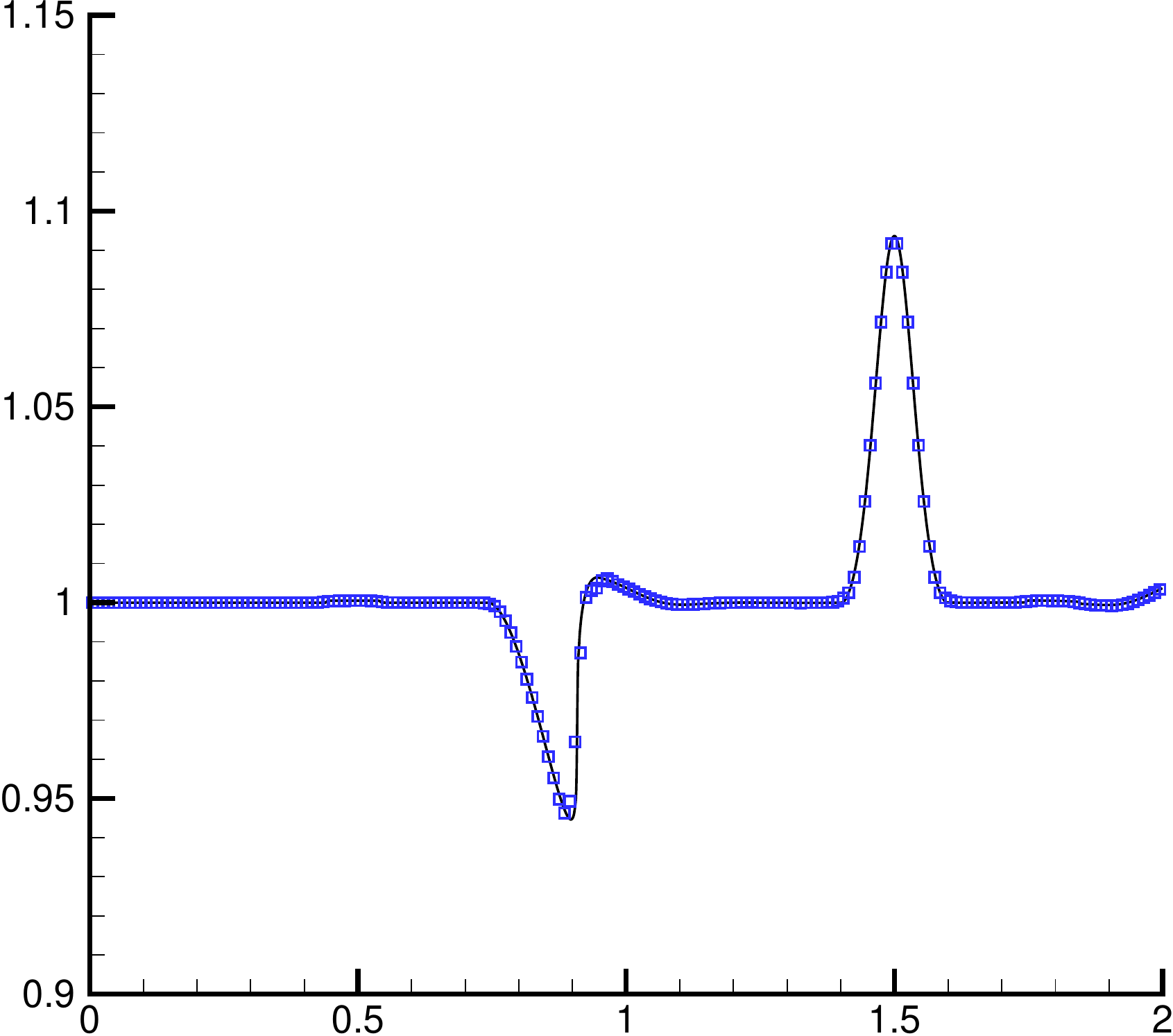}
		\caption{$h+b$ {for $\epsilon=0.001$}}
	\end{subfigure}
	\begin{subfigure}[b]{0.24\textwidth}
		\centering
		\includegraphics[width=1.0\textwidth]{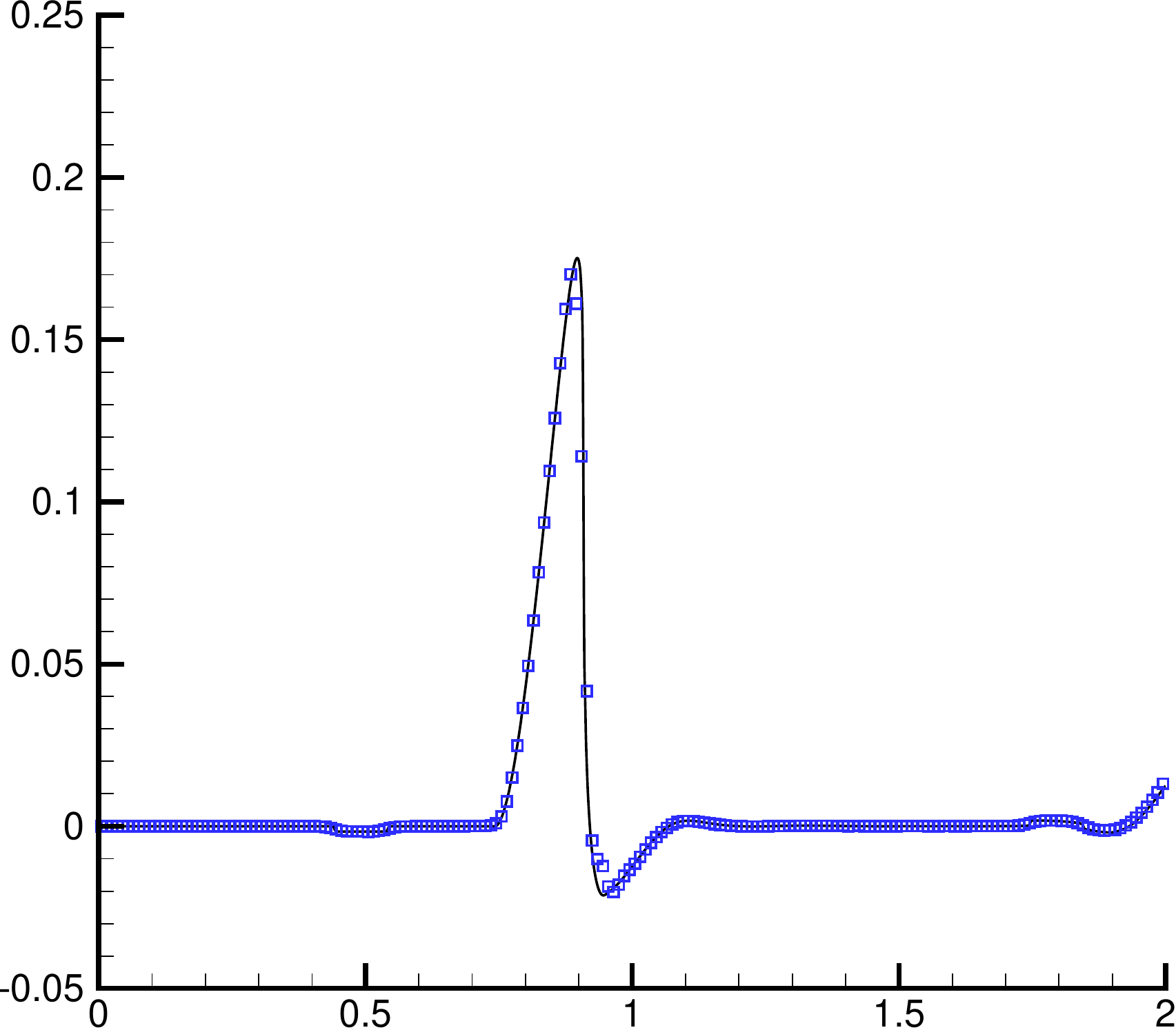}
		\caption{$h\vx$ { for $\epsilon=0.001$}}
	\end{subfigure}
	\caption{Same as Figure \ref{fig:1DSW_perturb} except for $h\Bx=1$. }
	\label{fig:1DSWMHD_perturb}
\end{figure}

\begin{example}[Riemann problem \cite{DeSterck2001}]\label{ex:RP1}\rm
  The initial data   are
  \begin{equation*}
    (h,\vx,\vy,\Bx,\By)=\begin{cases}
      (1,~0,~0,~1,~0),        & x<0, \\
      (2,~0,~0,~0.5,~1),      & x>0.
    \end{cases}
  \end{equation*}
The initial discontinuity will be decomposed into
two magnetogravity waves and two Alfv\'{e}n waves propagating away in two directions as the time increases.
The problem is solved until $t=0.4$ with $g=1$.
\end{example}

Figure \ref{fig:RP1} presents the solutions $h,\vx,\vy,\Bx,\By$ at $t=0.4$ obtained by the ES schemes with $N_x=100$.
One can see that
our numerical solutions are in good agreement with the
reference solutions, and the discontinuities are well captured without obvious oscillations. 

\begin{figure}[ht!]
  \begin{subfigure}[b]{0.19\textwidth}
    \centering
    \includegraphics[width=1.0\textwidth]{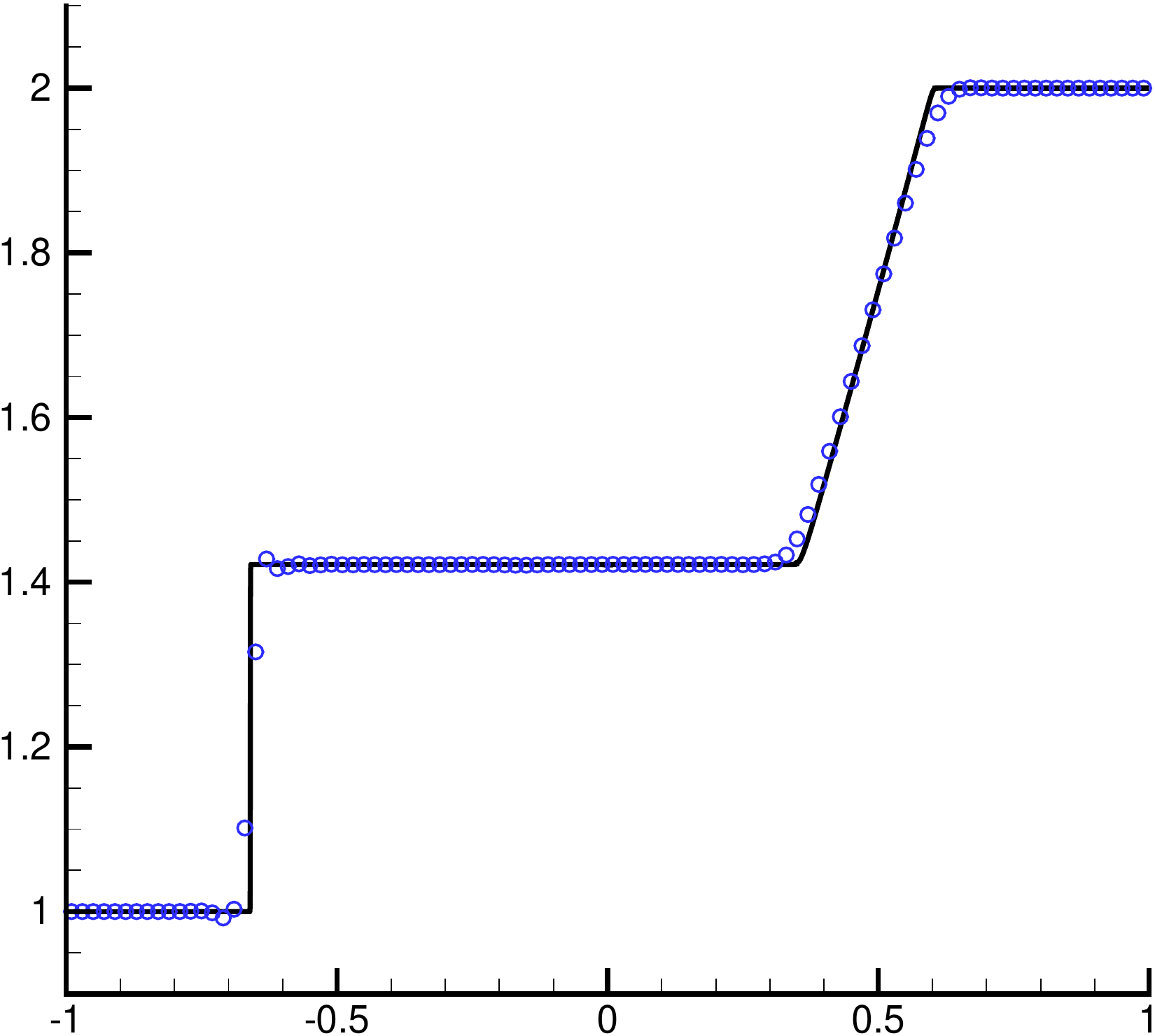}
    \caption{$h$}
  \end{subfigure}
  \begin{subfigure}[b]{0.19\textwidth}
    \centering
    \includegraphics[width=1.0\textwidth]{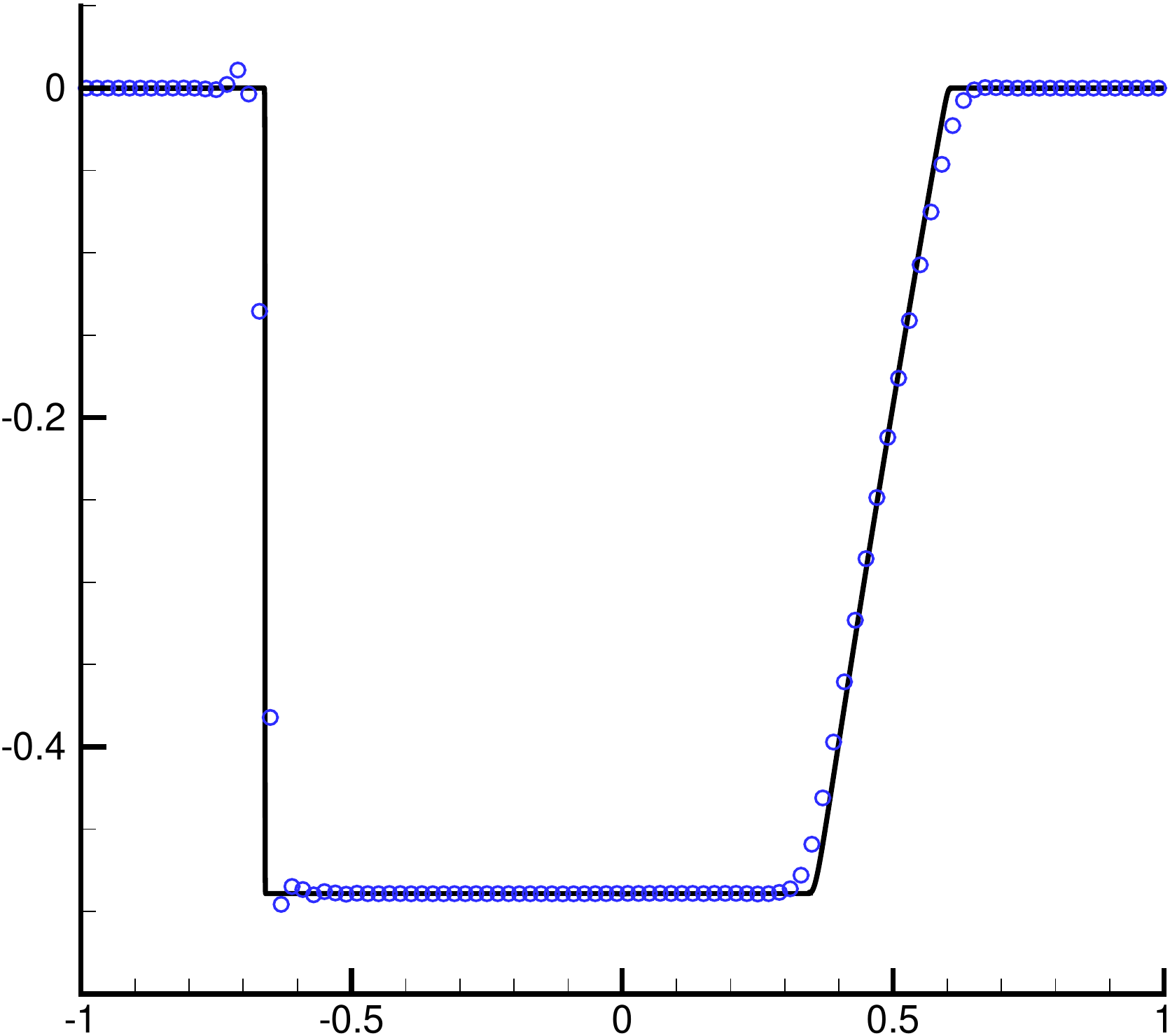}
    \caption{$\vx$}
  \end{subfigure}
  \begin{subfigure}[b]{0.19\textwidth}
    \centering
    \includegraphics[width=1.0\textwidth]{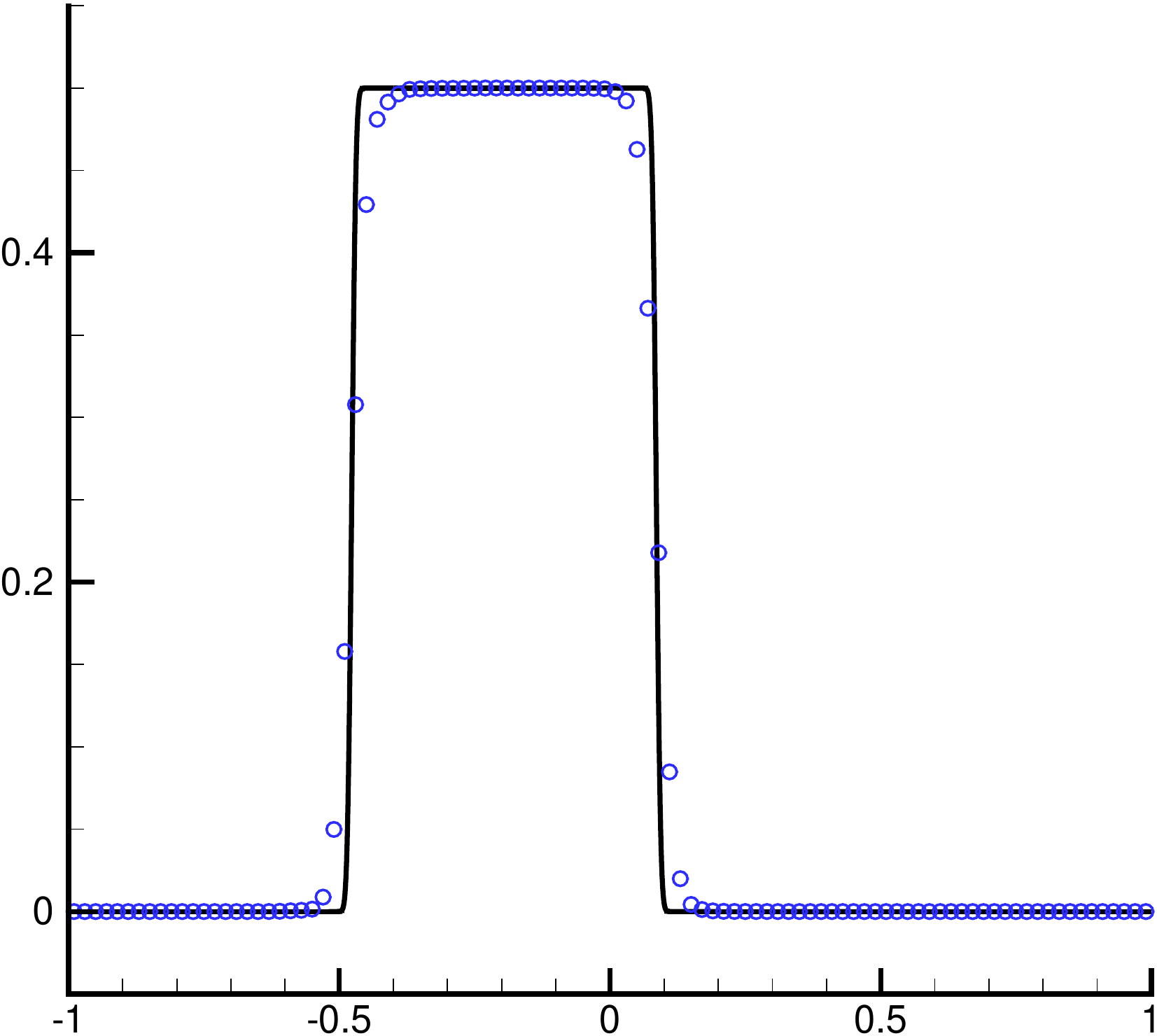}
    \caption{$\vy$}
  \end{subfigure}
\begin{subfigure}[b]{0.19\textwidth}
	\centering
	\includegraphics[width=1.0\textwidth]{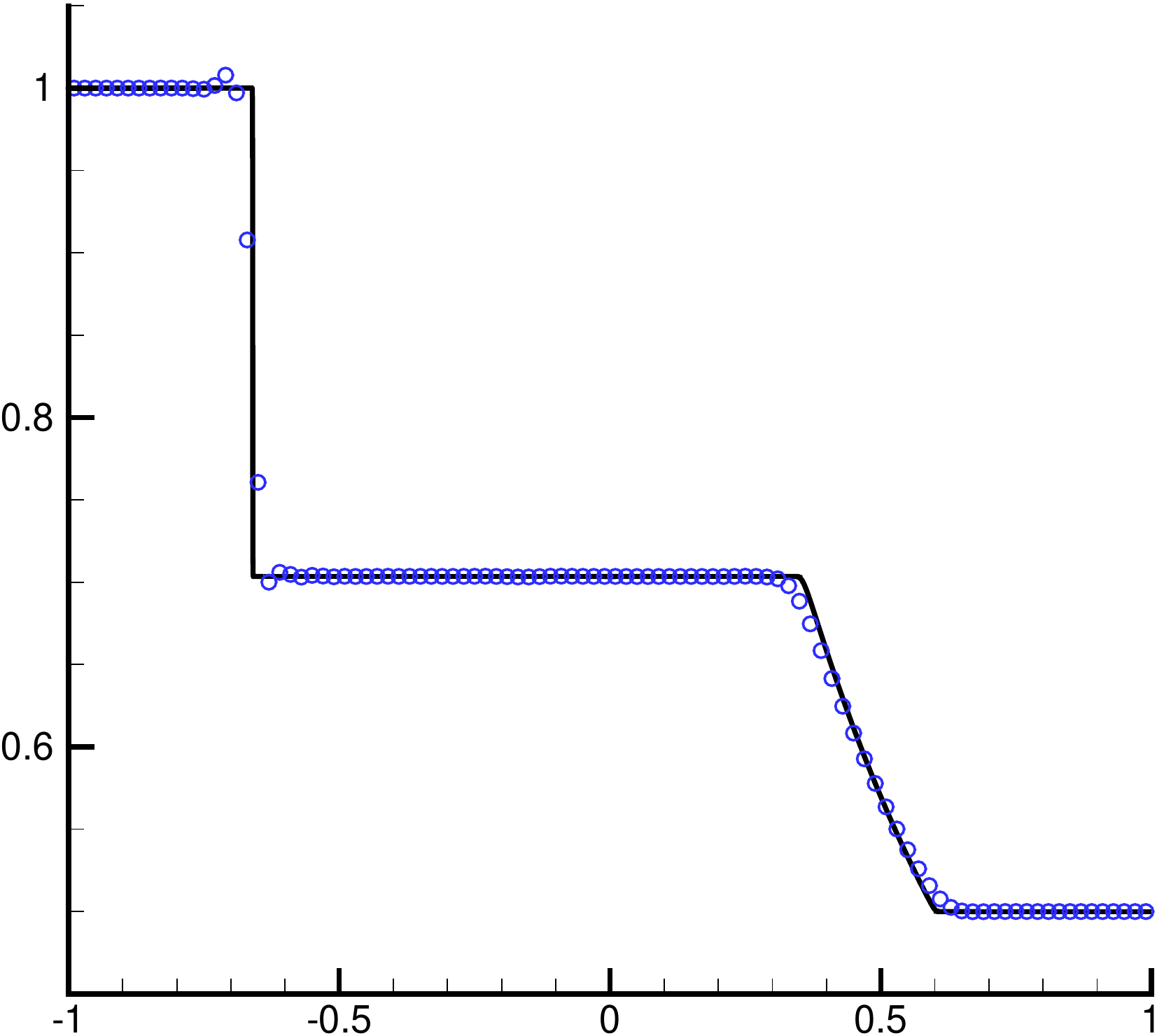}
	\caption{$\Bx$}
\end{subfigure}
\begin{subfigure}[b]{0.19\textwidth}
	\centering
	\includegraphics[width=1.0\textwidth]{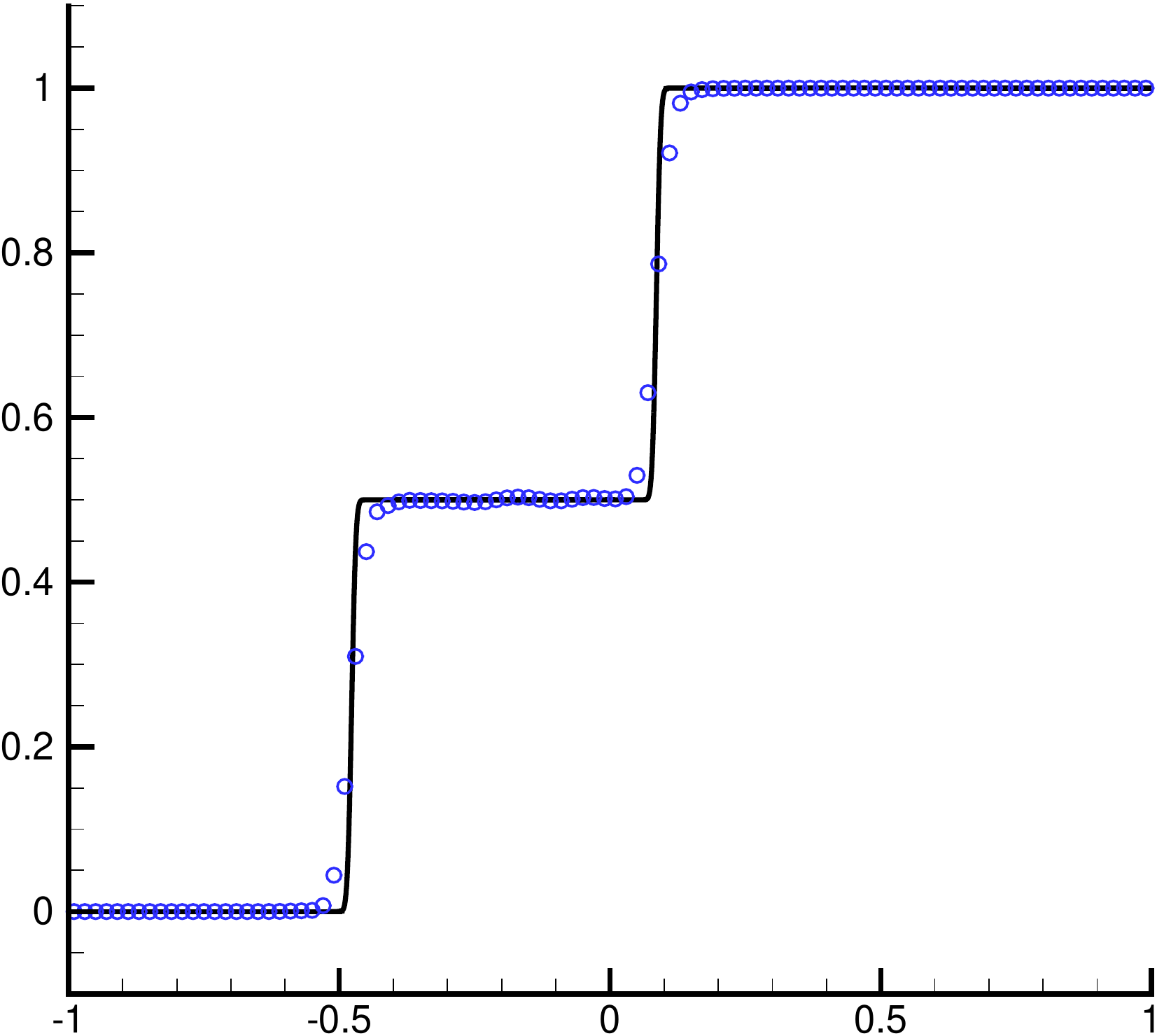}
	\caption{$\By$}
\end{subfigure}
  \caption{Example \ref{ex:RP1}: The solutions at $t=0.4$ (``$\circ$'') are  obtained by the ES schemes
with $N_x=100$. The solid lines denote
the reference solutions obtained by using the Lax-Friedrichs scheme with a fine mesh of $N_x=20000$.  }
  \label{fig:RP1}
\end{figure}

\subsection{Two-dimensional case}

\begin{example}[Vortex]\label{ex:vortex}\rm
This genuine 2D vortex problem is designed to test the accuracy and the positivity-preserving  property of our schemes.
  With the aid of the SWMHD equations in the polar coordinates given
  in \ref{app},
  a steady vortex is constructed as follows
  \begin{align*}
  &h'=h_{\text{max}}-\left(v_{\text{max}}^2-B_{\text{max}}^2\right)e^{1-r^2}/(2g),\\
  &(\vx', \vy')=v_{\text{max}} e^{0.5(1-r^2)}(-y, x),\\
  &(\Bx', \By')=B_{\text{max}} e^{0.5(1-r^2)}(-y, x),
  \end{align*}
  with $v_{\text{max}}=0.2,B_{\text{max}}=0.1,r=\sqrt{x^2+y^2}$.
 Using the Galilean transformation $x'= x-t,  y' =y-t, t'=t$
 can give a time-dependent exact solution
  \begin{align*}
    &h(x,y,t)=h'(x-t,y-t,t),\quad (\vx, \vy)(x,y,t)=(1,1) + (\vx', \vy')(x-t,y-t,t),\\
    &(\Bx, \By)(x,y,t)=(\Bx', \By')(x-t,y-t,t),
  \end{align*}
  which describes a vortex moving with a constant speed $(1,1)$.

\end{example}

  The computational domain is $[-8,8]^2$ with periodic boundary conditions,  $g=1$, $h_{\text{max}}=1$, and the output time is $t=16$ so that the vortex
travels and returns to the original position after a period.
Table \ref{tab:vortex} lists the errors in $h$ and
corresponding orders of convergence. The results show that
our EC   and ES schemes achieve the optimal convergence order.
Figure \ref{fig:vortex} plots the contours of $h$ and the
magnitude of the magnetic field $|\bm{B}|$ with $40$ equally spaced contour lines.
It can be seen that our schemes can preserve the shape of the vortex well after a whole period.

To check the positivity-preserving property, let us do a test with
$$h_{\text {max}}=10^{-6}+\left(v_{\text{max}}^2-B_{\text{max}}^2\right)e/(2g),$$
which implies that the lowest water height is $10^{-6}$. The errors and orders of convergence
obtained by using the ES scheme with or without the positivity-preserving (PP) limiter are listed in Table \ref{tab:vortex_PP}.
One can see that the ES scheme fails with $N_x=N_y= 20,40$ and without the positivity-preserving limiter,
and the positivity-preserving limiter does not destroy the high-order accuracy.

\begin{table}[!ht]
  \centering
  \begin{tabular}{r|cc|cc|cc|cc} \hline
  	\multirow{2}{*}{$N_x=N_y$} & \multicolumn{4}{c|}{EC scheme} & \multicolumn{4}{c}{ES scheme} \\ \cline{2-9}
  	& $\ell^1$ error & order & $\ell^\infty$ error & order & $\ell^1$ error & order & $\ell^\infty$ error & order \\ \hline
   20 & 4.079e-05 &  -   & 1.787e-02 &  -   & 3.756e-05 &  -   & 2.446e-02 &  -   \\
   40 & 7.025e-07 & 5.86 & 1.543e-03 & 3.53 & 4.082e-06 & 3.20 & 1.012e-02 & 1.27 \\
   80 & 6.401e-09 & 6.78 & 3.028e-05 & 5.67 & 1.015e-07 & 5.33 & 7.531e-04 & 3.75 \\
  160 & 5.244e-11 & 6.93 & 4.994e-07 & 5.92 & 1.582e-09 & 6.00 & 2.340e-05 & 5.01 \\
  320 & 4.143e-13 & 6.98 & 7.904e-09 & 5.98 & 2.453e-11 & 6.01 & 7.205e-07 & 5.02 \\
  	\hline
  \end{tabular}
  \caption{Example \ref{ex:vortex}: Errors and orders of convergence in $h$ at $t=16$ obtained by using the
  	EC and ES schemes, $h_{\text{max}}=1$.}
  \label{tab:vortex}
\end{table}

\begin{figure}[ht!]
  \begin{subfigure}[b]{0.5\textwidth}
    \centering
    \includegraphics[width=1.0\textwidth]{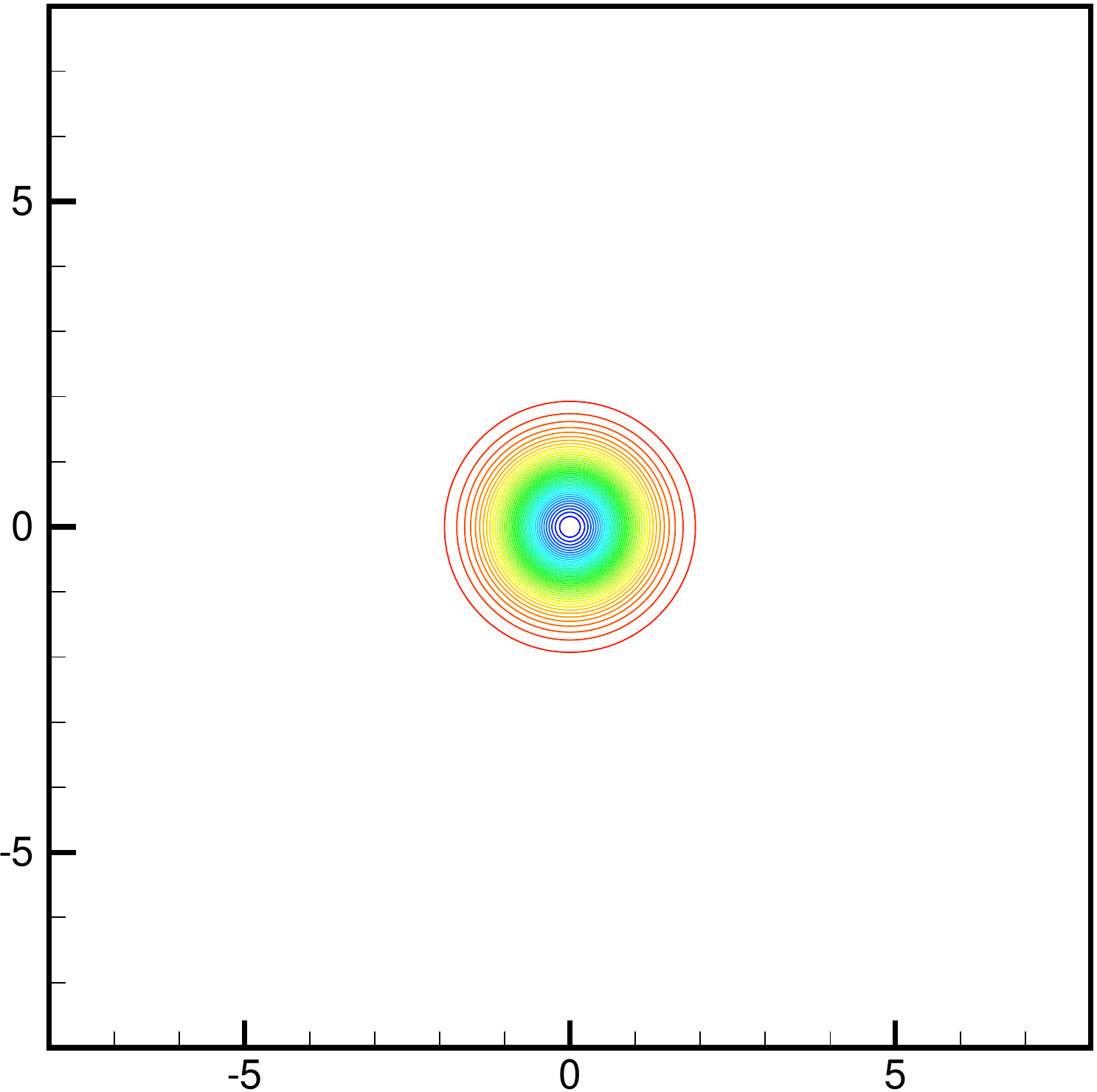}
  \end{subfigure}
  \begin{subfigure}[b]{0.5\textwidth}
    \centering
    \includegraphics[width=1.0\textwidth]{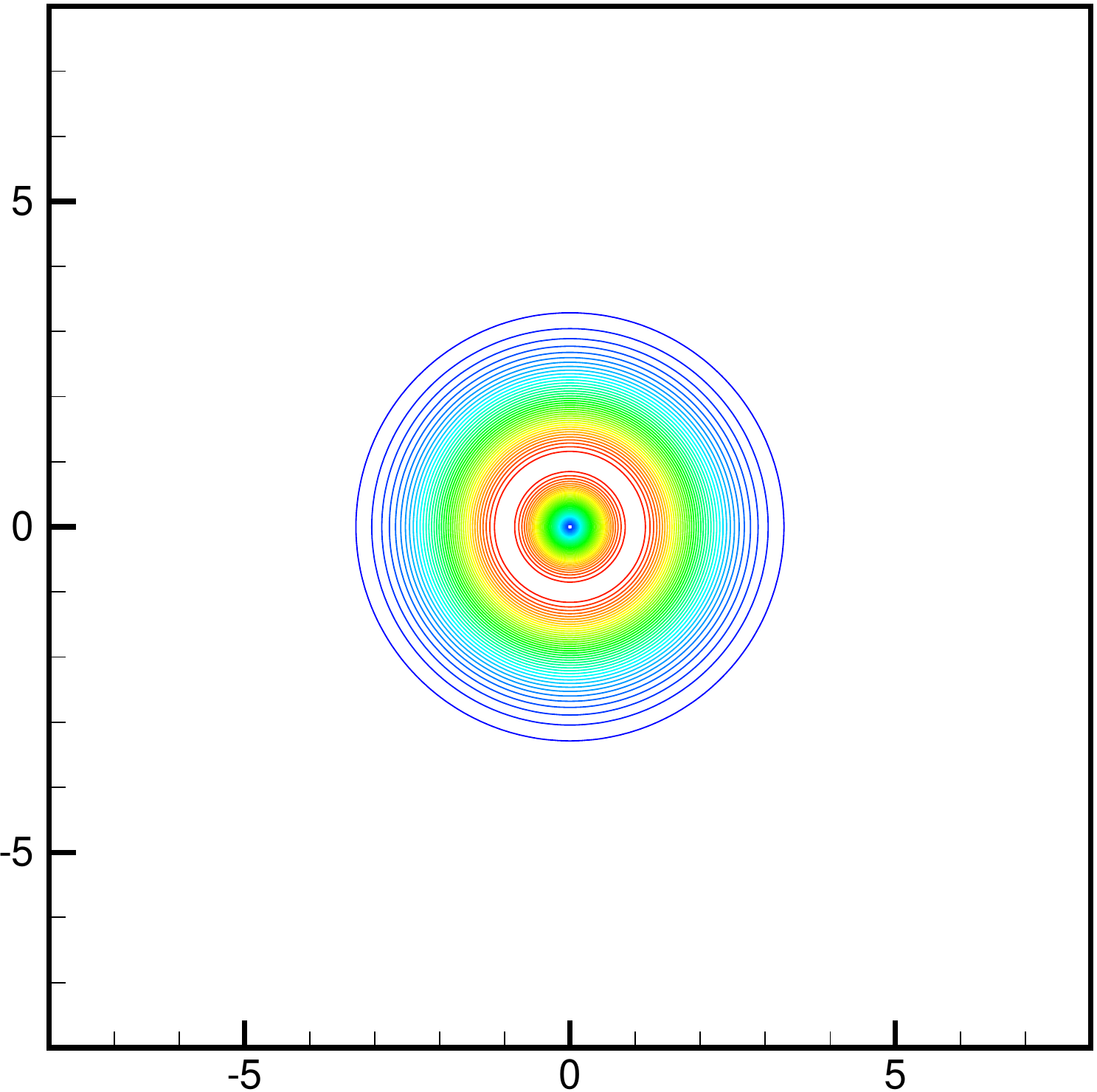}
  \end{subfigure}
  \caption{Example \ref{ex:vortex}: The contours of $h$ (left) and $|\bm{B}|$ (right) at $t=16$
  	obtained by using the ES scheme with $N_x=N_y=320$ and $40$ equally spaced contour lines.}
  \label{fig:vortex}
\end{figure}

\begin{table}[!ht]
	\centering
	\begin{tabular}{r|cc|cc|cc|cc} \hline
		\multirow{2}{*}{$N_x=N_y$} & \multicolumn{4}{c|}{Without PP limiter} & \multicolumn{4}{c}{With PP limiter} \\ \cline{2-9}
		& $\ell^1$ error & order & $\ell^\infty$ error & order & $\ell^1$ error & order & $\ell^\infty$ error & order \\ \hline
   20 &     -     &  -   &     -     &  -   & 4.007e-05 &  -   & 2.680e-02 &  -   \\
   40 &     -     &  -   &     -     &  -   & 3.426e-06 & 3.55 & 7.087e-03 & 1.92 \\
   80 & 1.109e-07 &  -   & 2.284e-03 &  -   & 1.209e-07 & 4.82 & 2.657e-03 & 1.42 \\
  160 & 6.357e-09 & 4.25 & 1.000e-03 & 1.41 & 5.520e-09 & 4.45 & 9.224e-04 & 1.53 \\
  320 & 3.094e-11 & 7.68 & 1.703e-05 & 5.88 & 3.089e-11 & 7.48 & 1.703e-05 & 5.76 \\
		\hline
	\end{tabular}
	\caption{Example \ref{ex:vortex}: Errors and orders of convergence in $h$  at $t=16$ obtained by using the ES schemes, $h_{\text {max}}=10^{-6}+\left(v_{\text{max}}^2-B_{\text{max}}^2\right)e/(2g)$.}
	\label{tab:vortex_PP}
\end{table}


\begin{example}[Well-balanced test \cite{LeVeque1998}]\label{ex:2DWB}\rm
	It is used to validate the well-balanced property of our EC and ES schemes.
	The bottom topography is taken as
	\begin{equation}\label{eq:2Dsmooth_b}
	  b(x,y)=0.8\exp(-5(x-0.9)^2-50(y-0.5)^2),
	\end{equation}
	or
	\begin{equation}\label{eq:2Ddiscontinuous_b}
	  b(x,y)=0.5\chi_{[0.5,1.5]\times [0.25,0.75]}.
	\end{equation}
	The computational domain is $[0,2]\times [0,1]$,  $g=1$, and
	the initial data are $h(x,y)=1-b(x,y)$ with zero velocity and magnetic field.
	
\end{example}
The problem is solved until $t=1$ with $N_x=N_y=40$.
The errors in $h,\vx,\vy$ are listed in Table \ref{tab:2DWB}. Similar to Example \ref{ex:1DWB},
one can see that the well-balanced property of the 2D schemes has been
verified in the sense that the errors are at the level of round-off errors for the double precision.

\begin{table}[!ht]
	\centering
	\begin{tabular}{rc|cc|cc} \hline
		\multirow{2}{*}{} &  & \multicolumn{2}{c|}{EC scheme} & \multicolumn{2}{c}{ES scheme} \\ \cline{3-6}
		&  & $\ell^1$ error & $\ell^\infty$ error & $\ell^1$ error & $\ell^\infty$ error \\ \hline
	\multirow{3}{*}{\eqref{eq:2Dsmooth_b}} 	  & $h$   &  4.943e-15 & 3.886e-14 & 2.293e-15 & 1.077e-14 \\
	                                          & $\vx$ &  5.091e-15 & 4.091e-14 & 2.577e-15 & 1.010e-14 \\
	                                          & $\vy$ &  3.766e-15 & 3.168e-14 & 1.887e-15 & 9.742e-15 \\ \hline
\multirow{3}{*}{\eqref{eq:2Ddiscontinuous_b}} & $h$   &  8.437e-16 & 5.329e-15 & 8.102e-16 & 4.663e-15 \\
	                                          & $\vx$ &  5.720e-16 & 4.585e-15 & 5.731e-16 & 3.792e-15 \\
	                                          & $\vy$ &  1.385e-15 & 7.659e-15 & 1.367e-15 & 6.306e-15 \\
		\hline
	\end{tabular}
	\caption{Example \ref{ex:2DWB}: Errors in $h,\vx,\vy$ at $t=1$ for the bottom topography \eqref{eq:2Dsmooth_b} and \eqref{eq:2Ddiscontinuous_b}.}
	\label{tab:2DWB}
\end{table}

\begin{example}[Small perturbation of a steady state (without magnetic field) \cite{LeVeque1998}]\label{ex:2Dperturb}\rm
	This problem is used to check the capability of the ES schemes for the perturbation of the steady state.
	The computational domain and the bottom topography are the same as the last test,
	and the initial data are
	\begin{equation*}
	h=\begin{cases}
	1.01-b(x),      & \text{if} ~~0.05<x<0.15, \\
	1-b(x),         & \text{otherwise},
	\end{cases}
	\end{equation*}
	with zero velocity and magnetic field.
    Outflow boundary conditions are used and  the problem is solved until $t=0.6$ with $g=9.812$.
\end{example}

Figure \ref{fig:2DSW_perturb} shows the contours of the surface level $h+b$ (40
equally spaced contour lines) at $t = 0.12,0.24,0.36,0.48,0.6$ obtained by using the ES scheme with
$N_x= 600,N_y = 300$,
which describe a right-going disturbance
as it propagates past the hump.
It can be seen that the complex small features are well captured without any spurious oscillations, and
the results are comparable to those in \cite{Xing2005}.

\begin{figure}[!htb]
		\includegraphics[width=0.5\textwidth]{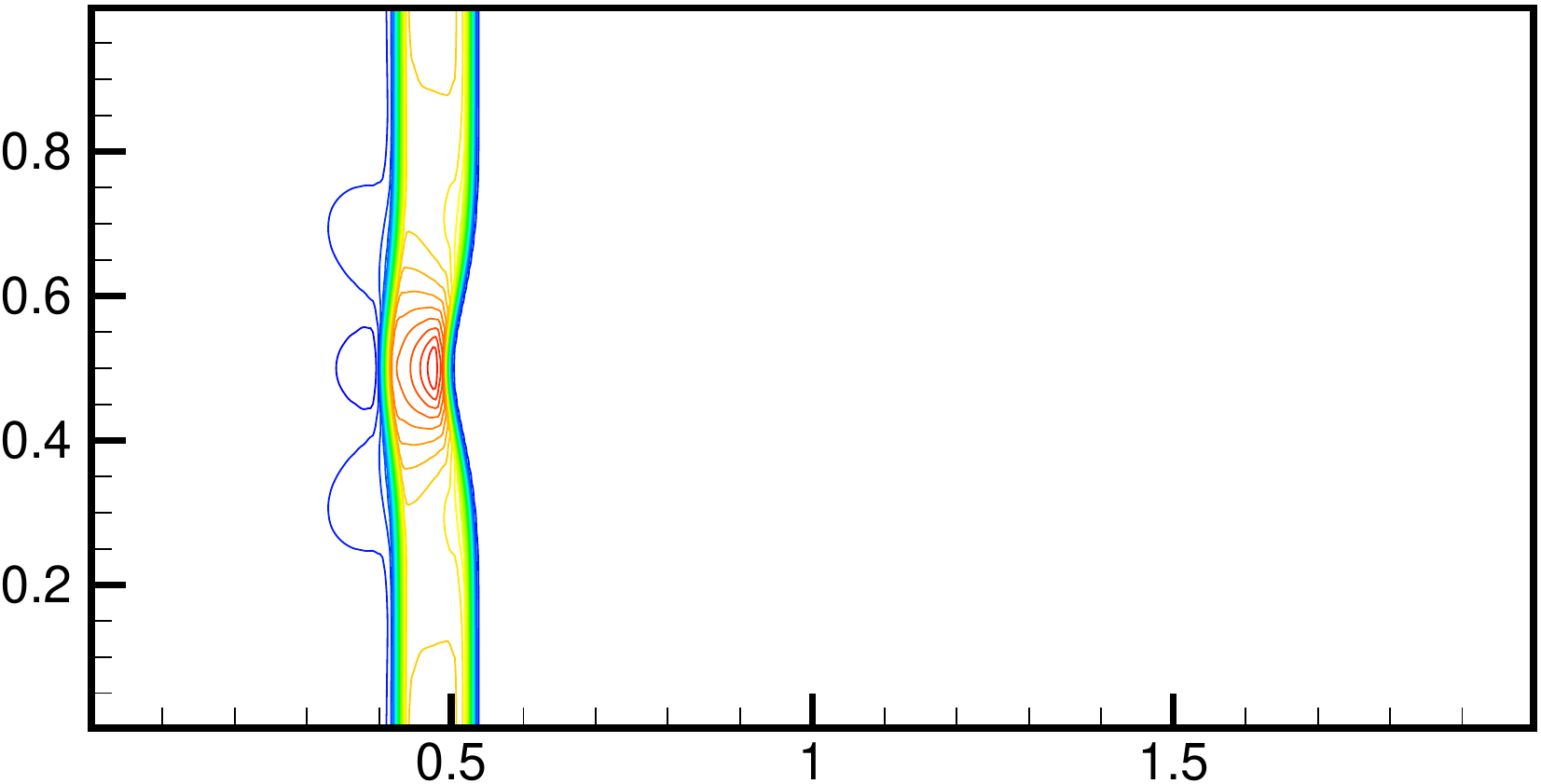}
		\includegraphics[width=0.5\textwidth]{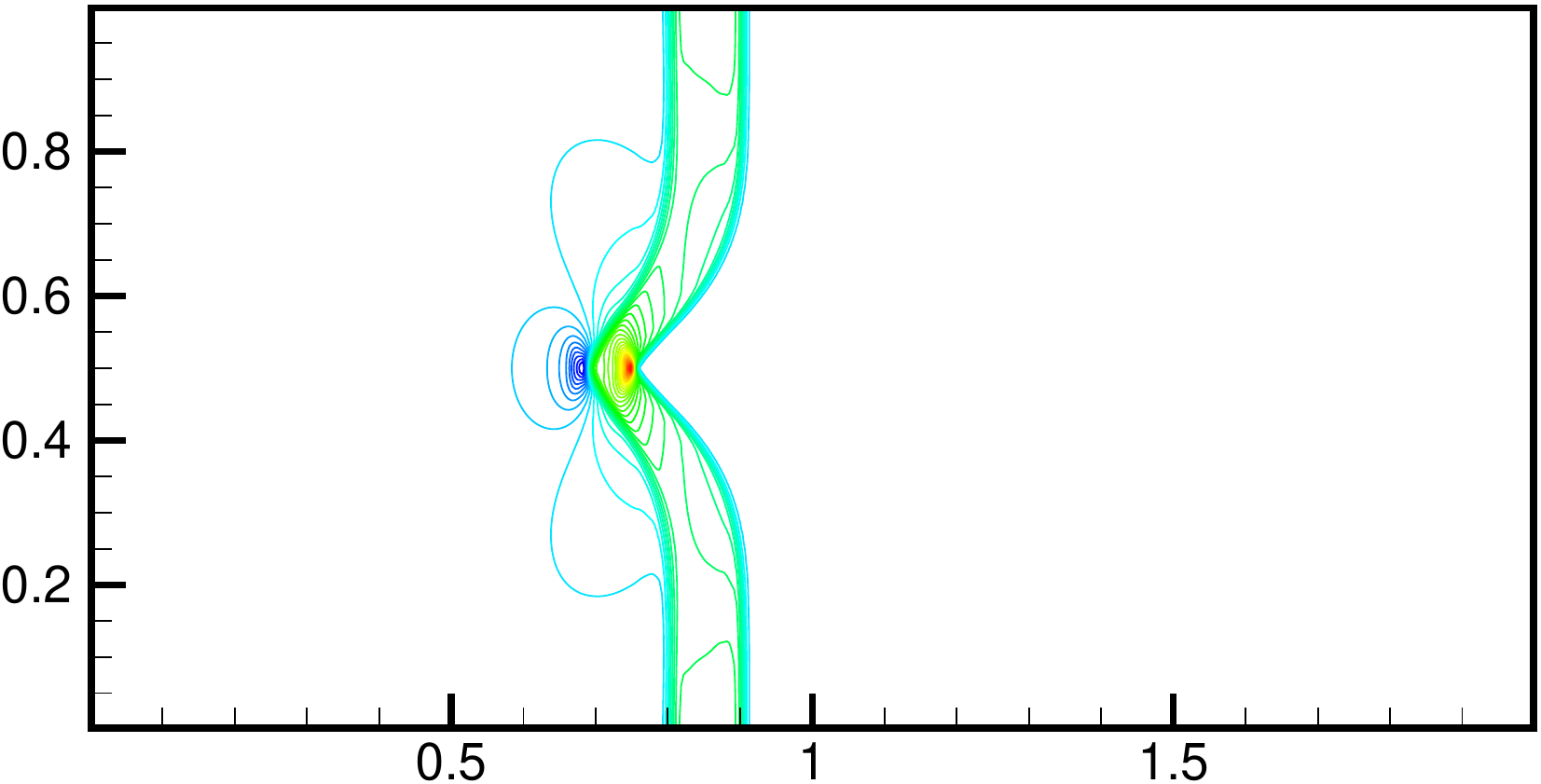}
		\includegraphics[width=0.5\textwidth]{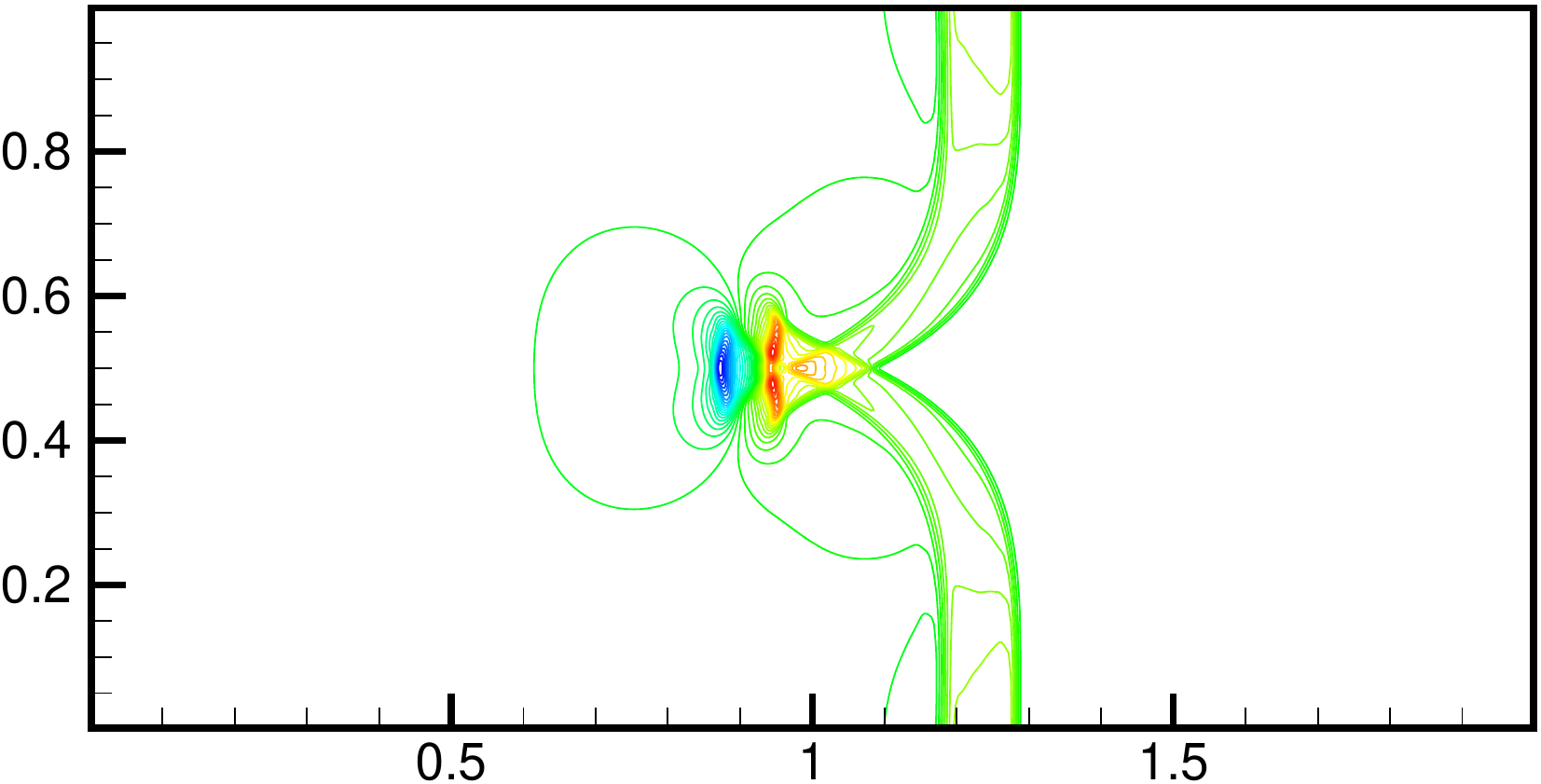}
		\includegraphics[width=0.5\textwidth]{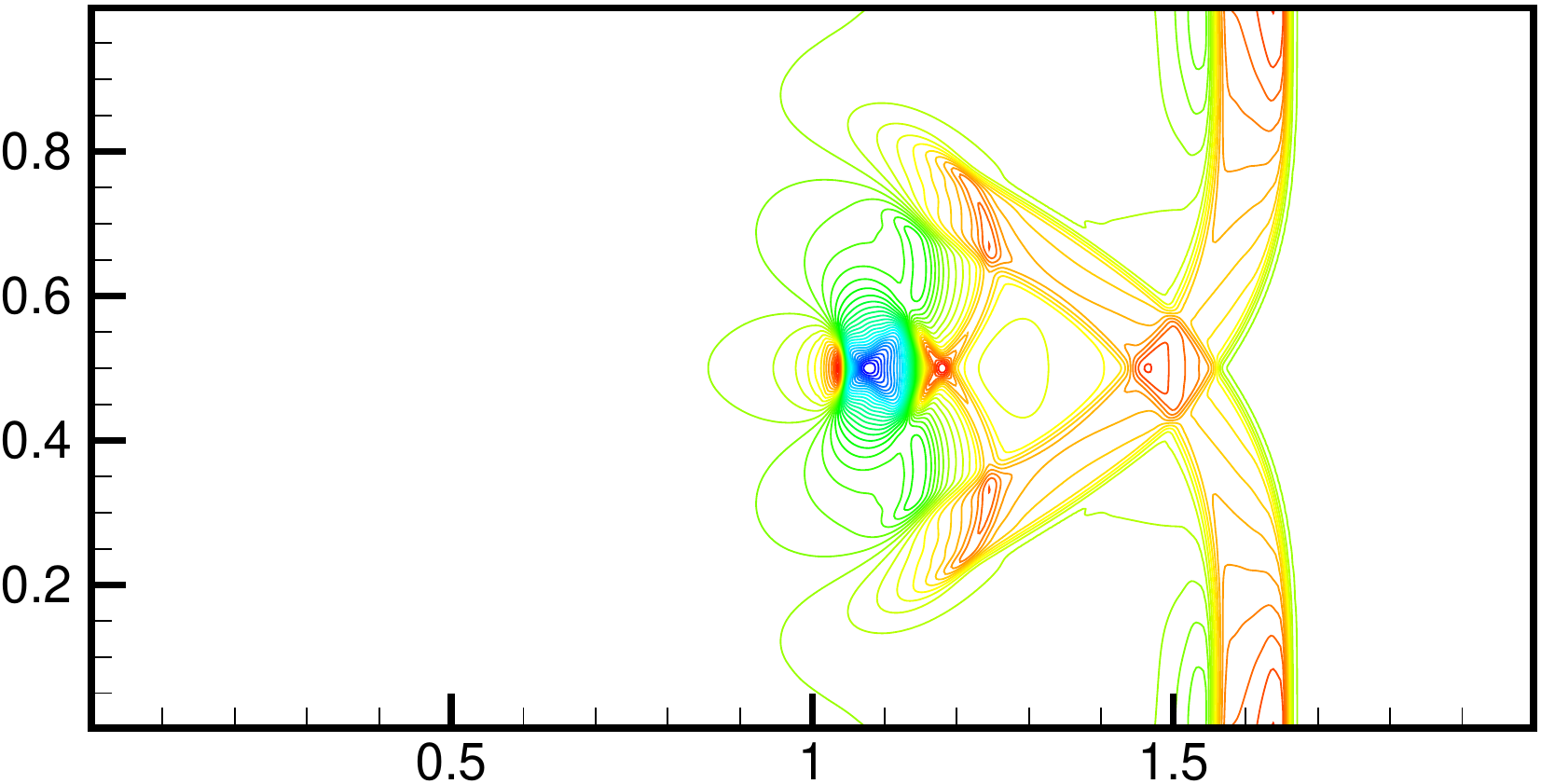}
        \includegraphics[width=0.5\textwidth]{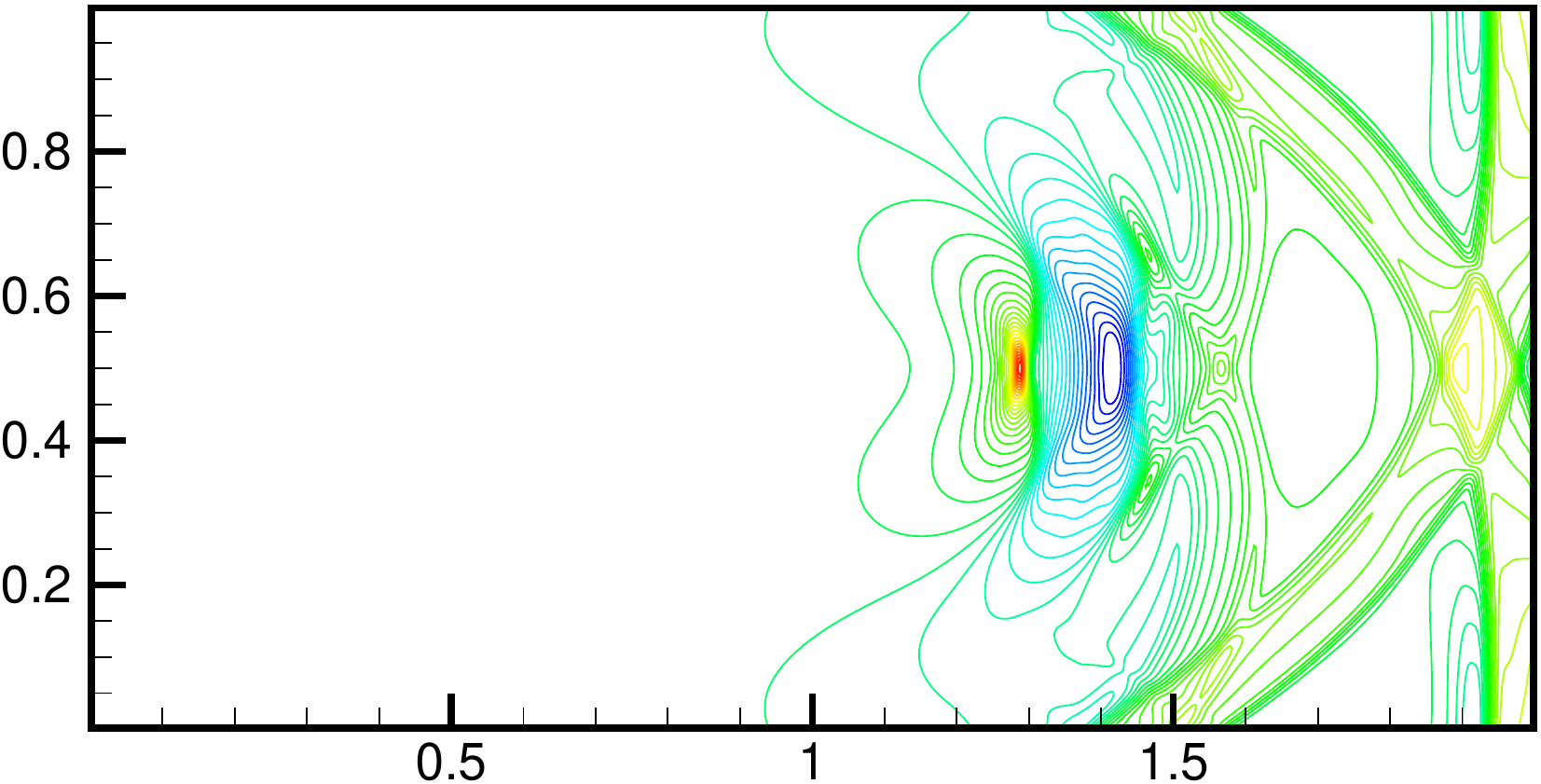}
	\caption{Example \ref{ex:2Dperturb}: The surface level $h+b$ at $t=0.12,0.24,0.36,0.48,0.6$ (From left to right, top to bottom)
		obtained by using the ES scheme with $N_x=600,N_y=300$.
$40$ equally spaced contour lines are used.}
	\label{fig:2DSW_perturb}
\end{figure}

\begin{example}[Orszag-Tang like problem \cite{Zia2014}]\label{ex:OrszagTang}\rm
	It is  similar to the Orszag-Tang problem for the ideal MHD equations \cite{Orszag1979}.
	The computational domain is $[0,2\pi]^2$ with periodic boundary conditions and $g=1$.
	The initial data are
	\begin{align*}
	(h,\vx,\vy,\Bx,\By)=(25/9,-\sin y,\sin x,-\sin y,\sin 2x).
	\end{align*}
\end{example}

The solution of this problem is smooth initially, but the complicated pattern will arise
as the time increases  and it has the turbulence behavior.
Figure \ref{fig:OrszagTang} presents the results obtained by using the ES scheme with $N_x=N_y=200$ at $t=1$ and $2$.
The solutions are in good agreement with those in \cite{Zia2014}.
The left plot in Figure \ref{fig:TotalEntropy} shows the time evolution of the discrete
total entropy $\sum_{i,j}\eta(\bU_{i,j})\Delta x\Delta y$   with three spatial resolutions $N_x=N_y=100$, 200 and 400.
The total entropy is conserved when the solutions are smooth during an initial period, but it decays when discontinuities arise.

\begin{figure}[ht!]
	\begin{subfigure}[b]{0.5\textwidth}
		\centering
		\includegraphics[width=1.0\textwidth]{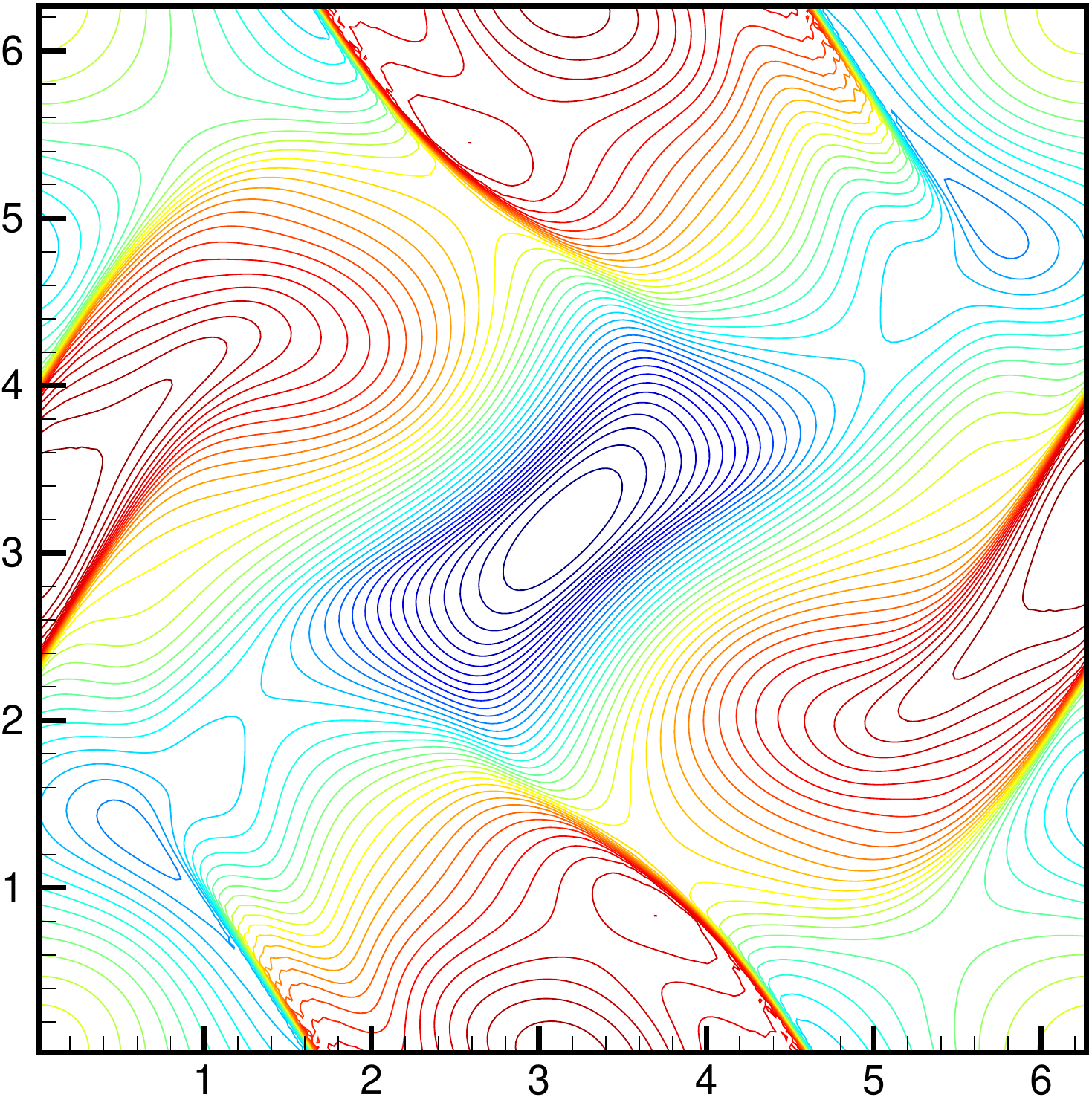}
		\caption{$h$}
	\end{subfigure}
	\begin{subfigure}[b]{0.5\textwidth}
		\centering
		\includegraphics[width=1.0\textwidth]{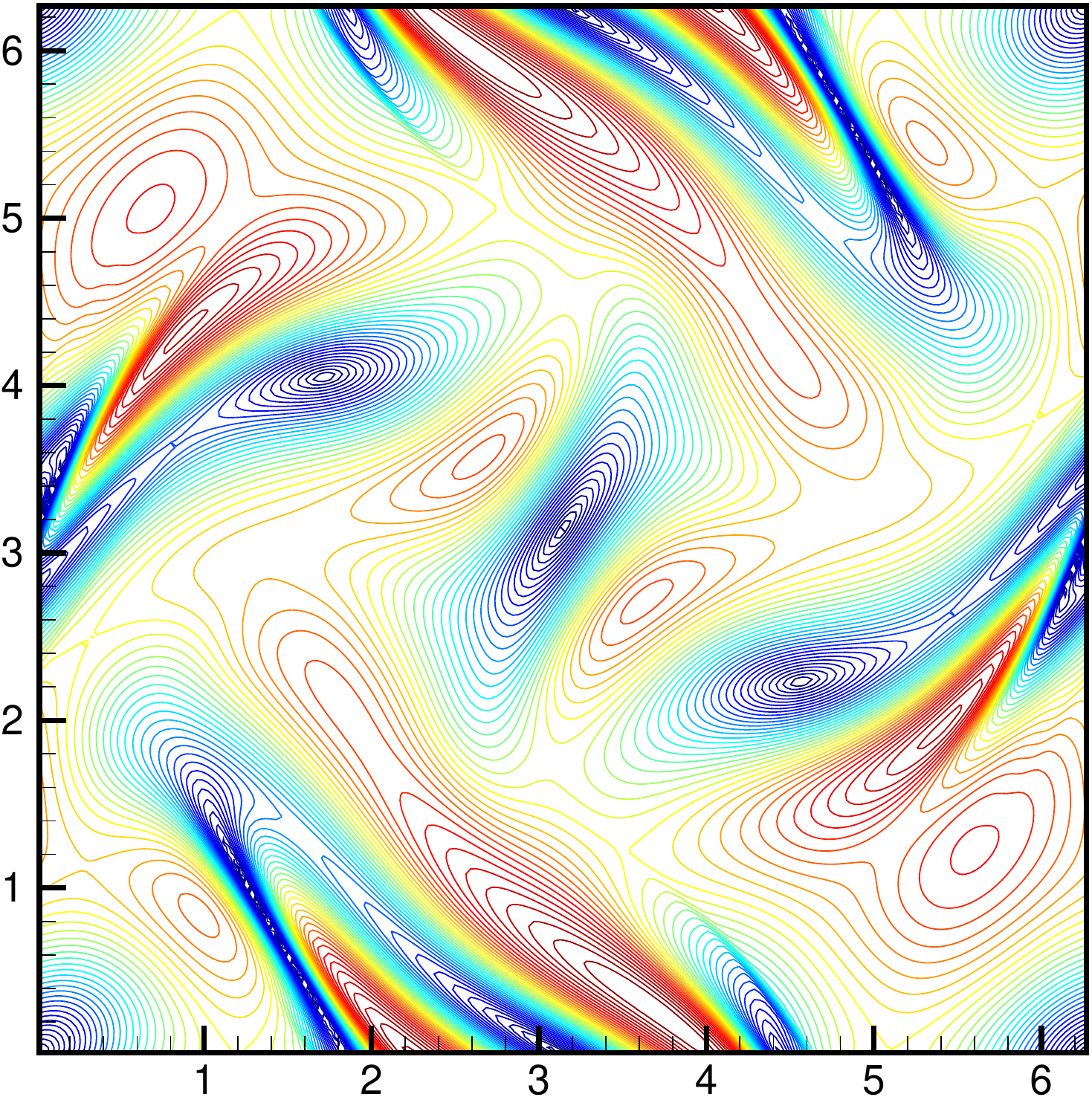}
		\caption{$\abs{\bm{B}}$}
	\end{subfigure}
	\begin{subfigure}[b]{0.5\textwidth}
		\centering
		\includegraphics[width=1.0\textwidth]{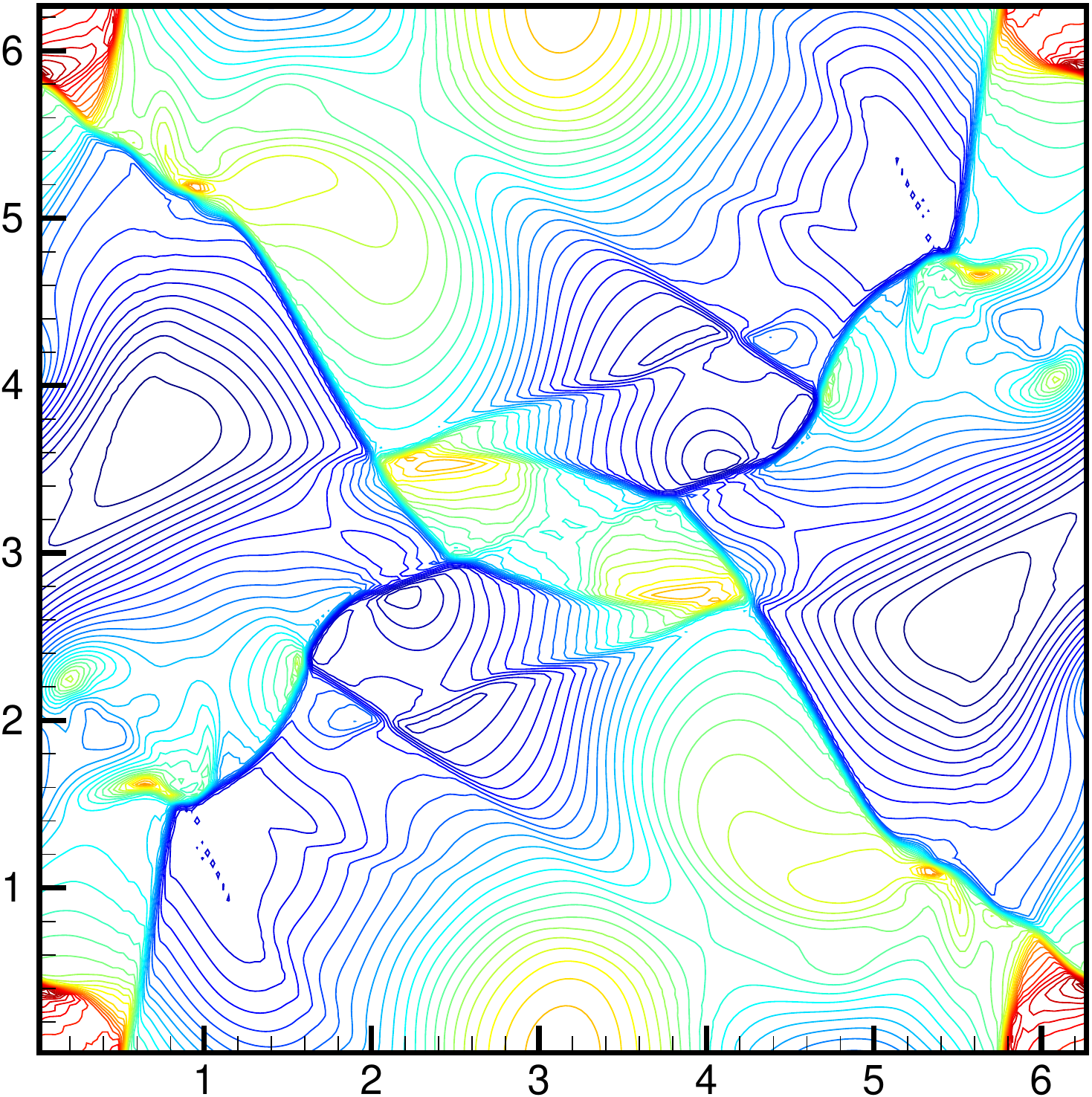}
		\caption{$h$}
	\end{subfigure}
	\begin{subfigure}[b]{0.5\textwidth}
		\centering
		\includegraphics[width=1.0\textwidth]{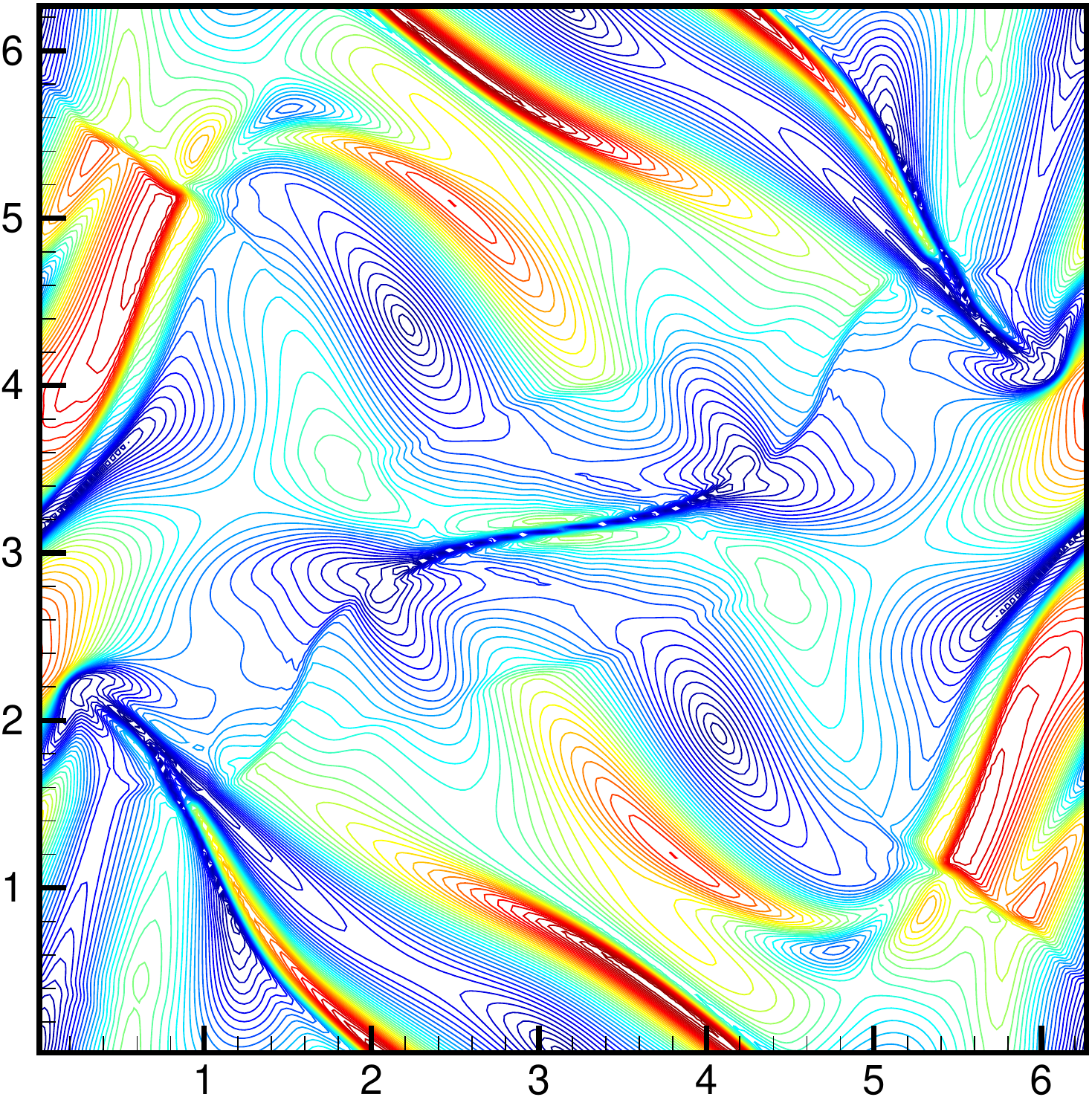}
		\caption{$\abs{\bm{B}}$}
	\end{subfigure}
	\caption{Example \ref{ex:OrszagTang}: The contours  of the height $h$
		 and the magnitude of the magnetic field $\abs{\bm{B}}$
			at $t=1$ (the 1st row) and $2$ (the 2nd row) obtained by using the ES scheme with $N_x=N_y=200$.
 $40$ equally spaced contour lines are used.}
	\label{fig:OrszagTang}
\end{figure}

\begin{example}[Rotor like problem \cite{Kroger2005}]\label{ex:Rotor}\rm
	It is an extension of the classical ideal MHD rotor test problem \cite{Kroger2005}.
	The computational domain is $[-1,1]^2$ with outflow conditions with $g=1$.
	Initially $h\Bx=1$ and $h\By=0$, and there is a disk of radius $r_0=0.1$ centered at $(0, 0)$,
	where fluid with large $h$ is rotating in the anti-clockwise direction.
	The ambient fluid is homogeneous for $r>r_0$, where $r=\sqrt{x^2+y^2}$.
	Specifically, the initial data are
	\begin{align*}
	(h,\vx,\vy)=\begin{cases}
	(10,-y,~x), & r<r_0, \\
	(1,~0,~0), & r>r_0. \\
	\end{cases}
	\end{align*}
	This problem is solved until $t=0.2$.
\end{example}

Figure \ref{fig:Rotor} shows the height $h$, the velocity $\bv$, and
the magnetic field $\bm{B}$ obtained by using the ES scheme with $N_x=N_y=400$.
The ES scheme gets the high resolution results without obvious spurious oscillations
comparable to those in \cite{Kroger2005}.
The right plot of Figure \ref{fig:TotalEntropy} displays the time evolution of the discrete
total entropy with three spatial resolutions $N_x=N_y=100$, 200 and 400.
The results show that the total entropy decays as expected,
and the fully discrete scheme is also ES.

\begin{figure}[htb!]
\begin{subfigure}[b]{1.0\textwidth}
	\centering
	\includegraphics[width=0.4\textwidth]{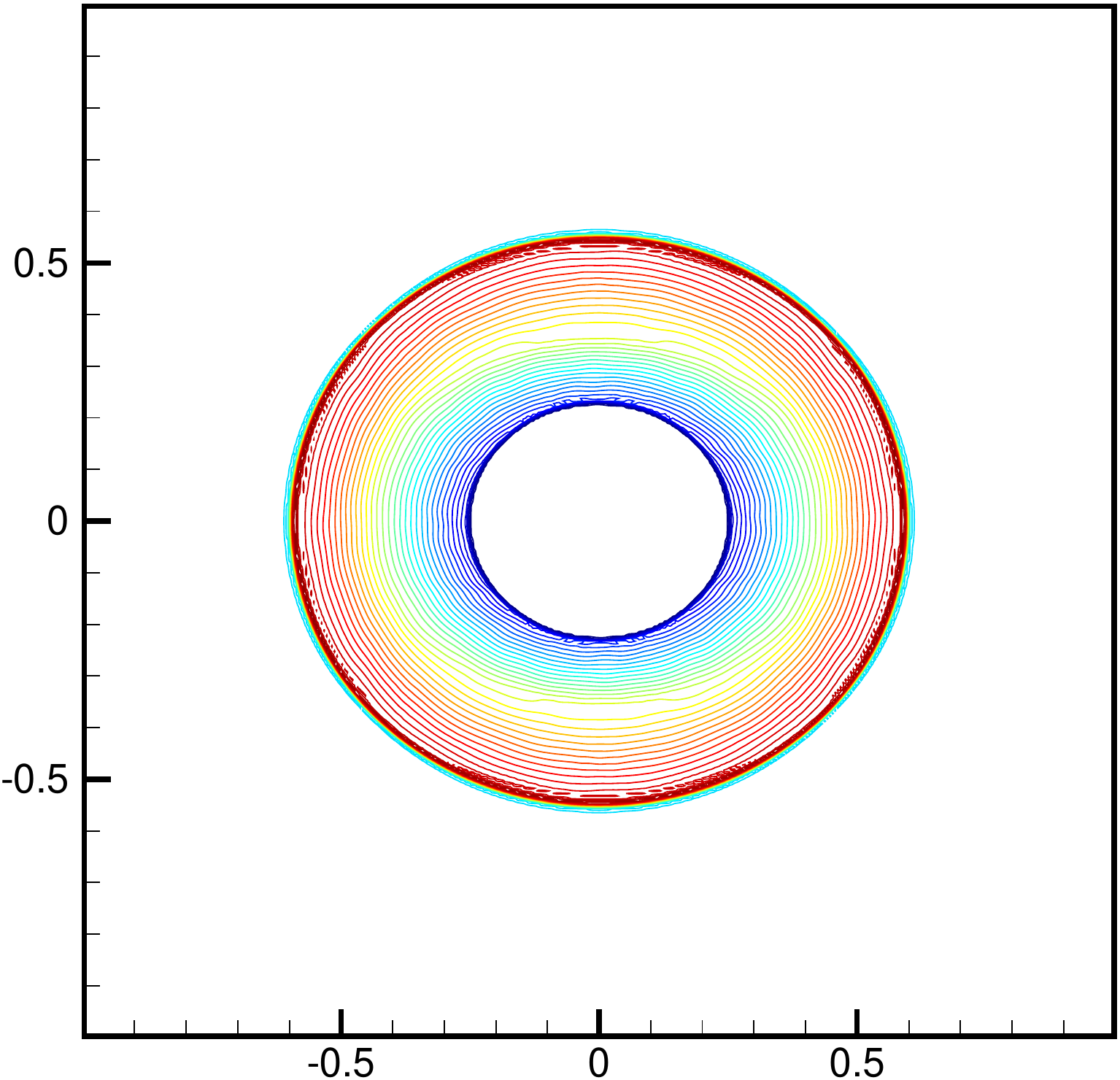}
	\caption{$h$}
\end{subfigure}
\begin{subfigure}[b]{0.5\textwidth}
	\centering
	\includegraphics[width=0.8\textwidth]{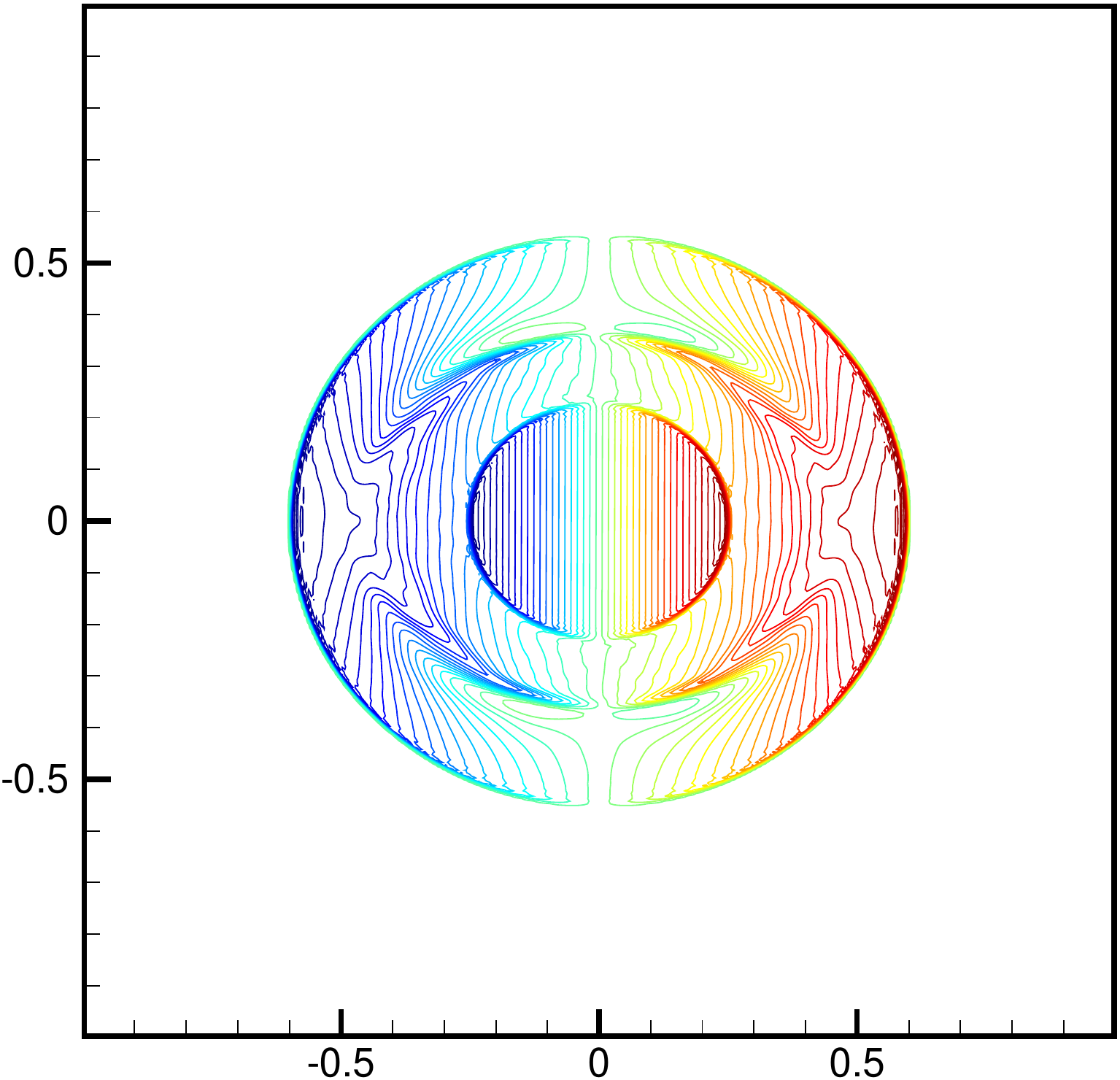}
	\caption{$\vx$}
\end{subfigure}
\begin{subfigure}[b]{0.5\textwidth}
	\centering
	\includegraphics[width=0.8\textwidth]{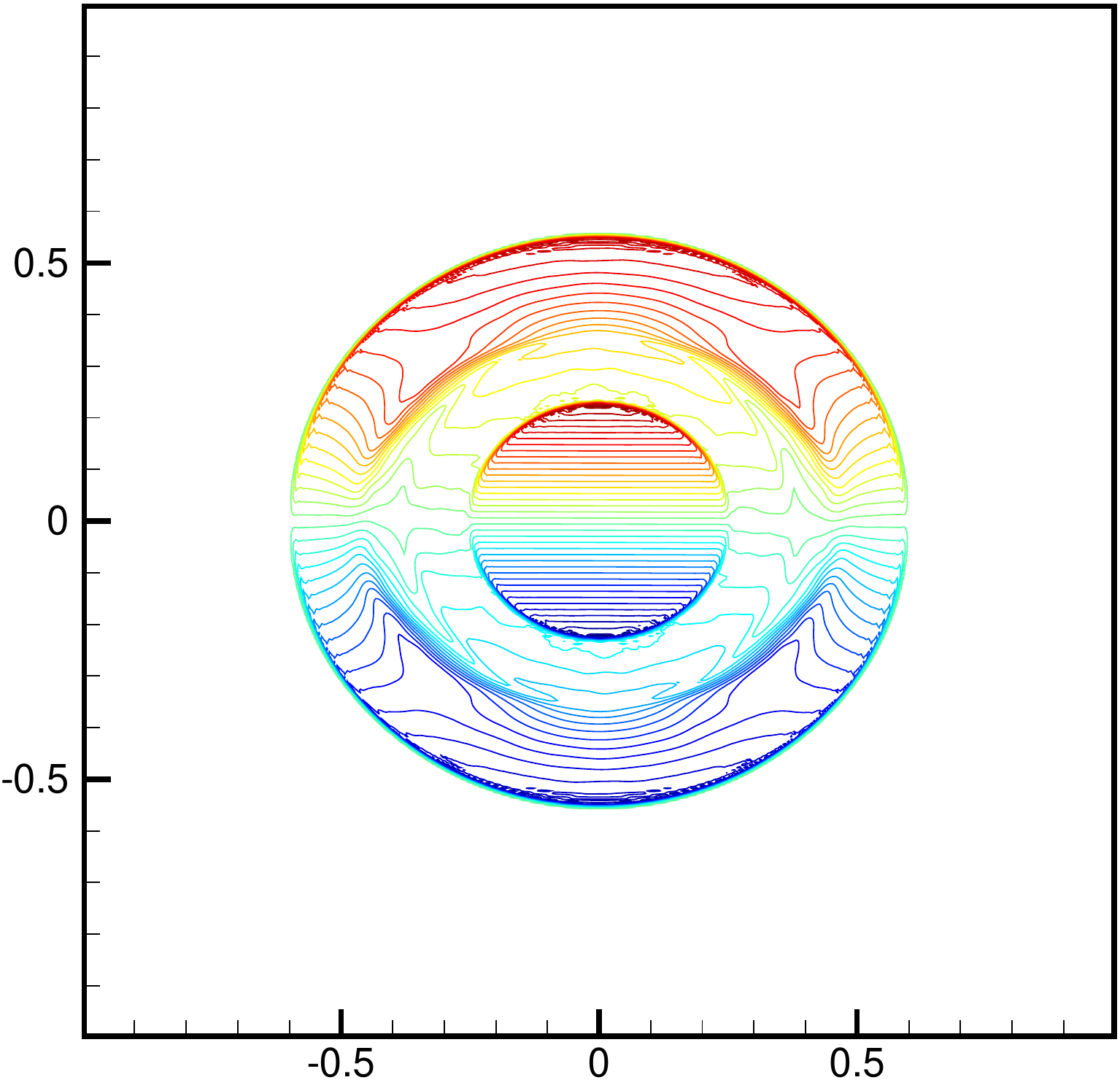}
	\caption{$\vy$}
\end{subfigure}
\begin{subfigure}[b]{0.5\textwidth}
	\centering
	\includegraphics[width=0.8\textwidth]{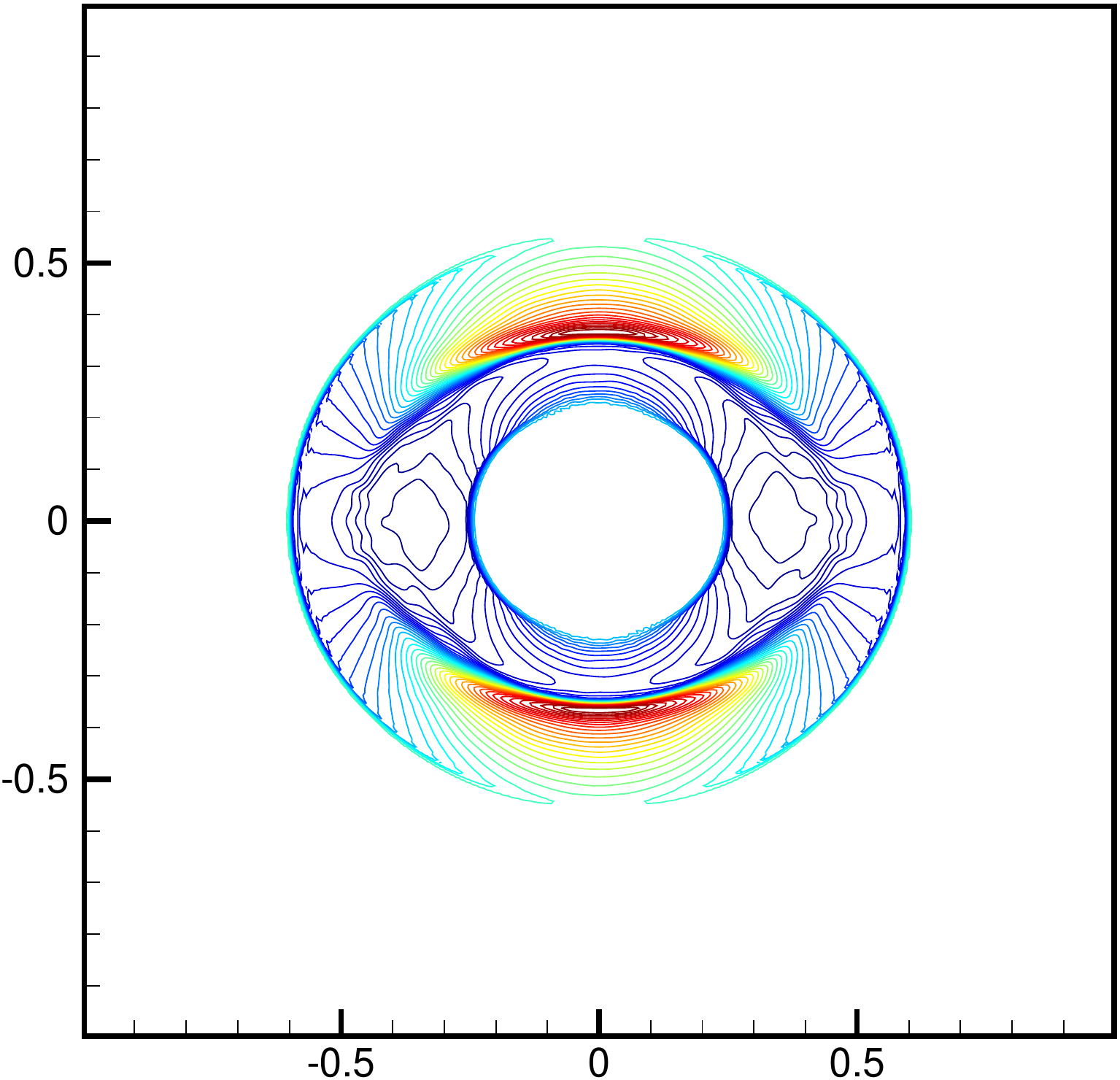}
	\caption{$\Bx$}
\end{subfigure}
\begin{subfigure}[b]{0.5\textwidth}
	\centering
	\includegraphics[width=0.8\textwidth]{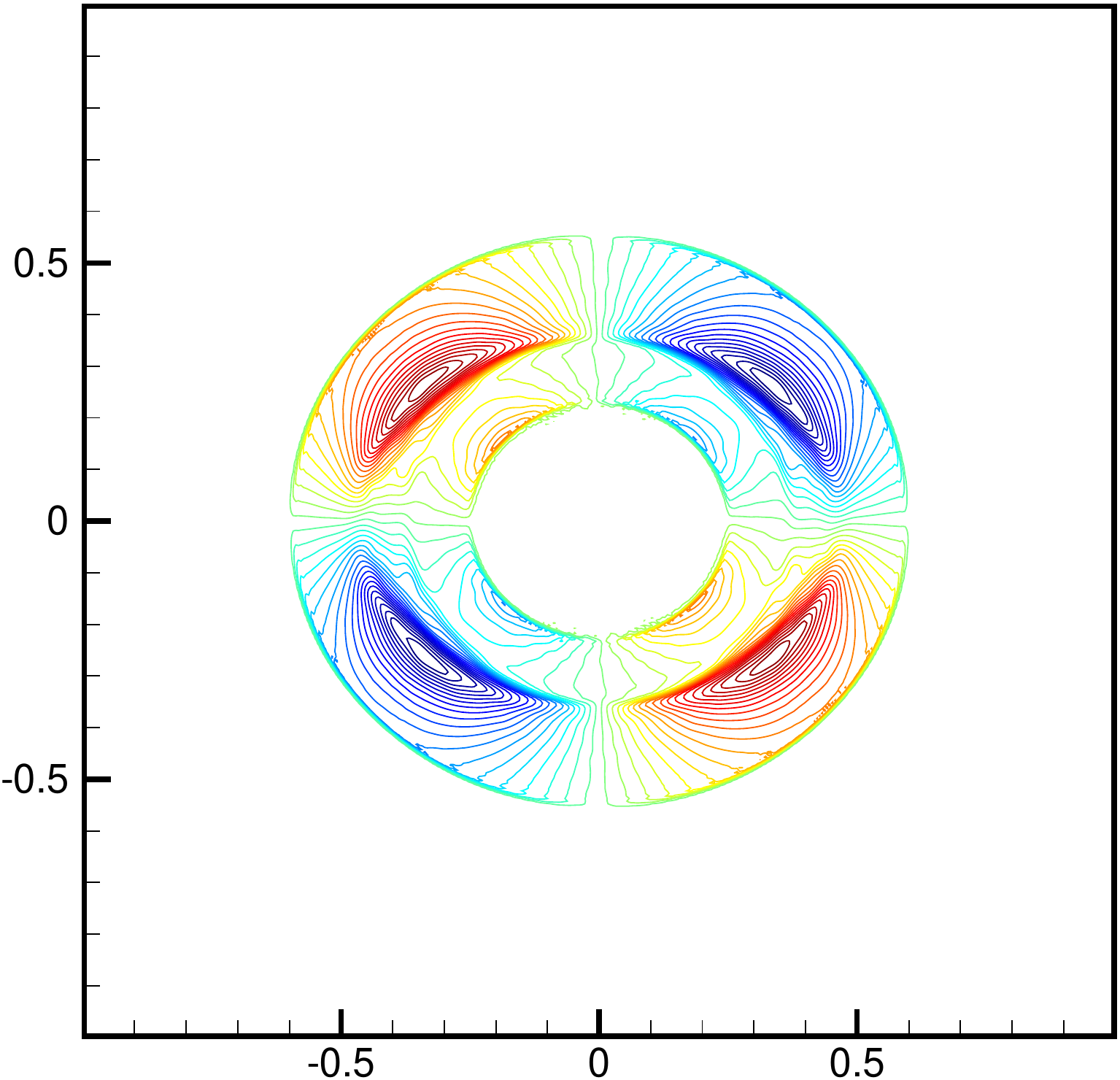}
	\caption{$\By$}
\end{subfigure}
  \caption{Example \ref{ex:Rotor}: The contours of of the $h,\vx,\vy,\Bx,\By$
   at $t=0.2$ obtained by using the ES scheme with $N_x=N_y=400$.
  $40$ equally spaced contour lines are used.}
  \label{fig:Rotor}
\end{figure}

\begin{figure}[htb!]
	\begin{subfigure}[b]{0.5\textwidth}
		\centering
		\includegraphics[width=1.0\textwidth]{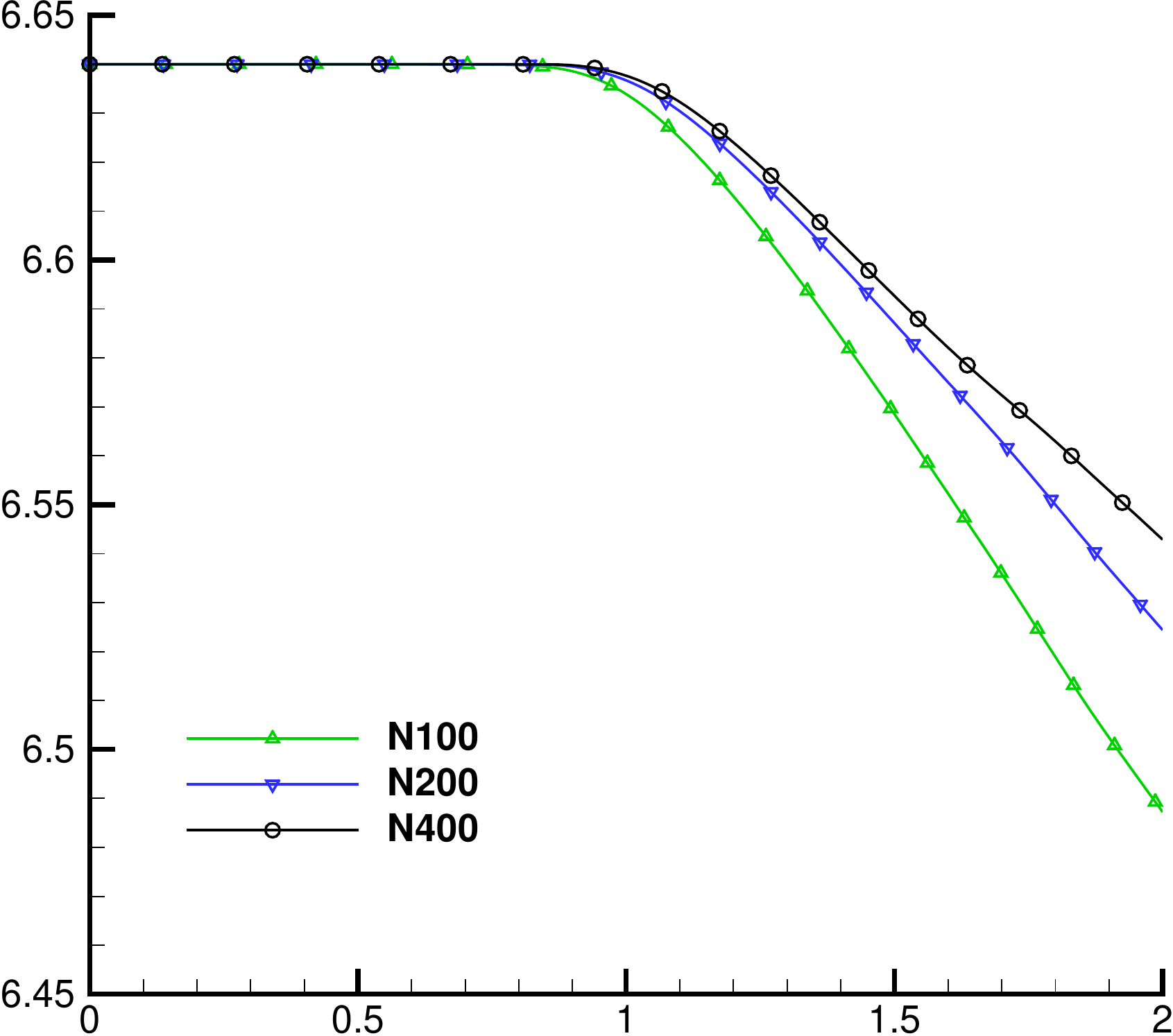}
		\caption{Example \ref{ex:OrszagTang}}
	\end{subfigure}
	\begin{subfigure}[b]{0.5\textwidth}
		\centering
		\includegraphics[width=1.0\textwidth]{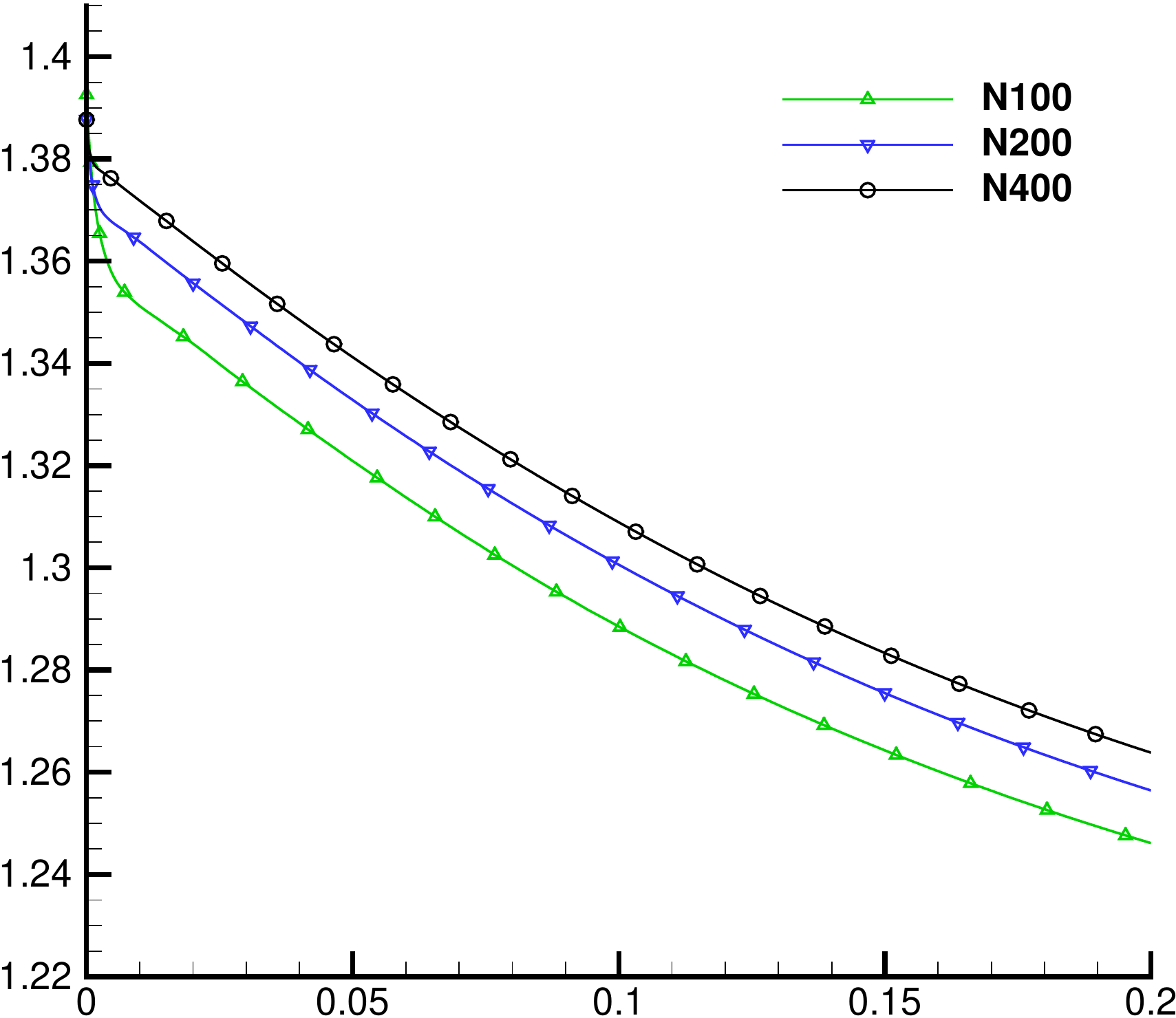}
		\caption{Example \ref{ex:Rotor}}
	\end{subfigure}
	\caption{The time evolution of the discrete total entropy for
	Example \ref{ex:OrszagTang} and \ref{ex:Rotor} with several different resolutions in space. }
	\label{fig:TotalEntropy}
\end{figure}

%

\section{Conclusion}\label{section:Conclusion}
The paper proposed the high-order accurate entropy stable (ES) finite difference schemes for
the one- and two-dimensional shallow water magnetohydrodynamic (SWMHD) equations with non-flat bottom topography.
 The Janhunen source term was added to the conservative SWMHD equations.
For the modified  SWMHD equations, the second-order accurate well-balanced semi-discrete entropy conservative (EC)
finite difference scheme was first constructed, in the sense that it satisfied the semi-discrete entropy identity for the given entropy pair (the total energy served as the mathematical entropy) and preserved the steady state of lake at rest (with zero magnetic field).
The key was to match both discretizations for the fluxes and the source term related to the non-flat river bed bottom and the Janhunen source term, and to find the affordable EC
fluxes of the second-order EC schemes.
Next, the high-order accurate well-balanced EC schemes were obtained
by using the second-order accurate EC schemes as building block.
In view of that the EC schemes might become oscillatory near the discontinuities, the appropriate dissipation terms were added to the EC fluxes to develop the semi-discrete well-balanced  ES
schemes satisfying the semi-discrete entropy inequality.
The WENO reconstruction of the scaled entropy variables and the high-order explicit Runge-Kutta time discretization
were implemented to obtain the fully-discrete high-order  schemes.
The ES and positivity-preserving properties of the Lax-Friedrichs scheme were also proved  without the assumption that
the 1D exact Riemann solution of $x$-split system was ES,
and then the high-order positivity-preserving ES schemes were  developed by using the positivity-preserving flux limiter.
Extensive numerical tests showed that our schemes could achieve
the designed
accuracy, were well-balanced or positivity-preserving, and could well capture the discontinuities.

\appendix
\section{SWMHD in polar coordinates}\label{app}
In the polar coordinate system $(r,\theta)$,
the modified   SWMHD equations \eqref{eq:symm} without the bottom topography  become
\begin{equation}
\begin{aligned}
&\pd{h}{t}+\nabla\cdot(h\bv)=0,\\
&\pd{(hv_r)}{t}+\nabla\cdot(hv_r\bv-hB_r\bm{B})+\pd{(gh^2/2)}{r}=\dfrac{h(v_\theta^2-B_\theta^2)}{r},\\
&\pd{(hv_\theta)}{t}+\nabla^r\cdot(hv_\theta\bv-hB_\theta\bm{B})+\pd{(gh^2/2)}{\theta}=0,\\
&\pd{(hB_r)}{t}+\dfrac1r\pd{(hv_\theta B_r-hv_rB_\theta)}{\theta}=-B_r\divhB,\\
&\pd{(hB_\theta)}{t}-\pd{(hv_\theta B_r-hv_rB_\theta)}{r}=-B_\theta\divhB,
\end{aligned}
\end{equation}
where
\begin{equation}
\begin{aligned}
&\nabla\cdot\bF=\dfrac{1}{r}\pd{(rF_r)}{r}+\dfrac{1}{r}\pd{F_\theta}{\theta},\quad
\nabla^r\cdot\bF=\dfrac{1}{r^2}\pd{(r^2F_r)}{r}+\dfrac{1}{r}\pd{F_\theta}{\theta},
\end{aligned}
\end{equation}
and $F_r,F_\theta$ are the radius and azimuth components of the vector $\bF$, respectively.

\section*{Acknowledgments}
The authors would like to thank the discussion of Dr. Kailiang Wu at The Ohio State University and Mr. Caiyou Yuan at Peking University.
The work was partially supported by the Special Project on High-performance Computing under the
National Key R\&D Program (No. 2016YFB0200603), Science Challenge Project (No. TZ2016002),
the National Natural Science Foundation of China (Nos. 91630310 \& 11421101),
and High-performance Computing Platform of Peking University.

\bibliographystyle{siam}
\bibliography{SW,MHD,SWMHD,ES,WENO,PP}

\end{document}